\documentclass{article}

\usepackage{amsmath,amssymb,amsthm}
\usepackage{bm}
\usepackage{mathrsfs}
\usepackage[all]{xy}
\usepackage{booktabs}
\usepackage{fullpage}
\usepackage{hyperref}
\usepackage{mathtools}
\hypersetup{
colorlinks=false,
linkcolor=red,
citecolor=green
}
\newcommand{\rl}{\mathbb{R}}
\newcommand{\cx}{\mathbb{C}}

\newcommand{\ai}{\sqrt{-1}}

\newcommand{\inj}{\hookrightarrow}
\newcommand{\surj}{\twoheadrightarrow}
\newcommand{\isom}{\stackrel{\sim}{\to}}

\newcommand{\lieg}{\mathfrak{g}}

\newcommand{\cst}{\mathrm{const}}
\newcommand{\kah}{K\"ahler }
\newcommand{\ke}{K\"ahler--Einstein }

\newcommand{\ddbar}{\partial \bar{\partial}}
\newcommand{\tr}{\mathrm{tr}}

\newcommand{\prj}{\mathbb{P}}

\newcommand{\rea}{\mathrm{Re}}

\newcommand{\lich}[1]{\mathcal{D}_{#1}^* \mathcal{D}_{#1}}

\theoremstyle{plain}
\newtheorem{theorem}{Theorem}[section]
\newtheorem{lemma}[theorem]{Lemma}
\newtheorem{proposition}[theorem]{Proposition}
\newtheorem{corollary}[theorem]{Corollary}
\theoremstyle{definition}
\newtheorem{definition}[theorem]{Definition}

\theoremstyle{definition}
\newtheorem{remark}[theorem]{Remark}

\newtheorem{problem}[theorem]{Problem}
\newtheorem{notation}[theorem]{Notation}

\begin{document}

\title{Quantisation of extremal K\"ahler metrics}
\author{Yoshinori Hashimoto}

\maketitle

\begin{abstract}
Suppose that a polarised K\"ahler manifold $(X,L)$ admits an extremal metric $\omega$. We prove that there exists a sequence of K\"ahler metrics $\{ \omega_k \}_k$, converging to $\omega$ as $k \to \infty$, each of which satisfies the equation $\bar{\partial} \text{grad}^{1,0}_{\omega_k} \rho_k (\omega_k)=0$; the $(1,0)$-part of the gradient of the Bergman function is a holomorphic vector field. 

\end{abstract}

\tableofcontents

\section{Introduction} \label{intro}

\subsection{Donaldson's quantisation} \label{intdonquansecqem}

Donaldson's work on the constant scalar curvature \kah (cscK) metrics and the projective embeddings \cite{donproj1, donproj2} is undoubtedly one of the most important results in \kah geometry in the last few decades. It states that, if the automorphism group $\text{Aut} (X,L)$ of a polarised compact \kah manifold $(X , L)$ is discrete (cf.~\S \ref{bgautom}) and $(X,L)$ admits a cscK metric $\omega \in c_1 (L)$, then for all large enough $k$ there exists a balanced metric at the level $k$ (cf.~Definition \ref{defbalmusu2defs}). Our starting point is a naive re-interpretation of the cscK metric as satisfying $\bar{\partial} S(\omega) =0$ and the balanced metric as satisfying $\bar{\partial}  \rho_k (\omega)=0$, where $\rho_k (\omega)$ is the Bergman function (cf.~Definition \ref{defofbergfn}). We also observe that $\text{Aut} (X,L)$ being discrete is equivalent to the connected component $\text{Aut}_0 (X,L)$ containing the identity of $\text{Aut} (X,L)$ being trivial, where we note that $\text{Aut}_0 (X,L)$ will be used more frequently in what follows. We record Donaldson's theorem in this form here.

\begin{theorem} \label{donquan}
\emph{(Donaldson \cite{donproj1})}
Suppose that the connected component of the automorphism group $\textup{Aut}_0 (X,L)$ of a polarised \kah manifold $(X , L)$ is trivial and $(X,L)$ admits a \kah metric $\omega \in c_1 (L)$ satisfying $\bar{\partial} S(\omega) =0$. Then for any large enough $k$ there exists a \kah metric $\omega_k \in c_1 (L)$ satisfying $\bar{\partial}  \rho_k (\omega_k) =0$ and $\omega_k \to \omega$ in $C^{\infty}$ as $k \to \infty$.
\end{theorem}

\begin{theorem} \label{donconv}
\emph{(Donaldson \cite{donproj1})}
If a sequence of \kah metrics $\{ \omega_k \}_k$, each of which satisfies $\bar{\partial}  \rho_k (\omega_k) =0$, converges to a \kah metric $\omega_{\infty} \in c_1 (L)$ in $C^{\infty}$, then the limit $\omega_{\infty}$ satisfies $\bar{\partial} S(\omega_{\infty}) =0$.
\end{theorem}

We note that Theorem \ref{donconv} does not assume the existence of a cscK metric or the triviality of $\text{Aut}_0 (X,L)$, unlike Theorem \ref{donquan}. The importance of Donaldson's theorem, in one direction, is that Theorem \ref{donquan} provides the first general result on the existence of cscK metric implying algebro-geometric ``stability'', along the line conjectured by Yau \cite{yauprob}, Tian \cite{tian97}, and Donaldson \cite{dongauge}, and also extending the previous works of Tian \cite{tian97} on \ke metrics to the cscK metrics. Namely, withe result of Luo \cite{luo} and Zhang \cite{zhang} we can prove that a polarised \kah manifold $(X,L)$ with trivial $\textup{Aut}_0 (X,L)$ which admits a cscK metric is asymptotically Chow stable. 


In another direction, Theorem \ref{donquan} provides an approximation scheme for the cscK metrics. Recall now that the existence of many cscK metrics (e.g.~Calabi--Yau metrics on compact \kah manifolds) is guaranteed only by abstract existence theorems and explicit formulae for these metrics are in general extremely difficult to obtain. However, we can in fact find a numerical algorithm for finding a balanced metric as explained in \cite{donproj2} and \cite{sano06}, and hence it is (in principle) possible to numerically approximate a cscK metric. Various mathematicians have used this method to attack this problem of ``explicitly'' approximating a cscK metric, and there already seems to be a substantial accumulation of research. We only mention here \cite{bbdo1, bbdo2, donnum, dkrl, kellerlukic}, which actually implemented the above algorithm.

That such a numerical algorithm should exist could be seen intuitively from the following fact. Suppose that we choose a basis $\{ Z_i \}$ for $H^0 (X , L^k)$ (for large enough $k$) so as to have an isomorphism $H^0 (X , L^k) \isom \cx^{N_k}$ and an embedding $\iota : X \inj \prj (H^0 (X , L^k)^*) \cong \prj^{N_k-1}$. We then consider the moment map for the $U(N_k)$-action on $\prj^{N_k-1}$, and integrate it over the image $\iota (X)$ of $X$ to get the centre of mass $\bar{\mu}_X$ (see \S \ref{quantisationsch} for the details); namely $\bar{\mu}_X$ is defined as
\begin{equation*}
\bar{\mu}_X : = \int_{\iota (X)} \frac{h_{FS} (Z_i , Z_j)}{\sum_{l=1}^{N_k} |Z_l|^2_{FS}} \frac{\omega^n_{FS}}{n!} \in \ai \mathfrak{u} (N_k)
\end{equation*}
where $h_{FS}$ is the Fubini--Study metric on $\mathcal{O}_{\prj^{N_k-1}} (1)$. We can move the image $\iota (X)$ by an $SL(N_k , \cx)$-action on $\prj^{N_k-1}$, and we write $\bar{\mu}_X (g)$ for the new centre of mass when we move $\iota (X)$ by $g \in SL(N_k , \cx)$ to $\iota_g (X)$, say. It is well-known (cf.~\cite{luo, zhang}, see also Theorem \ref{balembbalmetequiv}) that there exists a balanced metric at the level $k$ if and only if there exists $g \in SL(N_k , \cx)$ such that $\bar{\mu}_X (g)$ is equal to a constant multiple of the identity. Thus, the seemingly intractable PDE problem $\bar{\partial} \rho_k (\omega)=0$ can in fact be reduced to a \textit{finite dimensional} problem on the vector space $H^0 (X , L^k)$.


\begin{remark} \label{excsckchunst}
We now recall that the hypothesis of $\textup{Aut}_0 (X,L)$ being trivial is essential in Theorem \ref{donquan}. Indeed, Della Vedova and Zuddas showed \cite[Example 4.3]{dvz} that $\prj^2$ blown up at 4 points, all but one aligned, is Chow \textit{unstable} at the level $k$ for \textit{all} large enough $k$ with respect to an appropriate polarisation, although a well-known theorem of Arezzo and Pacard \cite{ap2} (see in particular \cite[Example 7.3]{ap2}) shows that it admits a cscK metric in that polarisation. We also recall that Ono, Sano, and Yotsutani \cite{osy} showed that there exists a toric \ke Fano manifold that are asymptotically Chow unstable (with respect to the anticanonical polarisation, even after replacing $K^{-1}_X$ by a higher tensor power).


\end{remark}

\subsection{Statement of the results} \label{intstmtofresults}

Our aim is to find how Theorems \ref{donquan} and \ref{donconv} can extend to the case where $\text{Aut}_0 (X , L)$ is no longer trivial. Since Theorem \ref{donquan} does fail to hold when $\text{Aut}_0 (X , L)$ is nontrivial (cf.~Remark \ref{excsckchunst}), we need a new ingredient. Suppose now that we replace $\bar{\partial}$ by an operator $\bar{\partial} \textup{grad}^{1,0}_{\omega}$ (cf.~\S \ref{blichop}) and consider the equation $\bar{\partial} \textup{grad}^{1,0}_{\omega} S(\omega) =0$, i.e.~$\omega$ is an \textit{extremal} metric, which can be regarded as a ``generalisation'' of cscK metrics when $\text{Aut}_0 (X , L)$ is no longer trivial (cf.~\S \ref{blichop}).


Now, when we change $\bar{\partial} S(\omega)=0$ to $\bar{\partial} \textup{grad}^{1,0}_{\omega} S (\omega)=0$, the corresponding equation $\bar{\partial} \rho_k (\omega_k) =0$ changes to
\begin{equation} \label{introrbaldef}
\bar{\partial} \textup{grad}^{1,0}_{\omega_k} \rho_k (\omega_k) =0 ,
\end{equation}
and this seems to suggest that this is the equation which ``quantises'' the extremal metric, when $\text{Aut}_0 (X,L)$ is no longer trivial; observe that when $\text{Aut}_0 (X,L)$ is trivial and hence $(X,L)$ admits no nontrivial holomorphic vector field, the above equation implies $\rho_k (\omega_k) = \cst$ and hence we recover the balanced metric.

The aim of this paper is to establish an ``extremal'' analogue of Theorems \ref{donquan} and \ref{donconv} by using the equation (\ref{introrbaldef}). First of all, an analogue of Theorem \ref{donconv} can be established as follows.

\begin{theorem} \label{sbalmconv}
If a sequence of \kah metrics $\{ \omega_k \}_k$ in $c_1 (L)$, each of which satisfies $\bar{\partial} \textup{grad}^{1,0}_{\omega_k}  \rho_k (\omega_k) =0$, converges to a \kah metric $\omega_{\infty} \in c_1 (L)$ in $C^{\infty}$, then the limit $\omega_{\infty}$ satisfies $\bar{\partial} \textup{grad}^{1,0}_{\omega_{\infty}} S(\omega_{\infty}) =0$, i.e.~is an extremal metric.
\end{theorem}

\begin{proof}
By recalling the well-known expansion of the Bergman function (Theorem \ref{bergexp}), we have 
\begin{equation*}
	0=\bar{\partial} \textup{grad}^{1,0}_{\omega_k} 4 \pi k  \rho_k (\omega_k) = \bar{\partial} \textup{grad}^{1,0}_{\omega_k} \left( S(\omega_k) + O(1/k) \right) .
\end{equation*}
Since $\{ \omega_k \}_k$ converges to $\omega_{\infty}$ in $C^{\infty}$ as $k \to \infty$, we have 
\begin{equation*}
0=\bar{\partial} \textup{grad}^{1,0}_{\omega_k} 4 \pi k  \rho_k (\omega_k) =\bar{\partial} \textup{grad}^{1,0}_{\omega_k} \left( S (\omega_{k}) - S(\omega_{\infty}) \right) +  \bar{\partial} \textup{grad}^{1,0}_{\omega_k} S(\omega_{\infty} ) + O(1/k) ,
\end{equation*}
and hence we get $\bar{\partial} \textup{grad}^{1,0}_{\omega_{\infty}} S(\omega_{\infty}) = \lim_{k \to \infty} \bar{\partial} \textup{grad}^{1,0}_{\omega_k} S(\omega_{\infty} ) =0$.

\end{proof}

An important aspect of the equation (\ref{introrbaldef}) is that, similarly to the case when $\text{Aut}_0 (X,L)$ is trivial, we can find an equivalent characterisation in terms of the centre of mass $\bar{\mu}_X$, so that solving the equation (\ref{introrbaldef}) can be reduced to an essentially finite dimensional problem (cf.~\S \ref{redtalalgprobsec}); we shall see (Proposition \ref{rbalmiffinvocomtgtiox} and Corollary \ref{crdlaprbal}) that the equation (\ref{introrbaldef}) holds if and only if there exists $g \in SL(N_k , \cx)$ such that $\bar{\mu}_X (g)^{-1}$ generates a holomorphic vector field on $\prj( H^0 (X,L^k)^*) \cong \prj^{N_k -1}$ that is tangential to the image $\iota(X)$ of $X$; see Lemma \ref{lembergfnitoham} for the perhaps surprising appearance of the inverse sign in $\bar{\mu}_X (g)^{-1}$.



Let $K := \textup{Isom} (\omega) \cap \textup{Aut}_0(X , L)$, where $\textup{Isom} (\omega)$ is the isometry group of the extremal metric $\omega$ (cf. \S \ref{secautgpoekm}). We now state our main result as follows; it is an analogue of Theorem \ref{donquan} when $\textup{Aut}_0 (X,L)$ is nontrivial.


\begin{theorem} \label{sbalmqext}
Suppose that $(X,L)$ admits an extremal metric $\omega \in c_1 (L)$. Replacing $L$ by $L^r$ for a large but fixed $r \in \mathbb{N}$ if necessary, for each $l \in \mathbb{N}$, there exists $k_l \in \mathbb{N}$ such that for all $k \ge k_l$ there exists a smooth $K$-invariant \kah metric $\omega_{k,l} \in c_1 (L)$ which satisfies $\bar{\partial} \textup{grad}^{1,0}_{\omega_{k,l}}  \rho_k (\omega_{k,l}) =0$ and converges to $\omega$ in $C^l$ as $k \to \infty$.
\end{theorem}

The reader is referred to Remark \ref{convclcinftydiagarg} for comments on the dependence on $l$, and the possibility of the convergence $\omega_{k,l} \to \omega$ in $C^{\infty}$. 

Finally, recalling the characterisation of the equation (\ref{introrbaldef}) in terms of the centre of mass (Proposition \ref{rbalmiffinvocomtgtiox}), we hope that Theorem \ref{sbalmqext} may potentially provide a numerical approximation to the extremal metrics, as in the cscK case.

\subsection{Comparison to previously known results}

We recall that, in fact, the problem of ``quantising'' the extremal metrics has been considered by several mathematicians, notably by Mabuchi \cite{mab04ext, mab05, mab09, mab11} and Sano--Tipler \cite{santip}. The work of Apostolov--Huang \cite{ah} is also related, and contains a neat survey of Mabuchi's work. We provide a brief survey on these results, while more recent results \cite{mab2016,santip17,seyrel} will be reviewed in \S \ref{reltrecrelwks}.

%

An important special case of Theorem \ref{sbalmqext} is when $\textup{Aut}_0 (X , L)$ is nontrivial but the centre $Z(K)$ of $K$ is discrete. As is well-known, if $\omega$ is extremal, the Hamiltonian vector field $v_s$ generated by $S(\omega)$ has to belong to the centre $\mathfrak{z} := \mathrm{Lie} (Z(K))$ of the Lie algebra $\mathfrak{k} : = \textup{Lie} (K)$ (cf.~Lemma \ref{lemrbgfgenvfcentre}). Thus, $Z(K)$ being discrete implies $v_s=0$, and hence $\omega$ is cscK. On the other hand, if $Z(K)$ is discrete and a $K$-invariant \kah metric $\omega_k$ satisfies $\bar{\partial} \textup{grad}^{1,0}_{\omega_k}  \rho_k (\omega_k) =0$, then Lemmas \ref{lemhol} and \ref{lemrbgfgenvfcentre} show that the Hamiltonian vector field $v$ generated by $\rho_k (\omega_k)$ has to lie in $\mathfrak{z}$; thus $Z(K)$ being discrete and Theorem \ref{sbalmqext} implies that $\rho_k (\omega_k)$ has to be constant, i.e.~$\omega_k$ is a balanced metric for all large (and divisible) $k$, and hence by a theorem of Zhang \cite{zhang}, $(X,L)$ is asymptotically Chow semistable.

This is in fact an easy consequence of the results proved by Futaki \cite{fut04} and Mabuchi \cite{mab04ob, mab04ext, mab05, mab11}, which we now recall. If $(X ,L)$ is cscK, Mabuchi \cite{mab04ob} proved that there exists an obstruction for $(X,L)$ being asymptotically Chow polystable when $\textup{Aut}_0 (X , L)$ is nontrivial, and also showed that the vanishing of these obstructions is sufficient for a cscK $(X,L)$ to be asymptotically Chow polystable \cite{mab05}. Futaki \cite{fut04} proved that the vanishing of Mabuchi's obstructions is equivalent to the vanishing of a series of integral invariants, which may be called ``higher Futaki invariants''. We can show that they all vanish when $(X,L)$ is cscK and $Z(K)$ is discrete as follows; since the higher Futaki invariants are Lie algebra \textit{characters} defined on $\mathrm{LieAut}_0 (X,L) = \mathfrak{k} \oplus \ai \mathfrak{k}$ (by Matsushima--Lichnerowicz theorem \cite{lichnerowicz, matsushima}), the centre of $\mathfrak{k}$ being trivial implies that these higher Futaki invariants are all equal to 0, and hence that $(X,L)$ is indeed asymptotically Chow polystable, which in particular implies that $(X,L)$ is asymptotically Chow semistable. 

We saw in Remark \ref{excsckchunst} the example of cscK, or even K\"ahler--Einstein, manifolds that are asymptotically Chow unstable even after replacing $L$ by a large enough tensor power. However, Theorem \ref{sbalmqext} implies that it is still possible to find a \kah metric $\omega_k$ with $\bar{\partial} \textup{grad}^{1,0}_{\omega_k}  \rho_k (\omega_k) =0$ on these manifolds, and an algebro-geometric result to be presented in \cite{yhstability} further implies that they are asymptotically Chow stable \textit{relative to the centre of $K$} \cite[Corollary 1.2]{yhstability}.

Finally, we recall the theorem of Stoppa and Sz\'ekelyhidi \cite{stosze}, which states that the existence of extremal metrics implies the $K$-polystability relative to a maximal torus in the automorphism group, where the notion of relative $K$-stability was introduced by Sz\'ekelyhidi \cite{szeext}.

\begin{remark} \label{remconjfunitfrbmet}
Recalling \cite[Corollary 5]{donproj1}, it is natural to expect that Theorem \ref{sbalmqext} implies the uniqueness of extremal metrics in $c_1 (L)$ up to $\mathrm{Aut}_0 (X,L)$-action. Indeed, we set up the problem of finding the solution to (\ref{introrbaldef}) as a variational problem of finding the critical point of the modified balancing energy $\mathcal{Z}^A$ on a finite dimensional manifold $\mathcal{B}_k^K$, where $A$ is essentially equal to $\textup{grad}_{\omega_k} \rho_k (\omega_k)$; see \S \ref{redtalalgprobsec} and \S \ref{qexmsgf} for more details. It is clear from the convexity (cf.~Remark \ref{remmbalhesbhess}, Theorem \ref{lemhessbalenfinod}) of  $\mathcal{Z}^A$ that the critical point of  $\mathcal{Z}^A$ is unique up to $\mathrm{Aut}_0 (X,L)$-action for each \textit{fixed} $A$. However, the problem is that we do not know whether there exist two metrics $\omega_1$ and $\omega_2$ in $c_1(L)$, both satisfying $\bar{\partial} \textup{grad}^{1,0}_{\omega_1} \rho_k (\omega_1) =0$ and $\bar{\partial} \textup{grad}^{1,0}_{\omega_2} \rho_k (\omega_2) =0$, but with $\textup{grad}_{\omega_1} \rho_k (\omega_1) \neq \textup{grad}_{\omega_2} \rho_k (\omega_2)$. The existence of such $\omega_1$ and $\omega_2$ would imply that we cannot prove the uniqueness of the ``quantised'' approximant (as in \cite[Theorem 1]{donproj1}) of extremal metrics, and hence the uniqueness of extremal metric itself. On the other hand, the uniqueness of extremal metrics itself was established by Mabuchi \cite{mab04uni}, Berman and Berndtsson \cite{bb}.

\end{remark}

\subsection{Organisation of the paper} \label{orgoutline}
The strategy of the proof of Theorem \ref{sbalmqext}, which occupies most of what follows, is essentially the same as in \cite{donproj1}; we construct an approximate solution to $\bar{\partial} \mathrm{grad}^{1,0}_{\omega_h} \rho_k (\omega_h) = 0$, reduce the problem to a finite dimensional one, and use the gradient flow on a finite dimensional manifold to perturb the approximate solution to the genuine one.

After reviewing in \S \ref{background} some well-known results on the automorphism group of polarised \kah manifolds and Donaldson's theory of quantisation, we construct approximate solutions in \S \ref{approximately}; after some preliminary work in \S \ref{premapprox}, we establish the main technical result Proposition \ref{approxsbal} and its consequence Corollary \ref{coraprbalmsolprob}. We establish in \S \ref{redtalalgprobsec} the characterisation of the equation $\bar{\partial} \mathrm{grad}^{1,0}_{\omega_h} \rho_k (\omega_h) = 0$ in terms of the centre of mass $\bar{\mu}_X$, so as to reduce the problem to a finite dimensional one; the main results of the section are Proposition \ref{rbalmiffinvocomtgtiox}, Corollaries \ref{corbdonaapprrbal} and \ref{crdlaprbal}. We set up the problem as a variational one in \S \ref{secmbaleza} by introducing the ``modified'' balancing energy $\mathcal{Z}^A$, so that the solution of $\bar{\partial} \mathrm{grad}^{1,0}_{\omega_h} \rho_k (\omega_h) = 0$ can be obtained by finding the critical point of $\mathcal{Z}^A$. By recalling the well-known estimates on the Hessian of the balancing energy in \S \ref{sechessbalen}, we run the gradient flow (\ref{grflbmblfokivs}) in \S \ref{secgrflza} driven by $\mathcal{Z}^A$. Unfortunately, the nontrivial automorphism group $\mathrm{Aut}_0 (X,L)$ means that the limit of the gradient flow does not achieve the critical point of $\mathcal{Z}^A$ (cf.~Proposition \ref{proplimgfh1}). However, in \S \ref{seciterconsorbalm} we set up an inductive procedure to (exponentially) decrease $\delta \mathcal{Z}^A$, which is shown to converge, so as to give the critical point of $\mathcal{Z}^A$ (Proposition \ref{propiterschrbalm}); the trick is in fact to perturb the auxiliary parameter $A$ to decrease $\delta \mathcal{Z}^A$.

Finally, we discuss in \S \ref{reltrecrelwks} similarities and differences to the related results \cite{mab2016, santip17, seyrel} that appeared since the appearance of the first preprint version of this paper.


\begin{notation} \label{notdimh0nvol}
In this paper, we shall consistently write $N = N_k$ for $\dim_{\cx} H^0 (X,L^k)$, and $V$ for $\int_X c_1 (L)^n / n!$.
\end{notation}

\subsection*{Acknowledgements}
Part of this work was carried out in the framework of the Labex Archim\`ede (ANR-11-LABX-0033) and of the A*MIDEX project (ANR-11-IDEX-0001-02), funded by the ``Investissements d'Avenir" French Government programme managed by the French National Research Agency (ANR). Most of this work was carried out when the author was a PhD student at the Department of Mathematics of the University College London, which he thanks for the financial support.

This work forms part of the author's PhD thesis submitted to the University College London. He is grateful to his supervisor Jason Lotay and the thesis examiners Joel Fine and Julius Ross for many helpful comments. He acknowledges with pleasure that Lemma \ref{lembergfnitoham} and Proposition \ref{rbalmiffinvocomtgtiox} were pointed out to him by Joel Fine.

\section{Background} \label{background}

\subsection{Automorphism groups of polarised \kah manifolds and extremal metrics} \label{bgautom}

\subsubsection{Holomorphic vector fields, the Lichnerowicz operator, and extremal metrics} \label{blichop}

We first briefly review various facts on holomorphic vector fields of $(X,L)$; the reader is referred to \cite{fujiki, kobayashi, lebsim} for more details on what is discussed here.

We write $\mathrm{Aut} (X)$ for the group of holomorphic transformations of $X$, consisting of diffeomorphisms of $X$ which preserve the complex structure $J$, and $\text{Aut}_0 (X)$ for the connected component of $\mathrm{Aut} (X)$ containing the identity. 
\begin{definition}
A vector field $v$ on $X$ is called \textbf{real holomorphic} if it preserves the complex structure, i.e. $L_v J =0$ where $L_v$ is the Lie derivative along $v$. A vector field $\Xi$ is called \textbf{holomorphic} if it is a global section of the holomorphic tangent sheaf $T_X$, i.e.~$\Xi \in H^0 (X,T_X)$.
\end{definition}

%
%

We now write $\text{Aut} (X,L)$ for the subgroup of $\textup{Aut} (X)$ consisting of the elements whose action lifts to an automorphism of the total space of the line bundle $L$, and write $\mathrm{Aut}_0 (X,L)$ for the connected component of $\mathrm{Aut} (X,L)$ containing the identity; see \S \ref{bgautomloag} for more details on this group. It is known that for any $v \in \mathrm{LieAut}_0 (X,L)$ and a \kah metric $\omega$ on $X$ there exists $f \in C^{\infty} (X, \cx)$ such that
\begin{equation*}
\iota (v^{1,0}) \omega = - \bar{\partial} f ,
\end{equation*}
called the \textbf{holomorphy potential} of $v^{1,0}$ with respect to $\omega$, where $\iota$ denotes the interior product. Conversely, if a holomorphic vector field $\Xi \in H^0 (X,T_X)$ admits a holomorphy potential, its real part $\mathrm{Re} (\Xi)$ lies in $\mathrm{LieAut}_0 (X,L)$. The reader is referred to \cite[Theorem 1]{lebsim}, and \cite[Theorems 9.4 and 9.7]{kobayashi} for more details.




Now we define an operator $\mathfrak{D}_{\omega} : C^{\infty} (X , \cx) \to C^{\infty} (T^{1,0} X \otimes \Omega^{0,1}(X))$ by
\begin{equation*}
\mathfrak{D}_{\omega} \phi: = \bar{\partial} (\text{grad}^{1,0}_{\omega} \phi)
\end{equation*}
where $\text{grad}^{1,0}_{\omega} \phi$ is the $T^{1,0}X$-component of the gradient vector field $\text{grad}_{\omega} \phi$ of $\phi$ with respect to $\omega$ and $\bar{\partial}$ is the (0,1)-part of the Chern connection on $TX$. Thus $\mathfrak{D}_{\omega} \phi =0$ if and only if $\text{grad}^{1,0}_{\omega} \phi$ is a holomorphic vector field. Writing $\mathfrak{D}_{\omega}^*$ for the formal adjoint of $\mathfrak{D}_{\omega}$ with respect to $\omega$, we have the following formula (cf.~\cite{lebsim})
\begin{equation} \label{defcxvolichopdds}
\mathfrak{D}_{\omega}^* \mathfrak{D}_{\omega} \phi = \Delta^2_{\omega} \phi + (\textup{Ric} (\omega) , \ai \ddbar \phi)_{\omega} +  (\partial S(\omega) , \bar{\partial} \phi)_{\omega} ,
\end{equation}
where $(,)_{\omega}$ stands for the pointwise inner product on the space of differential forms defined by $\omega$, and $\Delta_{\omega}$ is the negative $\bar{\partial}$-Laplacian $- \bar{\partial} \bar{\partial}^* - \bar{\partial}^* \bar{\partial} $. Note that this is a fourth order self-adjoint elliptic operator, but may \textit{not} be a real operator; $\mathfrak{D}_{\omega}^* \mathfrak{D}_{\omega} \phi$ may be a $\cx$-valued function even when $\phi$ is a real function, due to the third term $ (\partial S(\omega) , \bar{\partial} \phi)_{\omega}$ of $\mathfrak{D}_{\omega}^* \mathfrak{D}_{\omega}$. On the other hand, note the obvious $\ker \mathfrak{D}_{\omega}^* \mathfrak{D}_{\omega} = \ker \mathfrak{D}_{\omega}$, since $X$ is compact.

 We define another operator $\lich{\omega}: C^{\infty} (X , \rl) \to C^{\infty} (X , \rl)$ by 
\begin{equation} \label{deflichopomegavarphi}
\lich{\omega} \phi = \Delta^2_{\omega} \phi + (\textup{Ric} (\omega) , \ai \ddbar \phi)_{\omega} + \frac{1}{2} (dS(\omega) , d \phi)_{\omega} .
\end{equation}
This is a 4-th order self-adjoint elliptic operator, which we call the \textbf{Lichnerowicz operator}.

We observe that we can write $\lich{\omega} = \frac{1}{2} (\mathfrak{D}_{\omega}^* \mathfrak{D}_{\omega} + \overline{\mathfrak{D}_{\omega}^* \mathfrak{D}_{\omega}})$, where the operator $\overline{\mathfrak{D}_{\omega}^* \mathfrak{D}_{\omega}}$ is defined by $\overline{\mathfrak{D}_{\omega}^* \mathfrak{D}_{\omega}} \phi = \Delta^2_{\omega} \phi + (\textup{Ric} (\omega) , \ai \ddbar \phi)_{\omega} +  (\bar{\partial} S(\omega) , \partial \phi)_{\omega}$. Thus the kernels of $\lich{\omega}$ and $\mathfrak{D}_{\omega}^* \mathfrak{D}_{\omega}$ may not have anything to do with each other when we consider $\cx$-valued functions in general, but we have the following well-known lemma for the real functions.

\begin{lemma} \label{lemlichgr}
A real function $\phi \in C^{\infty}(X , \rl)$ satisfies $\lich{\omega} \phi=0$ if and only if $ \mathfrak{D}_{\omega} \phi =0$. 
\end{lemma}

Suppose now that we consider a Hamiltonian vector field $v_{\phi}$ generated by $\phi \in C^{\infty} (X , \rl)$ with respect to $\omega$. We use the sign convention $\iota(v_{\phi}) \omega = -d \phi$ for the Hamiltonian. We observe that we can write 
\begin{equation} \label{eqgrvfhamcfj}
\mathrm{grad}_{\omega} \phi =  - Jv_{\phi},
\end{equation}
where $J$ is the complex structure on $TX$.  Recall that $\textup{grad}^{1,0}_{\omega} \phi $ being a holomorphic vector field is equivalent to $\textup{grad}_{\omega} \phi$ being a real holomorphic vector field, and that a vector field $v_{\phi}$ is real holomorphic if and only if $Jv_{\phi}$ is real holomorphic \cite[Chapter IX]{kn}. We thus get the following well-known result.

\begin{lemma} \label{hamisomhvf}
Suppose that $\phi \in C^{\infty} (X , \rl)$ satisfies $\bar{\partial} \mathrm{grad}^{1,0}_{\omega}  \phi =0$ (or equivalently $\lich{\omega} \phi =0$). Then the Hamiltonian vector field $v_{\phi}$ generated by $\phi$ with respect to $\omega$ is a real holomorphic vector field. Conversely, if the Hamiltonian vector field $v_{\phi}$ is real holomorphic, we need to have $\bar{\partial} \mathrm{grad}^{1,0}_{\omega} \phi =0$ (or equivalently $\lich{\omega} \phi =0$).
\end{lemma}

\begin{remark} \label{remhamkilrehvfkah}
Note that, since $\omega$ is K\"ahler, a Hamiltonian real holomorphic vector field must preserve the associated Riemannian metric $g = \omega ( \cdot , J \cdot)$, and hence is necessarily a Hamiltonian Killing vector field with respect to $g$.
\end{remark}

We recall that the \textbf{scalar curvature} $S(\omega)$ of a \kah metric $\omega$ is the metric contraction of the Ricci curvature $\mathrm{Ric} (\omega)$ of $\omega$. 

We also recall the definition of an extremal metric, as introduced by Calabi \cite{cal1}.

\begin{definition}
A \kah metric $\omega \in c_1 (L)$ is called $\textbf{extremal}$ if it satisfies $\mathfrak{D}_{\omega} S(\omega) =0$, or equivalently $\lich{\omega} S(\omega) =0$. If $\omega$ satisfies a stronger condition $S(\omega) = \cst$, it is called a a \textbf{constant scalar curvature \kah metric} which is usually abbreviated as a $\textbf{cscK}$ metric.

\end{definition}

Suppose now that $\omega$ is an extremal metric, so that $\mathrm{grad}_{\omega}^{1,0} S(\omega)$ is a holomorphic vector field. By the above argument and the equation (\ref{eqgrvfhamcfj}), $J \mathrm{grad}_{\omega}S(\omega)$ is a real holomorphic vector field equal to to the Hamiltonian vector field $v_s$ generated by $S(\omega)$.

\begin{definition} \label{intdefextvecfieldci}
The Hamiltonian real holomorphic vector field $v_s$ generated by the scalar curvature $S(\omega)$ of an extremal metric $\omega$, satisfying $\iota (v_s) \omega = -d S(\omega)$, is called an \textbf{extremal vector field}.
\end{definition}

By taking the $(0,1)$-component of the equation $\iota (v_s) \omega = -d S(\omega)$, we have $\iota (v^{1,0}_s) \omega = - \bar{\partial} S(\omega)$, i.e.~$S(\omega)$ is the holomorphy potential of $v_s^{1,0}$, and hence $v_s \in \mathrm{LieAut}_0 (X,L)$ by the argument given earlier in this section. This implies that if $\mathrm{Aut}_0 (X,L)$ is trivial, we have $v_s=0$ and hence an extremal metric is necessarily a cscK metric.

%

\subsubsection{Linearisation of the automorphism group} \label{bgautomloag}


We now recall the following result concerning the automorphism group of polarised \kah manifolds and its linearisation. This is a well-known consequence of the results presented in \cite{fujiki, kobayashi, lebsim, mfk}.

\begin{lemma} \label{lemdefofthtosl}
By replacing $L$ by a large tensor power if necessary, we have a unique faithful group representation
\begin{equation*}
\theta : \textup{Aut}_0 (X , L) \to SL(H^0 (X , L^k))
\end{equation*}
for all $k \in \mathbb{N}$, which satisfies 
\begin{equation} \label{liftauttheta}
\theta (f) \circ \iota = \iota \circ f
\end{equation}
for any $f \in  \textup{Aut}_0 (X , L)$ and the Kodaira embedding $\iota : X \inj \prj (H^0 (X , L^k)^*)$. 
\end{lemma}

\begin{remark} \label{remlinalglinearisation}
Recalling that $\textup{Aut}_0 (X , L)$ is the maximal connected linear algebraic subgroup in $\text{Aut}_0 (X)$, Lemma \ref{lemdefofthtosl} is simply re-stating the well-known fact that, for any connected linear algebraic group $G$ acting on $X$, $L$ admits a $G$-linearisation after raising it to a higher tensor power, say $L^r$, if necessary (cf.~\cite[Corollary 1.6]{mfk}). In other words, having $\theta$ as above in Lemma \ref{lemdefofthtosl} is equivalent to fixing an $\textup{Aut}_0 (X , L)$-linearisation of the line bundle $L$, by replacing $L$ by $L^r$ if necessary. It is well-known that we cannot always take $r=1$ \cite[\S 3]{mfk}. It is also well-known that a linearisation of a $G$-action on a projective variety $X$ is unique up to the fibrewise $\cx^*$-action (cf.~\cite[pp105-106]{dolgachev}, \cite[Proposition 1.4]{mfk}).
\end{remark}

\subsubsection{Automorphism groups of extremal K\"{a}hler manifolds} \label{secautgpoekm}

Now, suppose that $(X,L)$ contains an extremal \kah metric $\omega$. As we remarked in \S \ref{blichop}, we have $\text{grad}_{\omega} S(\omega) = - Jv_s$, where $v_s$ is the extremal vector field (cf.~Definition \ref{intdefextvecfieldci}). Lemmas \ref{lemlichgr} and \ref{hamisomhvf} (and also Remark \ref{remhamkilrehvfkah}) imply that $v_s$ is a Hamiltonian Killing vector field of $\omega$. On the other hand, a well-known theorem of Calabi \cite{cal2} asserts that the identity component of the isometry group $\text{Isom} (\omega)$ of an extremal metric $\omega$ is a maximal compact subgroup of $\text{Aut}_0 (X)$. We now set and fix $K := \text{Isom} (\omega) \cap \text{Aut}_0 (X,L)$ once and for all as the (connected) maximal compact subgroup of $\text{Aut}_0 (X,L)$. The above discussion means that we have $v_s \in \mathfrak{k} := \text{Lie}(K)$. In fact, $v_s$ lies in the centre of $\mathfrak{k}$ by Lemma \ref{lemrbgfgenvfcentre}, which means, in particular, that the identity component $Z(K)_0$ of the centre $Z(K)$ of $K$ must be nontrivial if $X$ admits a non-cscK extremal metric.

Recall that we can write $\text{Aut}_0 (X,L) = K^{\cx} \ltimes R_u$ as a semidirect product of the complexification $K^{\cx}$ of $K$ and the unipotent radical $R_u$ of $\text{Aut}_0 (X,L)$ (recalling that it is a linear algebraic group, cf.~\cite{fujiki, futmab}).

\begin{notation} \label{notgrautliealg}
We summarise our notational convention as follows.
\begin{enumerate}
\item $G :=  \text{Aut}_0 (X,L)$ and $\theta : G \to SL(H^0 (X , L^k))$ is the faithful representation of $G$ as defined in Lemma \ref{lemdefofthtosl}, and we write $\theta_* : \textup{Lie}(G) \to \mathfrak{sl}(H^0 (X , L^k))$ for the induced (injective) Lie algebra homomorphism,
\item $K \le G$ is the group of isometries of the extremal \kah metric $\omega$ inside $G$; $K := \textup{Isom} (\omega) \cap G$. This is a maximal compact subgroup of $G$ and we write $G = K^{\cx} \ltimes R_u$ as a semidirect product of the complexification $K^{\cx}$ of $K$ and the unipotent radical $R_u$ of $G$,
\item $\lieg : = \text{Lie} (G)$, $\mathfrak{k} := \text{Lie} (K)$, and $\mathfrak{z} : = \text{Lie} (Z(K))$; we may also write $\mathfrak{sl}$ for $\mathfrak{sl}(H^0 (X , L^k))$.
\end{enumerate}
\end{notation}

In what follows, we occasionally confuse $G$ with $\theta (G) \le SL(H^0 (X , L^k))$, and $\lieg$ with $\theta_* (\lieg) \le \mathfrak{sl}(H^0 (X , L^k))$.


\subsubsection{Some technical remarks} \label{rmonadjuni}

Let $K$ be a maximal compact subgroup of $\text{Aut}_0 (X,L)$. By Lemma \ref{lemdefofthtosl}, we can consider the action of $K$ on $H^0 (X , L^k)$ afforded by $\theta$, and hence it makes sense to consider $K$-invariant (or more precisely $\theta (K)$-invariant) hermitian forms on $H^0 (X , L^k)$. Observe now the following lemma.


\begin{lemma} \label{rmonadjunilem}
If $f \in K$, $\theta (f)$ is unitary with respect to any $K$-invariant positive hermitian form on $H^0 (X , L^k)$, and $A \in  \theta_* (\ai \mathfrak{k})$ is a hermitian endomorphism with respect to any $K$-invariant positive hermitian form on $H^0 (X , L^k)$. Conversely, if $A \in \theta_* (\mathfrak{k} \oplus \ai \mathfrak{k})$ is hermitian with respect to a $K$-invariant hermitian form, then $A \in \theta_* (\ai \mathfrak{k})$.
\end{lemma}

In what follows, we shall confuse a positive definite hermitian form $\langle , \rangle_H$ with a positive definite hermitian endomorphism $H$, by fixing a reference $\langle, \rangle_{H_0}$. It is convenient in what follows to use a $\langle, \rangle_{H_0}$-orthonormal basis as a ``reference'' basis for $H^0 (X , L^k)$. Although it is simply a matter of convention, this certainly enables us to fix a ``reference'' once and for all.

\begin{notation}
In what follows, we shall write $\mathcal{B}_k$ for the set of all positive definite hermitian forms on $H^0 (X ,L^k)$. Observe $\mathcal{B}_k \cong GL(N , \cx) / U(N)$ and that the tangent space of $\mathcal{B}_k$ at a point is the set $\mathrm{Herm} (H^0 (X,L^k))$ of all hermitian endomorphisms on $H^0 (X,L^k)$. We shall also write $\mathcal{B}_k^K$ for the $\theta(K)$-invariant elements in $\mathcal{B}_k$, and $\mathrm{Herm} (H^0 (X,L^k))^K$ for the tangent space at a point in $\mathcal{B}^K_k$, which is the set of all hermitian endomorphisms on $H^0 (X,L^k)$ commuting with the elements in $\theta(K)$.
\end{notation}

Finally, since the action of $G$ on $X$ is holomorphic, observe
\begin{equation} \label{vinkjvinaik}
v \in \mathfrak{k} \Rightarrow Jv \in \ai \mathfrak{k} .
\end{equation}

\subsection{Review of Donaldson's quantisation} \label{quantisationsch}
We now recall the details of Donaldson's quantisation, namely the maps $Hilb$ (``quantising map'') and $FS$ (``dequantising map''), following the exposition given in \cite{donproj2}. Heuristically, it aims to associate the projective geometry of $\prj (H^0 (X,L^k)^*)$ to the differential geometry of $(X, L^k)$, up to an error which decreases as $k \to \infty$ (``semiclassical limit''), thereby hoping that a difficult PDE problem in differential geometry (e.g.~$\bar{\partial} S(\omega)=0$ or $\bar{\partial} \mathrm{grad}_{\omega}^{1,0} S(\omega)=0$) can be reduced to a finite dimensional problem on $H^0 (X,L^k)$ up to an error of order $k^{-1}$, say (cf.~Theorem \ref{bergexp}). Let $\mathcal{H} (X,L)$ be the space of all positively curved hermitian metrics on $L$, which is the same as the set of all \kah potentials $\mathcal{K} = \{ \phi \in C^{\infty} (X , \rl) \mid \omega_0 + \ai \ddbar \phi >0 \}$ in $c_1(L)$ (where $\omega_0 \in c_1 (L)$ is a reference metric). We may confuse $h \in \mathcal{H} (X,L)$ with the associated \kah metric $\omega_h \in \mathcal{K}$ when it seems appropriate. 

\begin{definition}
The map $Hilb : \mathcal{H} (X,L) \to \mathcal{B}_k$, where $\mathcal{B}_k$ is the set of all positive definite hermitian forms on $H^0 (X,L^k)$, is defined by
\begin{equation*}
Hilb (h) : = \frac{N}{V} \int_X h^k ( , ) \frac{\omega_h^n}{n!} 
\end{equation*}
(recalling Notation \ref{notdimh0nvol}), and the map $FS :  \mathcal{B}_k \to \mathcal{H} (X,L) $ is defined by the equation
\begin{equation}
\sum_{i=1}^N |s_i|^2_{FS(H)^k} =1 \label{defoffseq}
\end{equation}
where $\{ s_i \}$ is an $H$-orthonormal basis for $H^0 (X , L^k)$. $FS(H)$ may also be written as $h_{FS(H)}$. Observe that, fixing a reference hermitian metric $h_0$ on $L$ and writing $FS(H) = e^{- \varphi} h_0$, the equation (\ref{defoffseq}) implies $\varphi = \frac{1}{k} \log \left( \sum_{i=1}^{N} |s_i|^2_{h_0^k} \right) $. Thus, the equation (\ref{defoffseq}) uniquely defines a hermitian metric $h_{FS(H)}$ on $L$, and hence the map $FS$ is well-defined.
\end{definition}



We recall the following lemmas which the author believes are well-known to the experts. The proof of these results are given in \cite{yhhilb}.


\begin{lemma} \label{lemsurjhilb}
Suppose that $L^k$ is very ample. Then $Hilb : \mathcal{H} (X , L) \to \mathcal{B}_k$ is surjective.
\end{lemma}

\begin{lemma} \label{lemfsinj}
Suppose that we choose $k$ to be large enough, and that $H , H' \in \mathcal{B}_k$ satisfy 
\begin{equation*}
	FS(H)^k = (1 +f) FS(H')^k
\end{equation*}
with $\sup_X |f| \le \epsilon$ for $\epsilon \ge 0$ satisfying $N^{\frac{3}{2}} \epsilon \le 1/4$. 

Then we have $|| H - H' ||_{op} \le 2 N^2 \epsilon$, where $|| \cdot ||_{op}$ is the operator norm, i.e.~the maximum of the moduli of the eigenvalues (cf.~\S \ref{rmonadjuni}). In particular, considering the case $\epsilon=0$, we see that $FS$ is injective for all large enough $k$. 
\end{lemma}

In order to describe the map $FS \circ Hilb : \mathcal{H} (X,L) \to \mathcal{H} (X,L)$ (cf.~Theorem \ref{hilbfs}), we introduce the following function which is important in complex geometry and complex analysis.
\begin{definition} \label{defofbergfn}
Let $h \in \mathcal{H} (X,L)$, and let $\{ s_i \}$ be a $\int_X h^k ( , ) \frac{\omega_h^n}{n!}$-orthonormal basis for $H^0 (X , L^k)$. The \textbf{Bergman function} or the \textbf{density of states function} $\rho_k (\omega_h)$ is defined as
\begin{equation*}
\rho_k (\omega_h) : =\sum_{i=1}^{N} |s_i|^2_{h^k}. 
\end{equation*}
We will also use a scaled version of $\rho_k (\omega_h)$ defined as
\begin{equation*}
\bar{\rho}_k (\omega_h) := \frac{V}{N} \rho_k (\omega_h) ,
\end{equation*}
where the scaling is made so that the average of $\bar{\rho}_k (\omega_h)$ over $X$ is 1.
\end{definition}
It is easy to see that $\rho_k (\omega_h)$ depends only on the \kah metric $\omega_h$ rather than $h$ itself, i.e. is invariant under the scaling $h \mapsto e^c h$ for any $c \in \rl$. Recall now the following theorem, which easily follows from the definition (\ref{defoffseq}) of $FS$.
\begin{theorem}
\emph{(Rawnsley \cite{rawnsley})} \label{hilbfs}  
$FS(Hilb (h)) = ( \rho_k (\omega_h) V/N)^{-1/k} h$ for any $h \in \mathcal{H} (X,L)$ and large enough $k>0$ such that $L^k$ is very ample.
\end{theorem}


An obvious corollary of Theorem \ref{hilbfs} is that $FS(Hilb (h)) = h$ if and only if $\rho_k (\omega_h) = \cst = N/V$, and $h \in \mathcal{H} (X,L)$ satisfying this is called balanced.

\begin{definition} \label{defbalmusu2defs}
A hermitian metric $h \in \mathcal{H} (X,L)$ is called $\textbf{balanced}$ at the level $k$ if it satisfies the following two equivalent conditions.
\begin{enumerate}
\item $\rho_k (\omega_h) = N/V$ or $\bar{\rho}_k (\omega_h) =1$,
\item $FS(Hilb(h)) = h$.
\end{enumerate}
\end{definition}


An important point is that we have an ``extrinsic'' characterisation of balanced metrics, in terms of the Kodaira embedding. For this, we fix some basis $\{ Z_i \}$ for $H^0 (X , L^k)$, which may be called a \textbf{reference basis}. With this choice of basis, it is possible to identify $H^0 (X,L^k)$ with its dual, and also with $\cx^N$, and hence $\prj (H^0 (X,L^k)^*) \cong \prj^{N-1}.$ Note then that the Kodaira embedding $\iota$ can be written as
\begin{equation*}
\iota: X \ni x \mapsto [\text{ev}_x Z_1 : \cdots : \text{ev}_x Z_N] \in \prj^{N-1}
\end{equation*}
where $\text{ev}_x$ is the evaluation map at $x$. This embedding may be called a \textbf{reference embedding}, and will always be denoted by $\iota$ from now on. It is important to fix \textit{some} reference basis for the identification $\prj (H^0 (X,L^k)^*) \cong \prj^{N-1}$, but a different choice of reference basis will only result in moving (the image of) $X$ inside $\prj^{N-1}$ by an $SL(N, \cx)$-action (cf.~Remark \ref{anochrefbasis}).

\begin{definition}
Defining a standard Euclidean metric on $\cx^N$ which we write as the identity matrix $I$, we define the \textbf{centre of mass} as
\begin{equation*}
\bar{\mu}_X : = \int_{\iota(X)}\frac{\tilde{h}_{FS} (Z_i , Z_j)}{\sum_l |Z_l|_{\widetilde{FS}}^2} \frac{\omega^n_{\widetilde{FS}}}{ n!} = \int_{X}\frac{h^k_{FS} (s_i , s_j)}{\sum_l |s_l|_{FS^k}^2} \frac{k^n \omega^n_{FS}}{n!}  \in \ai \mathfrak{u} (N)
\end{equation*}
where $h^k_{FS}$ is (the pullback by the Kodaira embedding of) the Fubini--Study metric $\tilde{h}_{FS}$ on $\prj^{N-1}$ induced from $I$ on $\cx^N$ covering $\prj^{N-1}$ (see also Notation \ref{remonnotationzs} below).
\end{definition}

\begin{remark} \label{remdefcomitofs}
Note that the equation (\ref{defoffseq}) implies that we in fact have $\bar{\mu}_X = \int_X h^k_{FS} (s_i , s_j) \frac{k^n \omega_{FS}^n}{n!}$.
\end{remark}


\begin{notation} \label{remonnotationzs}
As a matter of notation, we will often write $\{ Z_i \}$ for a basis for $H^0 (X , L^k)$ when we see it as an abstract vector space and $\{ s_i \}$ when we see it as a space of holomorphic sections on $X$; thus we can write $\iota^* Z_i = s_i$ by using the Kodaira embedding $\iota$. We also write $\tilde{h}_{FS}$ for the Fubini--Study metric on $\mathcal{O}_{\prj^{N-1}} (1)$ induced from $I$ on $\cx^N$ covering $\prj^{N-1}$, and write $\omega_{\widetilde{FS}}$ for the corresponding \kah metric on $\prj^{N-1}$.
\end{notation}

We can now move the image of $X$ in $\prj^{N-1}$ by the $SL(N ,\cx)$-action on $\prj^{N-1}$ (or rather on the $\cx^N$ covering it). Writing $\xi_g : \prj^{N-1} \isom \prj^{N-1}$ for the biholomorphic map induced from $g \in SL(N, \cx)$, note that moving the image $\iota (X)$ of $X$ by $g \in SL(N, \cx)$ is equivalent to considering the embedding $\iota_g := \xi_g \circ \iota: X \inj \prj^{N-1}$, and the effect of $\xi_g$ is such that $Z_i$ changes to $Z'_i := \sum_j g_{ij} Z_j$, where $g_{ij}$ is the matrix for $g$ represented with respect to the basis $\{ Z_i \}$. Thus, the Fubini--Study metric $\omega_{FS} = \iota^*  \frac{\ai}{2 \pi k} \ddbar \log (\sum |Z_i|^2)$ changes to $(\xi_g \circ \iota)^*  \frac{\ai}{2 \pi k} \ddbar \log (\sum |Z_i|^2) =  \iota^*   \frac{\ai}{2 \pi k} \ddbar \log (\sum |Z'_i|^2) $, which we can see is equal to $\omega_{FS(H)}$, i.e.~(the pullback by $\iota$ of) the Fubini--Study metric on $\prj^{N-1}$ induced from the hermitian form $H := \overline{(g^{-1})^t} g^{-1}$ on $\cx^N$.

Thus, writing $\bar{\mu}_X (g)$ for the new centre of mass after moving the image of $X$ by $g$, namely the centre of mass of $X$ with respect to the embedding $\iota_g = \xi_g \circ \iota$, we have
\begin{align*}
\bar{\mu}_X (g)  & = \int_{\iota_g (X)} \frac{\tilde{h}_{FS}( Z_i , Z_j)}{\sum_l |Z_l|_{\widetilde{FS}}^2} \frac{\omega^n_{\widetilde{FS}}}{ n!}  \\
&=\int_{\iota (X)} \frac{ \tilde{h}_{FS(H)}( Z'_i , Z'_j)}{\sum_l |Z'_l|_{\widetilde{FS}(H)}^2} \frac{\omega^n_{\widetilde{FS}(H)}}{ n!} = \int_{X} \frac{h^k_{FS(H)}( s'_i , s'_j)}{\sum_l |s'_l|_{FS(H)^k}^2} \frac{k^n\omega^n_{FS(H)}}{n!} .
\end{align*}

\begin{remark} \label{anochrefbasis}
Suppose that we have another choice of reference basis, say $\{ Z'_i \}$, to compute the centre of mass, say $\bar{\mu}'_X$. Since we can write $Z'_i = \sum_j g_{ij} Z_j$ for some $g \in SL(N, \cx)$, we see that choosing a new reference basis is simply moving the image of $X$ inside $\prj^{N-1}$ (with respect to the old reference basis) by $g \in SL(N, \cx)$; namely $\bar{\mu}'_X = \bar{\mu}_X(g)$. Observe that the basis $\{ Z'_i\}$ is an $H$-orthonormal basis where the hermitian form $H$ is defined by $H= \overline{(g^{-1})^t} g^{-1}$.

\end{remark}


\begin{definition}
The Kodaira embedding $\iota : X \inj \prj^{N-1}$ is called \textbf{balanced} if there exists $g \in SL(N, \cx)$ such that $\bar{\mu}_X (g)$ is a multiple of the identity in $\ai \mathfrak{u} (N)$; equivalently, $\bar{\mu}_X (g)$ is in the kernel of the natural projection $\ai \mathfrak{u}(N) \surj \ai \mathfrak{su} (N)$.
\end{definition}

Note that the definition of being balanced does not depend on the choice of reference basis that we chose to have $\prj(H^0 (X , L^k)^*) \isom \prj^{N-1}$, by Remark \ref{anochrefbasis}. 

A fundamental result is the following, which easily follows from Lemma \ref{lemfsinj}, Definition \ref{defofbergfn}, and Remark \ref{remdefcomitofs}.

\begin{theorem} \label{balembbalmetequiv}
\emph{(Luo \cite{luo}, Zhang \cite{zhang})}
Kodaira embedding $\iota : X \inj \prj(H^0 (X,L^k)^*) \cong \prj^{N-1}$ is balanced if and only if $L$ admits a balanced metric at the level $k$.
\end{theorem}


Finally, we prove the following general lemma. The author thanks Joel Fine for pointing it out to him.
\begin{lemma} \label{lembergfnitoham}
Let $F$ be the Hamiltonian for the vector field on $\prj^{N-1}$ generated by $\ai k^n \bar{\mu}_X (g)^{-1}$ with respect to the \kah metric $\omega_{\widetilde{FS}(H)}$,  where $H = \overline{(g^{-1})^t} g^{-1}$. Then, $\rho_k (\omega_{FS(H)}) = \iota^* F$.
\end{lemma}

\begin{proof}
Let $\{ s_i \}$ be an $H$-orthonormal basis and $\{ s'_i \}$ be a $\int_X h^k_{FS(H)} (,) \frac{\omega_{FS(H)}^n}{n!}$-orthonormal basis. Let $P$ be the change of basis matrix from $\{ s_i \}$ to $\{ s'_i \}$. This implies 
\begin{equation*}
\sum_{l,q} P^*_{li} P_{jq}   (\bar{\mu}'_X)_{lq} =  \sum_{l,q} P^*_{li} P_{jq} \int_X h^k_{FS(H)} (s_l, s_q) \frac{k^n \omega^n_{FS(H)}}{n!} = k^n \delta_{ij},
\end{equation*}
which implies $\bar{\mu}_X ' = k^n (P^*P)^{-1} $, where $\bar{\mu}'_X$ is the centre of mass defined with respect to the basis $\{ s_i \}$. Note $\bar{\mu}'_X= \bar{\mu}_X (g)$ by Remark \ref{anochrefbasis}. Note also that

\begin{equation*}
\rho_k (\omega_{FS(H)}) = \sum_i |s'_i|^2_{FS(H)^k} =  \sum_{i,q,l} P^*_{li} P_{iq} h^k_{FS(H)} (s_l, s_q) =  \sum_{q,l} (P^* P)_{lq} h^k_{FS(H)} (s_l, s_q),
\end{equation*}
and hence we get
\begin{equation} \label{lbergfnitoham1}
\rho_k (\omega_{FS(H)}) =  \sum_{i,j} (k^n \bar{\mu}_X (g)^{-1})_{ij}  h^k_{FS(H)} (s_i, s_j).
\end{equation}

Now, using the homogeneous coordinates $\{ Z_i\}$ on $\prj^{N-1}$ corresponding to $\{ s_i \}$, i.e.~$\iota^*Z_i = s_i$, we have
\begin{equation} \label{lbergfnitoham2}
\sum_{i,j} (\bar{\mu}_X (g)^{-1})_{ij}  h^k_{FS(H)} (s_i, s_j) = \iota^* \left( \sum_{i,j} (\bar{\mu}_X (g)^{-1})_{ij} \frac{Z_i \bar{Z}_j}{ \sum_l |Z_l|^2} \right).
\end{equation}

Recall that for $A \in \mathfrak{u}(N)$ regarded as a Hamiltonian vector field on $\prj^{N-1}$, the Hamiltonian $F_A$ for $A$ with respect to $\omega_{\widetilde{FS}(H)}$ is given by (cf.~\cite[p88]{sze})
\begin{equation} \label{lbergfnitoham3}
F_A = - \ai \sum_{ij} A_{ij}\frac{Z_i \bar{Z}_j}{ \sum_l |Z_l|^2}.
\end{equation}
Thus, taking $A = \ai k^n \bar{\mu}_X (g)^{-1} \in \mathfrak{u}(N)$, we get the claimed statement from the equations (\ref{lbergfnitoham1}), (\ref{lbergfnitoham2}), and (\ref{lbergfnitoham3}).

\end{proof}

\subsection{A general lemma and its consequences} \label{lemholconseq}


We prove the following general lemma.

\begin{lemma} \label{lemhol}
For any $f \in  \textup{Aut}_0 (X, L)$, 
\begin{enumerate}
\item $f^* \rho_k (\omega_h) = \rho_k (f^* \omega_h)$,
\item $Hilb (f^*h) = \theta(f^{-1})^* Hilb(h) \theta(f^{-1})$, 
\item $f^* FS(H) =  FS(\theta(f^{-1})^* H \theta(f^{-1})) $ .
\end{enumerate}
\end{lemma}

\begin{remark} \label{remonpbofhmc}
We recall now that we have $\theta : \textup{Aut}_0 (X , L) \to SL(H^0 (X , L^k)) $ as in Lemma \ref{lemdefofthtosl} (by replacing $L$ by a large enough tensor power if necessary) which implies that we have a ``consistent'' choice of the lift $\tilde{f}$ of $f \in \textup{Aut}_0 (X , L)$ to the automorphism of the total space of the bundle $L$ so that $\tilde{f_1} \circ \tilde{f_2} = \widetilde{f_1 \circ f_2}$ (i.e.~fixed linearisation of the action; see Remark \ref{remlinalglinearisation}). For a hermitian metric $h$ on $L$, $f^*h$ in the above statement is meant to be $\tilde{f}^* h$ for this choice of $\tilde{f}$.


\end{remark}

\begin{proof}
Note the elementary
\begin{equation*}
\int_X h^k(s , s') \frac{\omega^n_h}{n!} = \int_X f^*( h^k(s , s') ) \frac{f^* \omega^n_h}{n!} = \int_X (f^*h^k)( \theta(f) s , \theta (f) s') \frac{f^*\omega^n_h}{n!} 
\end{equation*}
for any two sections $s$ and $s'$, by recalling (\ref{liftauttheta}). This means that, if $\{ s_i \}$ is a $Hilb(h)$-orthonormal basis, then $\{ \sum_j \theta(f)_{ij} s_j \}$ is a $Hilb(f^*h^k)$-orthonormal basis where $\theta(f)_{ij}$ is the matrix for $\theta (f)$ represented with respect to $\{ s_i \}$. We thus have $f^* \rho_k (\omega_h) = \sum_i |\sum_{j} \theta (f)_{ij} s_j|^2_{f^*h^k} = \rho_k (f^* \omega_h)$.

For the second part of the lemma, we just recall that $\{ \sum_j \theta (f)_{ij} s_j \}$ is a $Hilb(f^*h)$-orthonormal basis to see $Hilb (f^*h) = \theta(f^{-1})^* Hilb(h) \theta(f^{-1})$.

For the third part of the lemma, apply $f^*$ to the defining equation $\sum |s_i|^2_{FS(H)^k} = 1$ for $FS(H)$ (equation (\ref{defoffseq})), where $\{ s_i \}$ is an $H$-orthonormal basis. We then get $\sum_i | \sum_j \theta (f)_{ij} s_j|^2_{f^*FS(H)^k} = 1$, which means that $f^*(FS(H)) = FS(H')$ with $H'$ having $ \{ \sum_j \theta(f)_{ij} s_j \} $ as its orthonormal basis, i.e.~$H' = \theta(f^{-1})^*H \theta(f^{-1})$. Thus $f^* FS(H) =  FS(\theta(f^{-1})^* H \theta(f^{-1}))$.


\end{proof}

From Lemma \ref{lemhol} we obtain the following observation of Sano and Tipler.
\begin{lemma} \label{santipkinvfshilblem}
\emph{(Sano--Tipler \cite[\S 2.2.1]{santip})}
If $H$ is $\theta (K)$-invariant, then $FS(H)$ is $K$-invariant. Conversely, if the \kah metric $\omega_h$ is $K$-invariant, then $Hilb(h)$ defines a $\theta (K)$-invariant hermitian form on $H^0 (X , L^k)$.
\end{lemma}

\section{Construction of approximate solutions to $\bar{\partial} \mathrm{grad}^{1,0}_{\omega_k} \rho_k (\omega_k) =0$}

\label{approximately}

\subsection{Preliminaries} \label{premapprox}

For the sake of convenience, we decide to have the following naming convention.

\begin{definition}
We say that $\omega_{\phi}$ is \textbf{quantised extremal} or \textbf{q-ext} if it satisfies $\lich{\omega_{\phi}} \rho_k (\omega_{\phi}) =0$.
\end{definition}

Suppose now that $(X,L)$ admits an extremal metric $\omega \in c_1(L)$, and that $K$ stands for $\textup{Isom} (\omega) \cap \textup{Aut}_0 (X,L)$ from now on. $\omega$ being extremal, its scalar curvature $S(\omega)$ generates a Hamiltonian Killing vector field $v_s \in \mathfrak{k} $. The first step of the proof of Theorem \ref{sbalmqext} is to construct a metric $\omega'$ which ``approximately'' satisfies $\bar{\partial} \text{grad}^{1,0}_{\omega'} \rho_k (\omega')=0$. We thus consider the following problem.

\begin{problem} \label{probapprox}
Starting with an extremal metric $\omega$ satisfying $\lich{\omega} S(\omega) =0$, can one find for each $m \in \mathbb{N}$ a sequence $\{ H_{m} (k) \}_k$ with $H_{m} (k) \in \mathcal{B}_k^K$ so that $\omega_{(m)} := \omega_{FS(H_{m} (k))}$ satisfies $ ||\omega_{(m)} - \omega ||_{C^l , \omega} < c_{m,l } /k $ for some constant $c_{m,l} >0$ for each $l \in \mathbb{N}$ and all large enough $k$, and also
\begin{equation*}
\left| \left| \lich{\omega_{(m)}} \bar{\rho}_k (\omega_{(m)})  \right| \right|_{C^l} < C_{m,l} (\omega) k^{-m-2}
\end{equation*}
for each $l \in \mathbb{N}$, with a constant $C_{m,l} (\omega)$ that depends only on $m,l$ and $\omega$?
\end{problem}




As in the usual cscK case, the construction of approximately q-ext metrics will crucially depend on the well-known asymptotic expansion of the Bergman function (cf.~Theorem \ref{bergexp}), so that $\lich{\omega_{(m)}} \bar{\rho}_k (\omega_{(m)})$ is going to be zero ``order by order'' in the powers of $k^{-1}$. For this purpose, it turns out that it is easier to work with a pair of equations (cf.~(\ref{pairkinveqphifext})) that is equivalent to $\lich{\omega_{\phi}} \bar{\rho}_k (\omega_{\phi}) =0$, which we discuss shortly.

Before doing so, we briefly recall the explicit formula for describing how the Hamiltonian for the extremal vector field $v_s$ changes when we change the \kah metric from $\omega$ to $\omega_{\phi}:= \omega+ \ai \ddbar \phi$. We have a general lemma as follows.


\begin{lemma} \label{hamextvfwrtdifm}
\emph{(cf.~\cite[Lemma 4.0.1]{aps})}
Suppose that $v \in \mathfrak{k} = \mathrm{Lie}(K)$ is a Hamiltonian Killing vector field, with Hamiltonian $\tilde{F}$ with respect to $\omega$. Suppose also that the Lie derivative of $\phi \in C^{\infty} (X , \rl)$ along $v$ is zero. Then $\tilde{F} + \frac{1}{2}(d \tilde{F} , d \phi)_{\omega}$ is the Hamiltonian of $v$ with respect to $\omega + \ai \ddbar \phi$. Namely,
\begin{equation*}
\iota (v) (\omega + \ai \ddbar \phi) = - d \left( \tilde{F} + \frac{1}{2}  (d \tilde{F} , d \phi)_{\omega} \right) .
\end{equation*}

\end{lemma}

\begin{proof}
Since the complex structure $J$ is $K$-invariant (since $K \le \mathrm{Aut}_0 (X,L)$), we have $L_{v} (J d \phi) =0$, where $L_v$ is the Lie derivative along a vector field $v \in \mathfrak{k} = \textup{Lie} (K)$. In other words, $\iota (v) (dJd \phi) = - d(\iota (v) Jd \phi) =  d (Jv (\phi))$, where we recall that $J$ acts on a 1-form $\alpha$ by $J \alpha (\xi) = - \alpha (J \xi)$ for any vector field $\xi$, which also implies 
\begin{equation} \label{actionofjon1forms}
J \iota(v) \omega (\xi) = - \omega (v , J \xi) = \omega (Jv, \xi) = \iota(Jv) \omega (\xi)
\end{equation}
for any vector field $\xi$. Recall also that $dJd \phi = 2 \ai \ddbar \phi$ (cf.~\cite[\S 3.1]{aps}). We thus have $\iota(v) \ai \ddbar \phi = \frac{1}{2} d (Jv(\phi)) $. Note also that, when $v$ is generated by $\tilde{F}$, we have $-Jv = \text{grad}_{\omega} \tilde{F}$, and hence $Jv(\phi) = - \text{grad}_{\omega} \tilde{F} (\phi) = - (d \tilde{F} ,d \phi)_{\omega}$ where $(,)_{\omega}$ is the pointwise norm on the space of 1-forms defined by the metric $\omega$.

\end{proof}

In what follows, we shall apply the above lemma to the case where $v$ is the extremal vector field $v_s$ and $\tilde{F} = S(\omega)$. Observe also the following well-known fact.

\begin{lemma} \label{lemrbgfgenvfcentre}
Suppose that $\omega_h$ is $K$-invariant. Then the Hamiltonian vector field $v$ generated by a $K$-invariant function $\tilde{F}$ commutes with the action of any element in $K$. In particular, if $v$ is a Hamiltonian Killing vector field with respect to $\omega_h$, $v$ lies in the centre $\mathfrak{z}$ of $\mathfrak{k}$. In particular, the extremal vector field $v_s$ lies in $\mathfrak{z}$.
\end{lemma}

\begin{proof}
Applying $f \in K \le G$ to the equation $\iota(v) \omega_h = - d \tilde{F}$, we have $\iota((f^{-1})_* v) f^*\omega_h = - d f^* \tilde{F}$. Since $\omega_h$ and $\tilde{F}$ are $K$-invariant, this yields $\iota((f^{-1})_*v) \omega_h  = \iota(v) \omega_h  $. Since $\omega_h$ is non-degenerate, we have $(f^{-1})_*v = v$, which is equivalent to saying that the 1-parameter subgroup generated by $v$ commutes with the action of any element in $K$.
\end{proof}



We now consider a pair of equations 
\begin{equation*} 
\begin{cases}
S(\omega) + \frac{1}{2} (dS(\omega) , d \phi )_{\omega} =  4 \pi k \bar{\rho}_k (\omega_{\phi}) +f \\
\lich{\omega_{\phi}} f =0 
\end{cases}
\end{equation*}
to be solved for a pair of $K$-invariant functions $(\phi,f)$, which we will be concerned with from now on. The following lemma shows that solving this equation is equivalent to having a q-ext metric.

\begin{lemma}
Suppose that $\omega$ is an extremal metric and write $\omega_{\phi} = \omega + \ai \ddbar \phi$. There exists a $K$-invariant function $\phi$ which satisfies $\lich{\omega_{\phi}}  \bar{\rho}_k (\omega_{\phi}) =0$ if and only if we can find a pair of $K$-invariant functions $(\phi, f)$ which satisfies
\begin{equation} \label{pairkinveqphifext}
\begin{cases}
S(\omega) + \frac{1}{2} (dS(\omega) , d \phi )_{\omega} =  4 \pi k \bar{\rho}_k (\omega_{\phi}) +f \\
\lich{\omega_{\phi}} f =0 .
\end{cases}
\end{equation}
\end{lemma}

\begin{remark} \label{deponlich}
It is important to note that, in (\ref{pairkinveqphifext}), we need $\lich{\omega_{\phi}} f =0$ and \textit{not} $\lich{\omega}f=0$. This will cause an extra complication in the construction of approximately q-ext metric, which did not happen in the cscK case (cf.~Remark \ref{depfk}).
\end{remark}

Note that $\omega$ being an extremal metric is essential in the above lemma. 


\begin{proof}
Suppose that we can find a pair $(\phi , f)$ of $K$-invariant functions satisfying (\ref{pairkinveqphifext}). Then, recalling Lemma \ref{hamextvfwrtdifm}, we have
\begin{equation*}
\iota (v_s) (\omega + \ai \ddbar \phi) = - d \left( S(\omega)+ \frac{1}{2}  (d S(\omega) , d \phi)_{\omega} \right) = - d \left( 4 \pi k \bar{\rho}_k (\omega_{\phi}) +f \right) .
\end{equation*}
Since $f$ satisfies $\lich{\omega_{\phi}} f =0 $, there exists a real holomorphic vector field $v_f$ such that $\iota (v_f) (\omega + \ai \ddbar \phi) = - df$. Thus we get $\iota (v_s - v_f) (\omega + \ai \ddbar \phi) = - d \left( 4 \pi k \bar{\rho}_k (\omega_{\phi})  \right) $. Since $v_s - v_f$ is a real holomorphic vector field, we have $\lich{\omega_{\phi}}  \bar{\rho}_k (\omega_{\phi}) =0$ by Lemma \ref{hamisomhvf}.

Conversely suppose $\lich{\omega_{\phi}}  \bar{\rho}_k (\omega_{\phi}) =0$. Then there exists a real holomorphic vector field $v$ such that $\iota (v) \omega_{\phi} = - d 4 \pi k \bar{\rho}_k (\omega_{\phi})$. Then, writing $v = (v -v_s) +v_s$, we have
\begin{align*}
- d \left( 4 \pi k \bar{\rho}_k (\omega_{\phi}) \right) 
&= \iota(v_s) \omega_{\phi} + \iota (v -v_s) \omega_{\phi} \\
&= - d \left( S(\omega)+ \frac{1}{2}  (d S(\omega) , d \phi)_{\omega} \right) - d f
\end{align*}
where we put $f := 4 \pi k \bar{\rho}_k (\omega_{\phi}) - S(\omega) - \frac{1}{2}  (d S(\omega) , d \phi)_{\omega} $ in the third line. Note that Lemmas \ref{hamisomhvf} and \ref{hamextvfwrtdifm} imply $\lich{\omega_{\phi}} (S(\omega)+ \frac{1}{2}  (d S(\omega) , d \phi)_{\omega}) =0$. Recalling $\lich{\omega_{\phi}}  \bar{\rho}_k (\omega_{\phi}) =0$ in our assumption, we thus have $\lich{\omega_{\phi}} f =0$. Note that $f = 4 \pi k \bar{\rho}_k (\omega_{\phi}) - S(\omega) - \frac{1}{2}  (d S(\omega) , d \phi)_{\omega}$ is $K$-invariant if $\phi$ is $K$-invariant by Lemma \ref{lemhol}. This gives us an equation $4 \pi k \bar{\rho}_k (\omega_{\phi})  = S(\omega)+ \frac{1}{2}  (d S(\omega) , d \phi)_{\omega} +  f + \cst$. Replacing $f + \cst$ by $f$, we get the equation (\ref{pairkinveqphifext}).

\end{proof}




\subsection{Perturbative construction by using the asymptotic expansion} \label{asymptapprox}

We now recall the following famous theorem, which will be of fundamental importance for us. We refer to \cite{mm} (in particular to \cite[Theorem 4.1.2]{mm}) for more detailed discussions.

\begin{theorem} \label{bergexp}
	\emph{(Tian \cite{tian90}, Yau \cite{yauberg}, Bouche \cite{bouche}, Ruan \cite{ruan}, Zelditch \cite{zelditch}, Lu \cite{lu})}
The Bergman function $\rho_k (\omega_{\phi})$ admits the following asymptotic expansion in $k^{-1}$
\begin{equation*}
\rho_k (\omega_{\phi}) = k^n + k^{n-1} b_1 + k^{n-2} b_2 + \cdots 
\end{equation*}
with $b_1 = \frac{1}{4 \pi} S(\omega_{\phi})$, and each coefficient $b_i = b_i (\omega_{\phi})$ can be written as a polynomial in the curvature $\textup{Riem} (\omega_{\phi})$ of $\omega_{\phi}$ and its derivatives of order $\le 2i-2$, and the metric contraction by $\omega_{\phi}$.

More precisely, there exist smooth functions $b_i$ such that, for any $m,l \in \mathbb{N}$ there exists a constant $C_{m,l}$ such that for any $k \in \mathbb{N}$ we have
\begin{equation*}
\left| \left| \rho_k (\omega_{\phi}) -k^n- \sum_{i=1}^m b_i k^{n-i} \right| \right|_{C^l } < C_{m,l}k^{n-m-1} .
\end{equation*}
Moreover, the constant $C_{m,l}$ can be chosen independently of $\omega_{\phi}$ provided it varies in a family of uniformly equivalent metrics which is compact with respect the $C^{\infty}$-topology.
\end{theorem}

\begin{remark}
In what follows, we shall often use the standard shorthand notation for the asymptotic expansion to write $\rho_k (\omega_{\phi}) = k^n + k^{n-1} b_1 + k^{n-2} b_2 + O(k^{n-3})$ to mean the above statement.
\end{remark}

\begin{remark}
Since $\phi$ is $K$-invariant and $K$ acts as an isometry of $\omega$, each coefficient $b_i$ appearing in this expansion is $K$-invariant. 
\end{remark}

\begin{remark}
Theorem \ref{bergexp} and the Riemann--Roch theorem immediately implies the asymptotic expansion $\bar{\rho}_k (\omega_{\phi}) = 1 + \frac{1}{k} (b_1 - \bar{b}_1) + \frac{1}{k^2} (b_2 - \bar{b}_2) + \cdots $, where $\bar{b}_i$ is the average of $b_i$ over $X$ with respect to $\omega_{\phi}$, which is determined by the Chern--Weil theory (and hence depends only on $(X,L)$ and not on the specific choice of the metric). For notational convenience, we will often use this in the form
\begin{equation*}
\bar{\rho}_k (\omega_{\phi}) = c_1 + \frac{1}{k} b_1 + \sum_{i=2}^m \frac{1}{k^i} (b_i  - \bar{b}_i)+ O(k^{-m-1}),
\end{equation*}
with the constant $c_1: = 1 - \bar{b}_1/k $, in what follows.
\end{remark}

We now recall $S(\omega_{\phi}) = S(\omega) + \mathbb{L}_{\omega} \phi + O(\phi^2) $ where $\mathbb{L}_{\omega}$ is an operator defined by
\begin{equation*}
\mathbb{L}_{\omega} \phi := - \Delta^2_{\omega} \phi -  (\text{Ric} (\omega),  \ai \ddbar \phi)_{\omega}
\end{equation*}
with $\Delta_{\omega}$ being the negative $\bar{\partial}$-Laplacian $- \bar{\partial} \bar{\partial}^* - \bar{\partial}^* \bar{\partial}$. Recall the well-known identity \cite[(2.2)]{lebsim}
\begin{equation*}
\lich{\omega} \phi =  - \mathbb{L}_{\omega} \phi + \frac{1}{2} (dS(\omega) ,  d \phi)_{\omega} .
\end{equation*}

Given these remarks, we now study how the equation (\ref{pairkinveqphifext}) will be perturbed when we perturb the metric $\omega$ to $\omega_{(1)} : = \omega+ \ai \ddbar \phi_1 /k$. First of all we expand $S (\omega_{(1)})$ in $\phi_1/k$, which leads to the following asymptotic expansion
\begin{equation} \label{asexpbicbfo0}
S(\omega_{(1)}) = S(\omega) +\frac{1}{k} \mathbb{L}_{\omega} \phi_1 + O(k^{-2}) 
\end{equation}
in $k^{-1}$.


Note also that each coefficient $b_i $ in the asymptotic expansion of the Bergman function changes as
\begin{equation} \label{asexpbicbfo1}
b_i (\omega_{(1)}) = b_i (\omega) + O(1/k),
\end{equation}
noting that $b_i (\omega_{(1)})$ can be written as a polynomial in the curvature $\textup{Riem} (\omega_{(1)})$ and its derivatives, with the metric contraction by $\omega_{(1)}$. 

\begin{remark} \label{kinvtwo}
Note that this also implies that each coefficient of the powers of $k^{-1}$ in the above expansions (\ref{asexpbicbfo0}) and (\ref{asexpbicbfo1}) is $K$-invariant, if we can choose $\phi_1$ to be $K$-invariant. 
\end{remark}

Thus we have $\rho_k (\omega_{(1)}) = k^n + \frac{k^{n-1}}{4 \pi} S(\omega) + \frac{k^{n-2}}{4 \pi} \left( \mathbb{L}_{\omega} \phi_1 + 4 \pi b_2 (\omega)  \right) + O(k^{n-3})$, which means
\begin{equation*}
4 \pi k \bar{\rho}_k (\omega_{(1)})  =  4 \pi k c_1 + S(\omega) + \frac{1}{k} \left( \mathbb{L}_{\omega} \phi_1 +4 \pi b_2 (\omega) - 4 \pi \bar{b}_2 \right)+  O(k^{-2}) .
\end{equation*}

Note that, for any fixed $\phi'$, we have $\lich{\omega_{(1)}} \phi' = \lich{\omega} \phi' + O(1/k)$ by recalling the formula (\ref{deflichopomegavarphi}) and expanding it in $1/k$. Note also that the second $O(1/k)$ term can be estimated by $C(\phi' ; \omega, \phi_1)/k$, where $C(\phi' ; \omega, \phi_1)$ is a constant which depends only on the $C^4$-norm of $\phi'$ and the $C^{\infty}$-norm of $\omega$ and $\phi_1$. In what follows, we shall (rather abusively) refer to this fact by saying that ``we have $\lich{\omega_{(1)}} = \lich{\omega} + \frac{1}{k}D$, where $D$ is some differential operator of order at most 4 which depends on $\omega$ and $\phi_1$''.


Thus the equation (\ref{pairkinveqphifext}) to be solved becomes, up to the order $1/k$,
\begin{equation*}
\begin{cases}
S(\omega) + \frac{1}{2k} (dS(\omega) , d \phi_1 )_{\omega} \stackrel{?}{=}  4 \pi k c_1 + S(\omega) + \frac{1}{k} \left( \mathbb{L}_{\omega} \phi_1 +4 \pi b_2 (\omega) - 4 \pi \bar{b}_2  \right)+ f  + O(k^{-2}) \\
\lich{\omega_{(1)}} f \stackrel{?}{=} 0.
\end{cases}
\end{equation*}
We write $f_0 := - 4 \pi k c_1$ and decide to find $f$ that is of order $1/k$, i.e. decide to find $f_1$ independent of $k$ such that $f = f_1/k$. Namely, we re-write the above equation as
\begin{equation*}
\begin{cases}
&\frac{1}{2k} (dS(\omega) , d \phi_1 )_{\omega} \stackrel{?}{=}  4 \pi k c_1  +f_0 + \frac{1}{k} \left( \mathbb{L}_{\omega} \phi_1 +4 \pi b_2 (\omega)  +f_1 - 4 \pi \bar{b}_2 \right) + O(k^{-2}) \\
&\ \ \ \ \ \ \ \ \ \ \ \ \ \ \ \ \ \ \ \ \ \ \ \ \ \ \ \ =\frac{1}{k} \left( \mathbb{L}_{\omega} \phi_1 +4 \pi b_2 (\omega)  +f_1 - 4 \pi \bar{b}_2 \right) + O(k^{-2}) \\
&\lich{\omega_{(1)}}( f_0 + f_1/k) = \frac{1}{k} \lich{\omega_{(1)}} f_1 \stackrel{?}{=} 0,
\end{cases}
\end{equation*}
by noting that constant functions generate a trivial holomorphic vector field. We note that by Remark \ref{kinvtwo}, each coefficient of the powers of $k^{-1}$ in the above asymptotic expansion is $K$-invariant, if we choose $\phi_1$ to be $K$-invariant.

We now wish to solve this equation up to the leading order, i.e. the order $O(1/k)$. Namely, we wish to find a $K$-invariant $\phi_1$ such that 
\begin{equation*}
- \mathbb{L}_{\omega} \phi_1 + \frac{1}{2} (dS(\omega) , d \phi_1 )_{\omega} =  4 \pi b_2(\omega) - 4 \pi \bar{b}_2 + f_1  
\end{equation*}
for some $f_1$ which satisfies $\lich{\omega}f_1 =0$ and is $K$-invariant. Recalling that the left-hand side of this equation is equal to $\lich{\omega} \phi_1$ (cf.~(\ref{deflichopomegavarphi})), we are now in place to apply some well-known results concerning the Lichnerowicz operator, collected in the appendix. By applying Lemma \ref{lich1}, we can certainly have a pair $(\phi'_1 ,f'_1)$ of $C^{\infty}$-functions on $X$ which satisfies
\begin{equation*}
\begin{cases}
\lich{\omega}  \phi'_1 = 4 \pi b_2 (\omega) - 4 \pi \bar{b}_2 +f'_1 \\
\lich{\omega} f'_1 =0 .
\end{cases}
\end{equation*}
It remains to prove that $\phi'_1$ and $f'_1$ are both $K$-invariant. We now recall that $\omega$ is invariant under $K$, and hence $\lich{\omega}$ and $b_2 (\omega)$ are both invariant under $K$. Thus, we may take the average over $K$ of the above equation as
\begin{equation} \label{averk}
\begin{cases}
\lich{\omega}  \int_K g^* \phi'_1 d \mu = 4 \pi b_2 (\omega) - 4 \pi  \bar{b}_2+ \int_K g^*f'_1 d\mu \\
\lich{\omega}\int_K g^*f'_1 d \mu =0 .
\end{cases}
\end{equation}
where $g \in K$, and $d \mu$ is the normalised Haar measure on the compact Lie group $K$. Thus, setting $\phi_1 := \int_K g^* \phi'_1 d \mu$ and $f_1 := \int_K g^* f'_1 d \mu$, we find a pair $(\phi_1 , f_1)$ of $K$-invariant functions which satisfies $\lich{\omega}  \phi_1 = 4 \pi b_2 (\omega) - 4 \pi \bar{b}_2 +f_1$ and $\lich{\omega} f_1 =0$. Note that $\phi_1$ and $f_1$ as constructed above are independent of $k$.

This means that, going back to the equation (\ref{pairkinveqphifext}), we have found a metric $\omega_{(1)} = \omega + \ai \ddbar \phi_1 /k$ and $f_1$ such that
\begin{equation*}
\begin{cases}
S(\omega) + \frac{1}{2k} (dS(\omega) , d \phi_1 )_{\omega} =  4 \pi k \bar{\rho}_k (\omega_{(1)}) + f_0 + f_1/k + O(k^{-2}) \\
\lich{\omega} (f_0 + f_1 /k) = 0 .
\end{cases}
\end{equation*}
where we recall $f_0 = - 4 \pi k c_1$. Note that, knowing that $\phi_1$ is $K$-invariant means that $\omega_{(1)}$ is $K$-invariant, and hence each coefficient of the powers of $k^{-1}$ in the above asymptotic expansion is $K$-invariant.


It is important to note that we only have $\lich{\omega} f_1 =0$ and \textit{not} $\lich{\omega_{(1)}} f_1 =0$ (cf.~Remarks \ref{deponlich} and \ref{depfk}). However, noting $\lich{\omega_{(1)}} = \lich{\omega} + \frac{1}{k} D$ with some differential operator $D$ of order at most 4 which depends only on $\omega$ and $\phi_1$, we still have
\begin{equation*}
\lich{\omega_{(1)}} (f_0 + f_1/k) = \frac{1}{k} \lich{\omega_{(1)}} f_1 = O(k^{-2})
\end{equation*}
and the main point of what we prove in the following (Proposition \ref{approxsbal} and Corollary \ref{coraprbalmsolprob}) is that this is enough for solving Problem \ref{premapprox} by an inductive argument.

Our aim now is to repeat this procedure inductively to get an improved estimate. We thus wish to find a sequence of $K$-invariant smooth functions $(\phi_{1 ,k}, \dots , \phi_{m ,k})$ such that the metric $\omega_{(m)} := \omega + \ai \ddbar ( \sum_{i=1}^m \phi_{i ,k} / k^i )$ is approximately q-ext. Unlike the cscK case, we will not be able to have each $\phi_{i,k}$ independently of $k$ (see Remark \ref{depfk} below), and we will only be able to show that each $\phi_{i,k}$ converges to some $\phi_{i, \infty}$ in $C^{\infty}$ as $k \to \infty$, if $i \ge 2$. This convergence property is obviously of crucial importance in ensuring that $\omega_{(m)}$ converges to $\omega$ as $k \to \infty$ (in $C^{\infty}$-topology), from which it also follows that we can apply Theorem \ref{bergexp} for each of $\{ \omega_{(m)} \}_m$ as they only vary within a compact subset with respect to the $C^{\infty}$-topology, when $k$ is large enough.

\begin{remark}
In what follows, we allow each coefficient $B_i = B_{i,k}$ of the asymptotic expansion to depend on $k$ \textit{as long as it converges to some $B_{i ,\infty}$ in $C^{\infty}$ as $k \to \infty$}.


\end{remark}

For notational convenience, we decide to write $\phi_{(m)} := \sum_{i=1}^m \phi_{i, k} / k^i$ for a sequence of $K$-invariant functions $( \phi_{1,k} , \dots , \phi_{m,k})$, each $\phi_{i,k}$ converging to some $\phi_{i ,\infty}$ in $C^{\infty}$, and $\omega_{(m)}:= \omega + \ai \ddbar \phi_{(m)}$. We also write $\lich{(m)}$ for the Lichnerowicz operator $\lich{\omega_{(m)}}$ with respect to $\omega_{(m)}$. Given all these remarks, the main technical result of this section can be stated as follows. 

\begin{proposition} \label{approxsbal}
Suppose that for $m \ge 1$ there exist sequences $(\phi_{1,k} , \dots , \phi_{m,k})$ and $(f_{1,k} , \dots , f_{m,k})$ of $K$-invariant real functions with the following properties: each $\phi_{i,k}$ (resp. $f_{i,k}$) converges to some $\phi_{i ,\infty}$ (resp. $f_{i, \infty}$) in $C^{\infty}$ as $k \to \infty$, and the pair $(\phi_{(m)} , f_{(m)})$, with $\phi_{(m)} = \sum_{i=1}^m \phi_{i,k} /k^i $ and $f_{(m)} = \sum_{i=1}^m f_{i,k} /k^i $ satisfies
\begin{equation*}
\begin{cases}
S(\omega) + \frac{1}{2} (dS(\omega) ,  d \phi_{(m)} )_{\omega}  =  4 \pi k  \bar{ \rho}_k (\omega_{(m)}) + f_0 + f_{(m)}  +  O(k^{-(m+1)}) \\
\lich{(m-1)}f_{(m)} =0 
\end{cases}
\end{equation*}
such that each coefficient of the powers of $k^{-1}$ in the asymptotic expansion is $K$-invariant and converges in $C^{\infty}$ as $k \to \infty$, with $f_0 = - 4\pi k c_1 = -4 \pi k(1 - \bar{b}_1/k)$ being a constant. Then we can find a pair of $K$-invariant real functions $(\phi_{m+1 ,k} , f_{m+1 ,k})$, each converging to some $(\phi_{i ,\infty} , f_{i, \infty})$ in $C^{\infty}$ as $k \to \infty$ such that the pair $\phi_{(m+1)} = \sum_{i=1}^{m+1} \phi_{i ,k} /k^i $ and $f_{(m+1)} = \sum_{i=1}^{m+1} f_{i,k} /k^i $ satisfies
\begin{equation*}
\begin{cases}
S(\omega) + \frac{1}{2} (dS(\omega) , d \phi_{(m+1)} )_{\omega}   =  4 \pi k \bar{ \rho}_k (\omega_{(m+1)}) + f_0 + f_{(m+1)}  + O(k^{-(m+2)}) \\
\lich{(m)} f_{(m+1)} = 0 
\end{cases}
\end{equation*}
such that each coefficient of the powers of $k^{-1}$ in the asymptotic expansion is $K$-invariant and converges in $C^{\infty}$ as $k \to \infty$.
\end{proposition}

\begin{remark} \label{depfk}
We note that $\phi_{i,k}$ and $f_{i,k}$ ($i \ge 2$) cannot be chosen to be independent of $k$, and can only prove the existence of families of functions $\{ \phi_{i,k} \}_k$, $\{ f_{i,k} \}_k$ converging to some smooth functions $\phi_{i, \infty}$ and $f_{i ,\infty}$ in $C^{\infty}$ as $k \to \infty$. In particular, $\phi_{i,k}$'s and $f_{i,k}$'s vary in a bounded subset of $C^{\infty}(X , \rl)$ for all large enough $k$. This is an important part of the induction hypothesis, where we also note that it was certainly satisfied in the base case $m=1$, where $\phi_1$ and $f_1$ could be chosen independent of $k$.

This is inevitable, since when we solve the equation $\lich{(m)} \phi_i = B'_i$ for some $B'_i$ as we will do in the proof, the solution $\phi_i = \phi_{i,k}$ will depend on $k$ as $\lich{(m)}$ depends on $k$ (even when we have $B'_i$ independently of $k$). 

We note that this problem did not happen in the cscK case \cite{donproj1}, where we could solve $\lich{\omega} \phi = \cst$ at each order to get an approximately balanced metric, with respect to the \textit{fixed} (cscK) metric $\omega$. This should be fundamentally related to the fact that $\bar{\partial}$ in the cscK condition $\bar{\partial} S(\omega) =0$ (or the corresponding ``quantised'' equation $\bar{\partial} \bar{ \rho}_k (\omega) =0$) is independent of the metric $\omega$, whereas $\bar{\partial} \textup{grad}_{\omega}^{1,0}$ in $\bar{\partial} \textup{grad}_{\omega}^{1,0} S(\omega) =0$ (or the corresponding $\bar{\partial} \textup{grad}_{\omega}^{1,0} \bar{\rho}_k (\omega) =0$) does depend on $\omega$.
\end{remark}

Before we start the proof, we see the consequence of it.

\begin{corollary} \label{coraprbalmsolprob}
Problem \ref{probapprox} can be solved affirmatively.
\end{corollary}


\begin{proof}
Proceeding by induction on $m$, where we recall that we have established the base case $m=1$ at the beginning of this section, we find $\omega_{(m)}$ for each $m \in \mathbb{N}$ which satisfies the properties claimed in Proposition \ref{approxsbal}. We now compute $\lich{(m)} \bar{\rho}_k (\omega_{(m)}) $. Note that $ S(\omega) + \frac{1}{2} (dS(\omega) , d \phi_{(m)} )_{\omega}$ is the Hamiltonian of the real holomorphic vector field $v_s$ with respect to $\omega_{(m)}$ so $\lich{(m)} (S(\omega) + \frac{1}{2} (dS(\omega) , d \phi_{(m)} )_{\omega})=0 $. Since $\lich{(m-1)}f_{(m)}=0$, $\lich{(m)} = \lich{(m-1)} + O(k^{-m})$, and $f_{(m)} = O(k^{-1})$, we have $\lich{(m)}f_{(m)} = O(k^{-m-1})$. This means $\lich{(m)}  \bar{\rho}_k (\omega_{(m)}) = O(k^{-m-2}) $, and $\omega_{(m)}$ is $K$-invariant since each of $(\phi_{1,k} , \dots , \phi_{m,k})$ is $K$-invariant.

Now, arguing as in the appendix of \cite{donnum}, for any $\nu \in \mathbb{N}$ there exists some $H = H_{\nu ,m} \in \mathcal{B}_k$ such that $\omega_{(m)} = \omega_{FS(H)} + O(k^{- \nu})$. Note that $\omega_{(m)}$ being $K$-invariant implies that each coefficient in the expansion of $\rho_k (\omega_{(m)})$ is $K$-invariant. Thus, using Lemma \ref{santipkinvfshilblem} in applying the argument in the appendix of \cite{donnum}, we see that $H$ is in fact $\theta(K)$-invariant, i.e.~$H \in \mathcal{B}_k^K$. Thus, taking $\nu = 2m$ for example, we have $\lich{(m)} = \lich{\omega_{FS(H)}} + O(k^{-2m})$ and $\bar{\rho}_k (\omega_{(m)}) = \bar{\rho}_k (\omega_{FS(H)}) + O(k^{-2m})$, and hence
\begin{equation} \label{approxsolm3}
\lich{\omega_{FS(H)}}  \bar{\rho}_k (\omega_{FS(H)}) = O(k^{-m-2}) .
\end{equation}
This means that, without loss of generality, we may assume in what follows that $\omega_{(m)}$ is of the form $\omega_{FS(H)}$ for some $H \in \mathcal{B}_k^K$.

\end{proof}


We now prove Proposition \ref{approxsbal}. Some technical results, which are used in the proof, about the Lichnerowicz operator are collected in the appendix.

\begin{proof}[Proof of Proposition \ref{approxsbal}]
We invoke Theorem \ref{bergexp} to have the asymptotic expansion
\begin{align*}
&4 \pi k \bar{\rho}_k (\omega_{(m+1)}) + f_0 \\
&= S(\omega_{(m+1)}) + \frac{4 \pi}{k} \left( b_2 (\omega_{(m+1)}) - \bar{b}_2 \right)  + \cdots + \frac{4 \pi}{k^{i}} \left( b_{i+1} (\omega_{(m+1)}) - \bar{b}_{i+1} \right) + \cdots + O(k^{-{(m+2)}}) 
\end{align*}
where $i \le m+1$, which is valid as long as each of $\phi_{1,k} , \dots , \phi_{m+1 ,k}$ varies in a bounded subset of $C^{\infty} (X , \rl)$, as ensured by the induction hypothesis. We then expand each coefficient $S(\omega_{(m+1)})$ and $b_i (\omega_{(m+1)})$ ($2 \le i \le k+1$) in $\phi_{m+1 ,k}/k^{m+1}$. They are of the form 
\begin{equation*}
S(\omega_{(m+1)}) = S(\omega_{(m)}) + k^{-(m+1)} \mathbb{L}_{\omega_{(m)}} \phi_{m+1 ,k}  + O(k^{-2(m+1)})
\end{equation*}
and
\begin{equation*}
b_i (\omega_{(m+1)}) = b_i (\omega_{(m)}) + O(k^{-(m+1)}) .
\end{equation*}
Since $\omega$ and $(\phi_{1,k} , \dots , \phi_{m,k})$ are $K$-invariant, the same argument as in Remark \ref{kinvtwo} implies that each coefficient of the powers of $k^{-1}$ in the above expansions is $K$-invariant and converges to some smooth function in $C^{\infty}$ as $k \to \infty$ once we know that $\phi_{m+1,k}$ is $K$-invariant and converges to some $\phi_{m+1 ,\infty}$ in $C^{\infty}$ as $k \to \infty$.

We can thus write
\begin{align}
&4 \pi k \bar{\rho}_k (\omega_{(m+1)}) + f_0 \notag = S(\omega_{(m)}) + \frac{1}{k^{m+1}}\mathbb{L}_{\omega_{(m)}} \phi_{m+1,k} + \frac{4 \pi}{k} \left( b_2 (\omega_{(m)}) - \bar{b}_2 \right)+ \cdots  \notag \\
& \ \ \ \ \ \ \ \ \ \ \ \ \ \ \ \ \ \ \ \ \ \ \ \ \ \ \ \ \ \ \ \ \ \ \ \ \ \ \ \ \ \ \ \ \ \ \ \ \ \ + \frac{4 \pi}{k^{i}} \left( b_{i+1} (\omega_{(m)}) - \bar{b}_{i+1} \right)+ \cdots + O(k^{-{(m+2)}}) \notag \\
&\ \ \ \ \ \ \ \ \ \ \ \ \ \ \ \ \ \ \ \ =4 \pi k \bar{\rho}_k (\omega_{(m)}) + f_0 + \frac{1}{k^{m+1}} \mathbb{L}_{\omega_{(m)}} \phi_{m+1,k} +O(k^{-(m+2)}) . \label{rhom1}
\end{align}

By the induction hypothesis,
\begin{equation*}
\left( S(\omega) + \frac{1}{2} (dS(\omega) , d \phi_{(m)} )_{\omega} \right) =  4 \pi k  \bar{ \rho}_k (\omega_{(m)})  + f_0 +  f_{(m)}  + O(k^{-(m+1)}) ,
\end{equation*}
and there exists a family of $K$-invariant functions $\{ B_{m+1 ,k} \}_k$, converging to some $K$-invariant function $B_{m+1 , \infty}$ in $C^{\infty}$ as $k \to \infty$, such that
\begin{equation}
\left( S(\omega) + \frac{1}{2} (dS(\omega) , d \phi_{(m)} )_{\omega} \right) =  4 \pi k  \bar{\rho}_k (\omega_{(m)})  + f_0 +  f_{(m)} + k^{-(m+1)} B_{m+1,k} + O(k^{-(m+2)}) . \label{indhypm}
\end{equation}
Thus, combining (\ref{rhom1}) and (\ref{indhypm}),
\begin{align*}
&4 \pi k \bar{\rho}_k (\omega_{(m+1)}) \\
&= \left( S(\omega) + \frac{1}{2} (dS(\omega) , d \phi_{(m)} )_{\omega} \right) -  f_{(m)} + \frac{1}{k^{m+1}} ( \mathbb{L}_{\omega_{(m)}} \phi_{m+1,k} - B_{m+1,k}) +O(k^{-(m+2)}) \\
&= \left( S(\omega) + \frac{1}{2} (dS(\omega) , d \phi_{(m+1)} )_{\omega} \right) -  f_{(m)} \\
&\ \ \ \ \ \ \ \ \ \ \ \ \ \ \ \ \ + \frac{1}{k^{m+1}} ( \mathbb{L}_{\omega_{(m)}} \phi_{m+1,k} -\frac{1}{2} (dS(\omega) , d \phi_{m+1,k} )_{\omega} - B_{m+1 ,k}) +O(k^{-(m+2)}) .
\end{align*}

Note that, since the induction hypothesis $\phi_{(m)} = \sum_{i=1}^m \phi_{i,k} /k^i =O(1/k)$ implies $(d  \mathbb{L}_{\omega} \phi_{(m)} ,  d \phi_{m+1,k} )_{\omega} = O(1/k)$ and $(dS(\omega_{(m)}) , d \phi_{m+1,k} )_{\omega_{(m)}} = ( dS(\omega_{(m)}) , d \phi_{m+1,k} )_{\omega} + O(1/k) $, 
we have
\begin{align*}
&\mathbb{L}_{\omega_{(m)}} \phi_{m+1,k} -\frac{1}{2} (dS(\omega) , d \phi_{m+1,k} )_{\omega} \\
&=  \mathbb{L}_{\omega_{(m)}} \phi_{m+1,k} -\frac{1}{2} (dS(\omega_{(m)}) , d \phi_{m+1, k} )_{\omega} +\frac{1}{2}(d  \mathbb{L}_{\omega} \phi_{(m)} ,  d \phi_{m+1,k} )_{\omega} + O(k^{-2}) \\
&= \mathbb{L}_{\omega_{(m)}} \phi_{m+1,k} -\frac{1}{2} (dS(\omega_{(m)}) , d \phi_{m+1,k} )_{\omega_{(m)}} +O(k^{-1}) \\
&= - \lich{(m)} \phi_{m+1,k} + O(k^{-1}) .
\end{align*}
Thus
\begin{align}
4 \pi k \bar{\rho}_k (\omega_{(m+1)}) &= \left( S(\omega) + \frac{1}{2} (dS(\omega) , d \phi_{(m+1)} )_{\omega} \right) -  f_{(m)} \notag \\
&\ \ \ \ \ \ \ \ \ \ + \frac{1}{k^{m+1}} (- \lich{(m)} \phi_{m+1,k}  - B_{m+1,k}) +O(k^{-(m+2)}) . \label{stepm1}
\end{align}

Observe that, in all the expansions above, each coefficient of the powers of $k^{-1}$ is $K$-invariant and converges to some smooth function in $C^{\infty}$ as $k \to \infty$ once we know that $\phi_{m+1,k}$ is $K$-invariant and converges to some $\phi_{m+1 ,\infty}$ in $C^{\infty}$ as $k \to \infty$.

Our aim now is to find $K$-invariant functions $\phi_{m+1,k}$ and $f_{m+1,k}$ such that
\begin{equation*}
\begin{cases}
\lich{(m)} \phi_{m+1,k}  + B_{m+1,k} = f_{m+1,k} \\
\lich{(m)} f_{(m+1)} =0 .
\end{cases}
\end{equation*}

Recall first that we have $\lich{(m-1)}f_{(m)}= 0$ by the induction hypothesis. Since $\lich{(m)} = \lich{(m-1)} + \frac{1}{k^m}D$ with some differential operator $D$ of order at most 4 which depends only on $\omega$ and $(\phi_{1,k} , \dots , \phi_{m,k})$, and recalling $f_{(m)} = \sum_{i=1}^m f_{i,k} /k^i = O(1/k)$, we have $\lich{(m)}f_{(m)}=O(k^{-(m+1)})$. We first aim to find $\{ f'_{m+1,k} \}_k$ which satisfies $\lich{(m)} (f_{(m)} + f'_{m+1,k}/ k^{m+1}) =0$ and also converges to a smooth function in $C^{\infty}$ as $k \to \infty$.

Let $F_k := k^{m+1} (f_{(m)} - \text{pr}_{(m)} f_{(m)} )$ where $\text{pr}_{(m)} : C^{\infty} (X, \rl) \surj \ker \lich{(m)}$ is the projection onto the kernel of $\lich{(m)}$ in terms of the $L^2$-orthogonal direct sum decomposition $C^{\infty} (X , \rl) = \ker \lich{(m)} \oplus \ker \lich{(m)}^{\perp}$. We write $G_k:= \lich{(m)} F_k$. Since $f_{i,k}$ ($1 \le i \le m$) converges to a smooth function in $C^{\infty}$ as $k \to \infty$, $\omega_{(m)} \to \omega$ in $C^{\infty}$ as $k \to \infty$, and $\lich{(m)} f_{(m)} = O(k^{-m-1})$, $G_k$ converges to a smooth function, say $G_{\infty}$, in $C^{\infty}$ as $k \to \infty$. Now we observe that $F_k$ is the solution to the equation $\lich{(m)} F_k = G_k$ with the minimum $L^2$-norm.

We now aim to show that $F_k$ converges in $C^{\infty}$ as $k \to \infty$. We aim to use Lemma \ref{lemconv} in the end, but we first have to establish $G_{\infty} \in \textup{im} \lich{\omega}$ to do so. By using the $L^2$-orthogonal direct sum decomposition $C^{\infty} (X , \rl) = \ker{\lich{(m)}} \oplus \ker{\lich{(m)}}^{\perp}$ and recalling the standard elliptic regularity, for each $p \in \mathbb{N}$ there exists a constant $C_p (\omega , \{ \phi_{i,k} \})$ which depends only on $\omega$ and $\phi_{i,k}$ ($1 \le i \le m$) such that $|| F_k ||_{L^2_{p+4}} < C_p (\omega , \{ \phi_{i,k} \} ) ||G_k||_{L^2_{p}} $, with the Sobolev norm $||\cdot||_{L^2_p}$ for a fixed but large enough $p$. By noting that each $\phi_i$ converges in $C^{\infty}$ as $k \to \infty$ and $G_k$ converges in $C^{\infty}$ as $k \to \infty$, we can find a uniform bound on $||F_k||_{L^2_{p +1}}$ as $|| F_k ||_{L^2_{p+4}} < C'_p(\omega , \{ \phi_{i, \infty} \}) ||G_{\infty}||_{L^2_{p}} $ which is independent of $k$. Thus there exists a subsequence $\{ F_{k_l} \}$ of $\{ F_k \}$ which converges in the Sobolev space $L^2_{p+3}$ by Rellich compactness. Let $F_{\infty}$ be its limit in $L^2_{p+3}$. Now, recalling $\lich{(m)} F_{k_l} = G_{k_l}$, consider
\begin{equation*}
\lich{\omega} F_{k_l} - G_{k_l} = \lich{(m)} F_{k_l} - G_{k_l} - \frac{1}{{k_l}}D(F_{k_l}) = - \frac{1}{{k_l}} D(F_{k_l}) 
\end{equation*}
 where we wrote $\lich{(m)} = \lich{\omega} + \frac{1}{{k_l}}D$ with some differential operator $D$ of order at most 4, which depends only on $\omega$ and $\phi_{i,{k_l}} $ ($1 \le i \le m$). Thus, since each $\phi_{i,{k_l}}$ ($1 \le i \le m$) converges in $C^{\infty}$ by the induction hypothesis, we have
 \begin{equation*}
|| \lich{\omega} F_{k_l} - G_{k_l} ||_{L^2_{p-1}} \le   \frac{C''_p (\omega ,   \{ \phi_{i, \infty} \})}{{k_l}} || F_{k_l} ||_{L^2_{p-5}} , 
\end{equation*}
and hence by taking the limit ${k_l} \to \infty$, we have the equation $\lich{\omega} F_{\infty} = G_{\infty}$ in $L^2_{p-1}$. Since  $G_{\infty} \in C^{\infty} (X , \rl)$ we have $F_{\infty} \in C^{\infty} (X , \rl)$ by elliptic regularity, and hence $\lich{\omega} F_{\infty} = G_{\infty}$ in $C^{\infty} (X , \rl)$ shows $G_{\infty} \in \textup{im} \lich{\omega}$. Lemma \ref{projconv} shows $\mathrm{pr}_{\omega} F_{\infty} =0$, and hence $F_{\infty}$ is the $L^2$-minimum solution to the equation $\lich{\omega} F_{\infty} = G_{\infty}$. We can now apply Lemma \ref{lemconv} to conclude that $\{ F_k \}$ converges to $F_{\infty}$ in $C^{\infty}$ (we note that the convergence holds for the whole sequence $\{ F_k \}$ and not just for the subsequence $\{ F_{k_l} \}$ that we used).

Now setting $f'_{m+1,k}: = -F_k = -k^{m+1} (f_{(m)} - \mathrm{pr}_{(m)} f_{(m)})$, we have $\lich{(m)} (f_{(m)} + f'_{m+1 ,k}/k^{m+1}) =0$ and that $f'_{m+1,k}$ converges in $C^{\infty}$ as $k \to \infty$. Also, we note that $f'_{m+1,k}$ is $K$-invariant because $\lich{(m)}$ and $f_{(m)}$ are both $K$-invariant.




For this choice of $f'_{m+1,k}$, we solve 
\begin{equation*}
\lich{(m)} \phi_{m+1,k}  + B_{m+1,k} =  f'_{m+1,k}
\end{equation*}
modulo some function $f''_{m+1,k}$ with $\lich{(m)} f''_{m+1,k}=0$, i.e. we solve for $\phi_{m+1,k}$ the equation
\begin{equation*}
\lich{(m)} \phi_{m+1,k}  + B_{m+1,k} =  f'_{m+1,k} + f''_{m+1,k}
\end{equation*}
for some $f''_{m+1,k}$ with $\lich{(m)} f''_{m+1,k}=0$. This is possible by Lemma \ref{lich1}, where we also recall that $f''_{m+1,k}$ is in fact $-\text{pr}_{(m)} (f'_{m+1,k}-B_{m+1,k})$, where $\textup{pr}_{(m)}: C^{\infty} (X , \rl) \surj \ker \lich{(m)}$. Thus we have
\begin{equation*}
\lich{(m)} \phi_{m+1,k} = (f'_{m+1,k} - B_{m+1,k}) - \text{pr}_{(m)} (f'_{m+1,k}-B_{m+1,k}) .
\end{equation*}
$f'_{m+1,k}$ converges in $C^{\infty}$ as we saw above, and $B_{m+1,k}$ converges in $C^{\infty}$ since $\phi_{i,k}$ ($1 \le i \le m$) converges in $C^{\infty}$ and by the induction hypothesis. Thus, by Lemma \ref{projconv}, the right hand side of the above equation is a smooth function, with parameter $k$, that converges, say to $G_{\infty}$, in $C^{\infty}$ as $k \to \infty$, which is in the image of $\lich{\omega}$. Thus, writing $\phi_{m+1 ,\infty}$ for the solution to the equation $\lich{\omega} \phi_{m+1, \infty} = G_{\infty}$, we can apply Lemma \ref{lemconv} for $G_k : =  (f'_{m+1 ,k} - B_{m+1,k}) - \text{pr}_{(m)} (f'_{m+1,k}-B_{m+1,k})$ to conclude that $\phi_{m+1,k}$ converges to $\phi_{m+1, \infty}$ in $C^{\infty}$ (since without loss of generality we may assume that they are all $L^2$-minimum solutions).

Noting that $f''_{m+1,k}:= -\text{pr}_{(m)} (f'_{m+1,k}-B_{m+1,k})$ is $K$-invariant since $\lich{(m)}$, $B_{m+1,k}$, and $f'_{m+1,k}$ are all $K$-invariant, we can take the average over $K$ of $\phi_{m+1,k}$ as we did in (\ref{averk}). Thus $\phi_{m+1,k}$ can be chosen to be $K$-invariant. 

We then note that $\lich{(m)} (f_{(m)} + (f'_{m+1,k} + f''_{m+1,k})/k^{m+1}) =0$ so we define $f_{m+1,k} := f'_{m+1,k} + f''_{m+1,k}$ which is $K$-invariant and converges to some smooth function as $k \to \infty$. For this choice of $f_{m+1,k}$, we thus have two $K$-invariant functions $\phi_{m+1,k}$ and $f_{m+1,k}$ with $\lich{(m)} f_{(m+1)} =0$ and $ \lich{(m)} \phi_{m+1,k}  + B_{m+1,k} =  f_{m+1,k}$ each converging to a smooth function in $C^{\infty}$ as $k \to \infty$.

Now, going back to the equation (\ref{stepm1}), we find
\begin{align*}
4 \pi k \bar{\rho}_k (\omega_{(m+1)}) &= \left( S(\omega) + \frac{1}{2} (dS(\omega) , d \phi_{(m+1)} )_{\omega} \right) -  f_{(m)} - \frac{1}{k^{m+1}} f_{m+1,k} +O(k^{-(m+2)}) \\
&= \left( S(\omega) + \frac{1}{2} (dS(\omega) , d \phi_{(m+1)} )_{\omega} \right) -  f_{(m+1)}  +O(k^{-(m+2)}) 
\end{align*}
with $f_{(m+1)}$ satisfying $\lich{(m)} f_{(m+1)} =0$. As remarked just after the equation (\ref{stepm1}), each coefficient of the powers of $k^{-1}$ in the above asymptotic expansion is $K$-invariant and converges to some smooth function in $C^{\infty}$ as $k \to \infty$ since $\phi_{m+1,k}$ is $K$-invariant and converges to some $\phi_{m+1 ,\infty}$ in $C^{\infty}$ as $k \to \infty$. Since $\phi_{m+1,k}$ and $f_{m+1,k}$ are both $K$-invariant functions that converge to smooth functions in $C^{\infty}$ as $k \to \infty$, the sequences $(\phi_{1,k} , \dots , \phi_{m+1,k})$ and $(f_{1,k} , \dots , f_{m+1 ,k} )$ satisfy all the requirements stated in the induction hypothesis in Proposition \ref{approxsbal} for the induction to continue.

\end{proof}

It is immediate that for each $l \in \mathbb{N}$ there exists $C_{l,m}>0$ so that $|| \omega_{(m)} - \omega ||_{C^l , \omega} < C_{l,m} /k$. The equation (\ref{approxsolm3}) and the standard elliptic regularity mean that we can find $F_{m,k} \in C^{\infty}(X ,\rl)$, for each $k$, such that
\begin{equation*} 
\lich{(m)} (4 \pi k \bar{\rho}_k (\omega_{(m)}) +F_{m,k} /k^{m+1}) =0 
\end{equation*}
and that for each $p \in \mathbb{N}$ the $L^2_p$-norm of $\{ F_{m,k} \}_k$ is bounded uniformly of $k$. In particular, observe that $\sup_X |F_{m,k}|$ is bounded uniformly of $k$. Moreover, since $\omega_{(m)}$ and $\lich{(m)}$ are both $K$-invariant, we may choose $F_{m,k}$ to be $K$-invariant, as we did in ($\ref{averk}$). This means that the vector field $v_{(m)}$ defined by
\begin{equation} \label{eqapprbalwefk}
\iota (v_{(m)}) \omega_{(m)} = - d \left( \bar{\rho}_k (\omega_{(m)}) + \frac{F_{m,k}}{4 \pi k^{m+2}} \right)
\end{equation}
is real holomorphic and lies in the centre $\mathfrak{z}$ of $\mathfrak{k}$ by Lemmas \ref{lemrbgfgenvfcentre} (recall also that Lemma \ref{lemhol} shows that $\bar{\rho}_k (\omega_h)$ is indeed $K$-invariant if $K \le \mathrm{Isom} (\omega_h)$).

\begin{remark} \label{roeovfgbbrik}
Note further that $d \left( \bar{\rho}_k (\omega_{(m)}) \right)= \frac{1}{4 \pi k} dS(\omega_{(m)}) + O(k^{-2}) = \frac{1}{4 \pi k} dS(\omega) + O(k^{-2})$, and that $F_{m,k}$ is of order 1 in $k^{-1}$ imply that $4 \pi k v_{(m)}$ converges to the extremal vector field $v_s$ generated by $S(\omega)$. In particular, if we use the (pointwise) norm $| \cdot |_{k \omega}$ on $TX$ defined by $ k \omega$, we have $\sup_X | Jv_{(m)} |^2_{k \omega} \le \cst /k$. This fact will be important in Lemma \ref{lemctlaps}.
\end{remark}


\section{Reduction to a finite dimensional problem} \label{redtalalgprobsec}

Recall that the equation $\bar{\partial} \rho_k (\omega_h) = 0$ (or equivalently $\rho_k (\omega_h) = \cst$) is equivalent to finding a balanced embedding, i.e.~the embedding where $\bar{\mu}_X (g)$ is a constant multiple of the identity (cf.~Theorem \ref{balembbalmetequiv}). This means that the seemingly intractable PDE problem $\bar{\partial} \rho_k (\omega_h)=0$ can be reduced to a finite dimensional problem of finding the balanced embedding. The main result (Corollary \ref{crdlaprbal}, see also Proposition \ref{rbalmiffinvocomtgtiox}) of this section is to establish this reduction in the setting where $\mathrm{Aut}_0 (X,L)$ is nontrivial, namely to establish a connection between the equation $\lich{\omega_h} \rho_k (\omega_h) =0$ and the projective embedding in terms of the centre of mass $\bar{\mu}_X $.

In what follows, we shall be mostly focused on the \kah metrics of the form $\omega_{FS(H)}$ with $H \in \mathcal{B}_k$ or $H \in \mathcal{B}_k^K $. To simplify the notation, we will often write as follows.

\begin{notation}
We will often write $\omega_H$ for $\omega_{FS(H)}$, and $\lich{H}$ for $\lich{\omega_{FS(H)}}$. 
\end{notation}

\subsection{General lemmas and their consequences}

We start with the following general lemmas.

\begin{lemma} \label{lemactautmetfs}
\emph{(cf.~\cite[(5.3)]{pssurv})}
For any $f \in \textup{Aut}_0 (X , L)$ and any $H \in \mathcal{B}_k$, we have
\begin{equation} \label{chfsactsigma}
f^* \omega_{H} = \omega_{H} + \frac{\ai}{2 \pi k} \ddbar \log (\sum_i | \sum_j \theta(f)_{ij} s_j |^2_{FS(H)^k} ) ,
\end{equation}
where $\{ s_j \}$ is an $H$-orthonormal basis for $H^0 (X , L^k)$.
\end{lemma}

\begin{proof}
Suppose that we write (at first) $\{ Z'_i \}$ for an $H$-orthonormal basis for $H^0(X , L^k)$, giving an isomorphism $H^0(X , L^k) \cong \cx^N$ and hence defining an embedding $\iota' : X \inj \prj^{N-1}$. By recalling $\theta (f) \circ \iota' = \iota' \circ f$ (the equation (\ref{liftauttheta})), we have 
\begin{align*}
f^* \omega_{H} &= f^* \iota'^* \frac{\ai}{2 \pi k} \ddbar \log \left( \sum_i |Z'_i|^2 \right)\\
&= \iota'^*  \frac{\ai}{2 \pi k} \ddbar \log \left( \sum_i |\sum_j \theta (f)_{ij} Z'_j|^2 \right)
\end{align*}
where $\theta (f)_{ij}$ is the matrix for $\theta (f)$ represented with respect to $\{ Z'_i \}$. We can write the above as (cf.~\cite[(5.3)]{pssurv})
\begin{equation*}
f^* \omega_{H} = \omega_{H} + \iota'^*  \frac{\ai}{2 \pi k} \ddbar \log \left( \frac{\sum_i |\sum_j \theta (f)_{ij} Z'_j|^2}{\sum_i |Z'_i|^2} \right),
\end{equation*}
where we note that $\frac{\sum_i |\sum_j \theta (f)_{ij} Z'_j|^2}{\sum_i |Z'_i|^2}$ is a well-defined function on $\prj^{N-1}$. Pick now any hermitian metric $\tilde{h}$ on $\mathcal{O}_{\prj^{N-1}} (1)$. We now observe that, by choosing a local trivialisation of $\mathcal{O}_{\prj^{N-1}} (1)$ and writing $\tilde{h} = e^{- \phi}$ locally, multiplying both the denominator and the numerator by $e^{ - \phi}$ yields
\begin{equation*}
\frac{\sum_i |\sum_j \theta (f)_{ij} Z'_j|^2}{\sum_i |Z'_i|^2} = \frac{\sum_i |\sum_j \theta (f)_{ij} Z'_j|_{\tilde{h}}^2}{\sum_i |Z'_i|_{\tilde{h}}^2} ,
\end{equation*}
by noting that any ambiguity in choosing the local trivialisation in the denominator is cancelled by the one in the numerator. Thus, choosing $\tilde{h}$ to be the hermitian metric $\tilde{h}_{FS(H)}$ on $\mathcal{O}_{\prj^{N-1}} (1)$ induced from $H$ (so that $\iota'^* \tilde{h}_{FS(H)} = h^k_{FS(H)}$) and writing $s_i := \iota'^* Z'_i$, we have 
\begin{equation*}
\iota'^* \ai \ddbar \log \left( \frac{\sum_i |\sum_j \theta (f)_{ij} Z'_j|^2}{\sum_i |Z'_i|^2} \right) = \ai \ddbar \log \left( \frac{\sum_i |\sum_j \theta (f)_{ij} s_j|_{FS(H)^k}^2}{\sum_i |s_i|_{FS(H)^k}^2} \right).
\end{equation*}
On the other hand, $\sum_i |s_i|_{FS(H)^k}^2$ is constantly equal to 1 since $\{ s_i \}$ is an $H$-orthonormal basis, by the definition (\ref{defoffseq}) of $FS(H)$. Thus we finally have (\ref{chfsactsigma}).
\end{proof}


\begin{lemma} \label{lemhamlinalg}
Suppose that $\psi$ is a Hamiltonian of the Killing vector field $v \in \mathfrak{k}$ with respect to $\omega_{H}$, $H \in \mathcal{B}_k^K$, so that we have $\iota (v) \omega_H = -d \psi$ and $\lich{H} \psi=0 $.

Suppose also that we write (cf.~(\ref{vinkjvinaik})) $A:= \theta_* (Jv) \in \theta_* (\ai \mathfrak{k})$ for the real holomorphic vector field $ Jv$ and the (injective) Lie algebra homomorphism $\theta_* : \lieg = \textup{LieAut}_0 (X , L) \to \mathfrak{sl} (N , \cx)$. Then, writing $\{ s_i \}$ for an $H$-orthonormal basis for $H^0(X , L^k)$, we have
\begin{equation} \label{psiecstplaijsisj}
\psi = - \frac{1}{2 \pi k} \sum_{i,j} A_{ij} h^k_{FS(H)} (s_i , s_j) + \cst
\end{equation}
where $A_{ij}$ is the matrix for $A$ represented with respect to $\{ s_i \}$.
\end{lemma}


\begin{proof}
Take the 1-parameter subgroup $\{ \sigma(t) \} \le \mathrm{Aut}_0 (X,L)$ generated by $Jv$. Then, by Lemma \ref{lemactautmetfs}, we have
\begin{equation*}
\sigma(t)^* \omega_{H} = \omega_{H} + \frac{\ai}{2 \pi k} \ddbar \log \left( \sum_i | \sum_j \theta(\sigma(t))_{ij} s_i|^2_{FS(H)^k} \right)
\end{equation*}
for an $H$-orthonormal basis $\{ s_i \}$. We observe $\theta (\sigma (t)) = e^{tA}$ by the definition $A = \theta_* (Jv)$. We now see
\begin{align*}
L_{Jv} \omega_{H} &= \lim_{t \to 0} \frac{\sigma (t)^* \omega_{FS(H)} - \omega_{FS(H)}}{t} \\
&= \lim_{t \to 0} \frac{\ai}{2 \pi k} \frac{\ddbar \log \left( \sum_i | \sum_j \theta(\sigma(t))_{ij} s_j|^2_{FS(H)^k} \right)}{t} \\
&=  \frac{\ai}{2 \pi k} \ddbar \left( \left. \frac{d}{dt} \right|_{t=0} \log (\sum_i | \sum_j (e^{tA})_{ij} s_j|^2_{FS(H)^k}) \right)\\
&= \frac{\ai}{ \pi k} \ddbar \left( \sum_{i, j} A_{ij} h^k_{FS(H)}(s_i,s_j) \right) ,
\end{align*}
by noting that $A$ is hermitian since $H$ is $\theta(K)$-invariant (cf.~Lemma \ref{rmonadjunilem}).

Note on the other hand that, since $\psi$ is the Hamiltonian for $v$, we have, by using the Cartan homotopy formula,
\begin{equation*}
L_{Jv} \omega_{H}=d \iota(Jv) \omega_{H} = dJ \iota(v) \omega_{H} =-  dJd \psi = -  2  \ai  \ddbar \psi   
\end{equation*}
where we used (\ref{actionofjon1forms}) in the second equality. We thus have $\psi = - \frac{1}{2 \pi k} \sum_{i,j} A_{ij} h^k_{FS(H)} (s_i , s_j) + \cst$ as claimed.

\end{proof}

\begin{remark} \label{gaiaikpsiinkolichop}
Conversely, given $A \in \theta_* (\ai \mathfrak{k})$, it immediately follows that $\psi$ as defined in (\ref{psiecstplaijsisj}) satisfies $\lich{H} \psi =0$, generating a real holomorphic vector field $v:= J^{-1} \theta_*^{-1} (A)$.
\end{remark}

Suppose now that $H \in \mathcal{B}^K_k$ satisfies $ \lich{H} \bar{\rho}_k (\omega_{H}) =0$. Then Lemma \ref{lemhamlinalg} and (\ref{defoffseq}) implies
\begin{equation} \label{eqbrkhwccnca}
\bar{\rho}_k (\omega_{H}) = C - \frac{1}{2 \pi k} \sum_{i,j} A_{ij} h^k_{FS(H)} (s_i , s_j) = \sum_{i,j} \left( CI - \frac{1}{2 \pi k} A\right)_{ij} h^k_{FS(H)} (s_i , s_j)
\end{equation}
for $A:= \theta_* (- \textup{grad} \bar{\rho}_k (\omega_{H})  ) \in \theta_* (\ai \mathfrak{k})$ and some constant $C \in \rl$ which can be determined by integrating both sides of the equation, so that the average over $X$ of both sides is 1. We now have the following proposition, pointed out to the author by Joel Fine, which distills the essential point of our main result Corollary \ref{crdlaprbal} to be proved later.

\begin{proposition} \label{rbalmiffinvocomtgtiox}
The equation $ \lich{H} \bar{\rho}_k (\omega_{H}) =0$ (or equivalently $\bar{\partial} \mathrm{grad}_{\omega_H}^{1,0} \rho_k (\omega_H) =0$) holds if and only if $\bar{\mu}_X (g)^{-1}$ generates a holomorphic vector field on $\prj^{N-1}$ that is tangential to $\iota (X) \subset \prj^{N-1}$, where $H = \overline{(g^{-1})^t} g^{-1}$.
\end{proposition}

\begin{proof}
Lemma \ref{lemhamlinalg} and Remark \ref{gaiaikpsiinkolichop} imply that $\bar{\partial} \mathrm{grad}_{\omega_H}^{1,0} \rho_k (\omega_H) =0$ is satisfied if and only if $\rho_k (\omega_H) = \frac{N}{V} \left( CI - \frac{1}{2 \pi k} A \right)_{ij} h^k_{FS(H)} (s_i, s_j)$, where $C \in \rl$ is some constant and $A = \theta_* (- \mathrm{grad} \rho_k (\omega_H)) \in \theta_* (\ai \mathfrak{k})$. Combined with (\ref{lbergfnitoham1}), we see that $\bar{\partial} \mathrm{grad}_{\omega_H}^{1,0} \rho_k (\omega_H) =0$ holds if and only if
\begin{equation*}
\sum_{i,j} \left( k^n\bar{\mu}_X (g)^{-1} - \frac{N}{V} \left( CI - \frac{1}{2 \pi k} A \right) \right)_{ij}  h^k_{FS(H)} (s_i, s_j) =0,
\end{equation*}
which holds if and only if $\bar{\mu}_X (g)^{-1} = \frac{N}{Vk^n} \left( CI - \frac{1}{2 \pi k} A \right)$ by applying Lemma \ref{lemfsinj} (see also its proof given in \cite[Lemma 3.1]{yhhilb}).

We are thus reduced to proving the following: if $\bar{\mu}_X (g)^{-1}$ generates a holomorphic vector field on $\prj^{N-1}$ that is tangential to $\iota (X) \subset \prj^{N-1}$, then its trace-free part $(\bar{\mu}_X (g)^{-1})_0$ must lie in $\theta_* (\ai \mathfrak{k})$. Observe first that by Lemma \ref{lemdefofthtosl}, $\bar{\mu}_X (g)^{-1}$ generates a holomorphic vector field on $\prj^{N-1}$ that is tangential to $\iota (X) \subset \prj^{N-1}$ if and only if $(\bar{\mu}_X (g)^{-1})_0 \in \theta_* (\lieg)$. Write $\lieg = ( \mathfrak{k} \oplus \ai \mathfrak{k} ) \oplus_{\pi} \mathfrak{n}$, where $\mathfrak{n}: = \mathrm{Lie}(R_u)$ is a nilpotent Lie algebra and $\oplus_{\pi}$ is the semidirect product in the Lie algebra corresponding to $G = K^{\cx} \ltimes R_u$ (cf.~Notation \ref{notgrautliealg}). Then, noting that $\bar{\mu}_X (g)^{-1}$ is hermitian, the $\mathfrak{n}$-component of $(\bar{\mu}_X (g)^{-1})_0$ must be zero by the Jordan--Chevalley decomposition, and Lemma \ref{rmonadjunilem} implies that the trace-free part of $\bar{\mu}_X (g)^{-1}$ must lie in $\theta_* (\ai \mathfrak{k})$.


\end{proof}

Note that the above proof implies that $ CI - \frac{1}{2 \pi k} A  $ is always positive definite, and in particular invertible. However, for the later argument (cf.~Remark \ref{ronposcma}), it will be necessary to have more precise estimates on the operator norm $|| A||_{op}$ of $A$ (i.e.~the maximum of the moduli of the eigenvalues of $A$) and $|C|$. In particular, we shall need to focus on the case where $||A||_{op}$ is bounded uniformly of $k$. First of all, we see that $|C|$ can be bounded in terms of $||A||_{op}$ as follows. Note from (\ref{defoffseq}) that we have
\begin{equation} \label{bdofaitfopnoa}
\left| C - \frac{1}{2 \pi k} ||A||_{op} \right| \le \sum_{i,j} \left( CI - \frac{1}{2 \pi k} A\right)_{ij} h^k_{FS(H)} (s_i , s_j) \le C + \frac{1}{2 \pi k} ||A||_{op} .
\end{equation}
By assuming that $||A||_{op}$ is bounded uniformly for all large enough $k$, we have $C - \frac{1}{2 \pi k} ||A||_{op} >0$ for all large enough $k$. We now take the average of (\ref{bdofaitfopnoa}) over $X$ with respect to $\omega_H$ to get $1 - \frac{1}{2 \pi k} ||A||_{op } \le C \le 1 + \frac{1}{2 \pi k} ||A||_{op }$, and get the following.


\begin{proposition} \label{proprbalitomatrices}
Suppose now that $H \in \mathcal{B}^K_k$ satisfies $ \lich{H} \bar{\rho}_k (\omega_{H}) =0$ and the operator norm of $A:= \theta_* (- \textup{grad} \bar{\rho}_k (\omega_{H})  )$ is bounded uniformly of $k$. Then we can write
\begin{equation*}
\bar{\rho}_k (\omega_{H}) = \sum_{i,j} \left( I + C_A I - \frac{1}{2 \pi k} A\right)_{ij} h^k_{FS(H)} (s_i , s_j),
\end{equation*}
where $C_A$ is a constant which satisfies $- \frac{1}{2 \pi k} ||A||_{op} \le C_A \le  \frac{1}{2 \pi k} ||A||_{op }$, and hence is of order $1/k$, in particular.
\end{proposition}

\begin{remark} \label{ronposcma}
The uniform bound for $||A||_{op}$ will be crucially important in \S \ref{apsoldbgitotcom} and \S \ref{qexmsgf}. In what follows, we shall discuss some sufficient conditions under which we can assume the bound $||A||_{op } < \cst$ uniformly of $k$. It turns out that these conditions are always satisfied for our purpose (cf.~Corollary \ref{corbdonaapprrbal} and (\ref{unesommaprbfexmbrp})).



\end{remark}

We now discuss the operator norm of $A$. In what follows, we occasionally write $H(k) \in \mathcal{B}^K_k$ for $H \in \mathcal{B}^K_k$, just in order to make clear its dependence on $k$. 

\begin{lemma} \label{lemctlaps}
Suppose that we have a real holomorphic vector field $v \in \ai \mathfrak{k}$ on $X$ and a sequence $\{ H(k) \}_k$ with $H(k) \in \mathcal{B}^K_k$, which satisfy $|v|^2_{ k \omega_{H(k)}} = O(1/k) $, where $| \cdot |_{k \omega_{H(k)}}$ is a pointwise norm on $TX$ defined by $k \omega_{H(k)}$.
Then $|| \theta_* (v) ||_{op} \le \cst$ uniformly for all large enough $k$.
\end{lemma}

\begin{proof}
Recall now the equation (\ref{liftauttheta}) so that we can write
\begin{equation} \label{expdconiotav}
e^{\theta_* (v)} \circ \iota = \iota \circ e^{v} .
\end{equation}
Since $\iota$ is an isometry if we choose the metrics $k \omega_{H(k)}$ on $X$ and $H(k)$ on $\cx^N \cong H^0 (X,L^k)$ covering $\prj^{N-1}$, the assumption $|v|^2_{k \omega_{ H(k)}} = O(1/k)$ implies $| \iota_* \circ v |^2_{H(k)} = |v|^2_{k \omega_{H(k)}} = O(1/k)$, where $| \cdot |_{k \omega_{H(k)}}$ (resp. $| \cdot |_{H(k)}$) is a pointwise norm on $TX$ given by $k \omega_{H(k)}$ (resp. on $T \prj^{N-1}$ by $H(k)$). This means that for all point $p \in X$ we have
\begin{equation} \label{distconviotav}
\mathrm{dist}_{H(k)} (\iota ( e^v (p) ), \iota(p)) \to 0
\end{equation}
as $k \to \infty$, where $\mathrm{dist}_{H(k)}$ is the distance in $\prj^{N-1}$ given by the Fubini--Study metric defined by $H(k)$.

Suppose now that $||\theta_* (v)||_{op} \to + \infty$ as $k \to + \infty$ (by taking a subsequence if necessary) and aim for a contradiction. Then for each (large enough) $k$ there exists a vector $w_k$ in $\cx^N \cong H^0 (X,L^k)$ such that
\begin{equation*}
\frac{||\theta_* (v) w_k ||_{H(k)}}{||w_k||_{H(k)}} \to + \infty ,
\end{equation*}
where $|| \cdot ||_{H(k)}$ is the norm on $H^0 (X,L^k)$ defined by $H(k)$. Since $X$ is not contained in any proper linear subspace of $\prj^{N-1}$, this means that there exists a constant $\delta>0$ such that for all large enough $k$ there exists a point $q_k \in X$ with $\mathrm{dist}_{H(k)} (e^{\theta_* (v)} \circ \iota ( q_k), \iota(q_k)) > \delta $. Recalling the equation (\ref{distconviotav}), this contradicts (\ref{expdconiotav}).


\end{proof}

We apply Lemma \ref{lemctlaps} to prove the following.

\begin{lemma} \label{lemposdefciak}
Suppose that we have a reference metric $\omega_0$ and a sequence $\{ H(k) \}_k$ with $H(k) \in \mathcal{B}^K_k$ which satisfies
\begin{equation} \label{bddhypnrbdd}
\sup_X | k  \omega_{H(k)} - k \omega_0 |_{\omega_0} < R'
\end{equation}
for some constant $R' >0$ uniformly of $k$, and $A = \theta_* (v)$ for $v \in \ai \mathfrak{k}$ such that $|v|_{k \omega_0}^2= O(1/k)$ where $| \cdot |_{k \omega_0}$ is a (pointwise) norm on $TX$ defined by $k \omega_0$. Then $||A||_{op} < C(R')$ for some constant $C(R')>0$ which depends only on $R'$ and is independent of $k$.
\end{lemma}


\begin{proof}
Note that $\sup_X | k  \omega_{H(k)} - k \omega_0 |_{\omega_0} < R'$ uniformly of $k$, combined with $|v|_{k \omega_0}^2= O(1/k)$, implies $|v|_{k \omega_{H(k)}}^2= O(1/k)$. Thus we can just apply Lemma  \ref{lemctlaps}.
\end{proof}

In what follows, we take the reference metric $\omega_0$ to be the extremal metric $\omega$. Recalling Remark \ref{roeovfgbbrik}, we thus obtain the following corollary of Proposition \ref{proprbalitomatrices}.

\begin{corollary} \label{corbdonaapprrbal}
Suppose that we have a sequence $\{ H(k) \}_k$ with $H (k) \in \mathcal{B}_k^K$, each of which satisfies $\lich{H(k)} \bar{\rho}_k (\omega_{H(k)})=0$ and $\sup_X |k  \omega_{H(k)} - k \omega |_{\omega} < R'$ for some constant $R' >0$ uniformly of $k$. Then we can write
\begin{equation} \label{brohitocala}
\bar{\rho}_k (\omega_{H(k)}) = \sum_{i,j} \left( I + C_A I - \frac{1}{2 \pi k} A\right)_{ij} h^k_{FS(H(k))} (s_i , s_j),
\end{equation}
where $A:= \theta_* (- \mathrm{grad} \bar{\rho}_k (\omega_{H(k)})) \in \theta_* (\ai \mathfrak{k})$ satisfies $||A||_{op } < C(R')$ uniformly of $k$, and $C_A$ is a constant which satisfies $- \frac{1}{2 \pi k} ||A||_{op} \le C_A \le  \frac{1}{2 \pi k} ||A||_{op }$, so that the average over $X$ of both sides of (\ref{brohitocala}) is 1.
\end{corollary}

\begin{remark} \label{remmataificola}
Note that Lemmas \ref{lemhol} and \ref{lemrbgfgenvfcentre} imply that $A$ in fact lies in $\theta_* (\ai \mathfrak{z}) \le \theta_* (\ai \mathfrak{k})$.
\end{remark}

\begin{remark} \label{remexpcaitoav}
Recalling that the centre of mass $\bar{\mu}'_X$ with respect to the basis $\{ s_i \}$ is given by $(\bar{\mu}'_X)_{ij} := k^n \int_X h^k_{FS(H(k))} (s_i , s_j) \frac{ \omega_{H(k)}^n}{n!}$, we integrate both sides of the equation (\ref{brohitocala}) to find that we have $k^n V = \sum_{i,j} \left( I + C_A I - \frac{1}{2 \pi k} A\right)_{ij} (\bar{\mu}'_X)_{ij} $. Noting $\tr (\bar{\mu}'_X) = k^n V$ which follows from (\ref{defoffseq}), we thus get $C_A$ explicitly as $C_A = \frac{1}{2 \pi k^{n+1} V} \tr (A \bar{\mu}'_X ) $.

\end{remark}


By replacing $\{ s_i \}$ by an $H$-unitarily equivalent basis if necessary, we may assume that $I +C_A I - \frac{A}{2 \pi k}$ is diagonal: $I +C_A I - \frac{A}{2 \pi k} = \textup{diag} (a_1 , \dots ,a_N)$ with each $a_i \in \rl$ satisfying $a_i = 1 + O(1/k)$, by Corollary \ref{corbdonaapprrbal}, thus $a_i >0$ for $k \gg 1$. This implies
\begin{equation} \label{eqrhol2bhonbrta}
\bar{\rho}_k (\omega_{H}) = \sum_i a_i |s_i|^2_{FS(H)^k} = \sum_i  |\sqrt{a_i} s_i|^2_{FS(H)^k} .
\end{equation}
Writing $H'$ for the hermitian form $\int_X h^k_{FS(H)}(,)\frac{\omega^n_{H}}{n!}$ on $H^0 (X,L^k)$ and $\{ s'_i \}$ for an $H'$-orthonormal basis, we thus have
\begin{equation*}
\bar{\rho}_k (\omega_{H})= \frac{V}{N} \sum_i |s'_i|^2_{FS(H)^k}  = \sum_i |\sqrt{a_i} s_i|^2_{FS(H)^k},
\end{equation*}
where the first equality is the definition of $\bar{\rho}_k$ (cf.~Definition \ref{defofbergfn}) and the second equality is provided by (\ref{eqrhol2bhonbrta}). This means that the basis $\{ \sqrt{a_i} s_i \}$ must be $H'$-unitarily equivalent to $\{ \sqrt{V/N} s'_i \}$ by the following lemma.

\begin{lemma} \label{leml2onbinj}
Suppose that we write $H'$ for the hermitian form $\int_X h^k_{FS(H)}(,)\frac{\omega^n_{FS(H)}}{n!}$ on $H^0(X,L^k)$ and that $\{ s'_i \}$ is a $H'$-orthonormal basis. If we have $\rho_k (\omega_{H}) =  \sum_i |s'_i|^2_{FS(H)^k}  = \sum_i | \tilde{s}_i|^2_{FS(H)^k}$ for another basis $\{ \tilde{s}_i \}$, then $\{ \tilde{s}_i \}$ is $H'$-unitarily equivalent to $\{ s'_i \}$.
\end{lemma}

\begin{proof}
We now write $h_{FS(H)} = e^{\phi} h_{FS(H')}$ for some $\phi \in C^{\infty} (X , \rl)$. Multiplying both sides of the equation $\sum_i |s'_i|^2_{FS(H)^k}  = \sum_i | \tilde{s}_i|^2_{FS(H)^k}$ by $e^{- k \phi}$, we get $1 = \sum_i |s'_i|^2_{FS(H')^k}  = \sum_i | \tilde{s}_i|^2_{FS(H')^k}$ since $\{ s'_i \}$ is an $H'$-orthonormal basis (cf. the equation (\ref{defoffseq})). Since $FS$ is injective (Lemma \ref{lemfsinj}), this means that $\{ \tilde{s}_i \}$ must be $H'$-unitarily equivalent to $\{ s'_i \}$.
\end{proof}

We thus obtain the following (cf.~Remark \ref{gaiaikpsiinkolichop}).

\begin{proposition} \label{propbchhqext}
Suppose that we have a sequence $\{ H(k) \}_k$ with $H(k) \in \mathcal{B}^K_k$, each of which satisfies $\lich{H(k)} \bar{\rho}_k (\omega_{H(k)})=0$ and $\sup_X |k  \omega_{H(k)} - k \omega |_{\omega} < R'$ uniformly of $k$. Then, writing $\{ s_i \}$ for an $H(k)$-orthonormal basis and $A:= \theta_* ( - \mathrm{grad} \bar{\rho}_k (\omega_{H(k)})) \in \theta_* (\ai \mathfrak{k})$, the basis $\{ s'_i \}$ defined by
\begin{equation} \label{cobhtl2bh}
s'_i := \sqrt{\frac{N}{V}} \left( I +C_A I -  \frac{A}{2 \pi k} \right)^{1/2}_{ij} s_j ,
\end{equation}
is a $\int_X h^k_{FS(H(k))}(,)\frac{\omega^n_{H(k)}}{n!}$-orthonormal basis, where $C_A$ is some constant of order $1/k$ and $\left(  I +C_A I- \frac{A}{2 \pi k} \right)_{ij} $ is the matrix for $ I +C_A I-  \frac{A}{2 \pi k}$ represented with respect to $\{ s_i \}$.

Conversely, if the basis $\{ s'_i \}$ as defined in (\ref{cobhtl2bh}) is a $\int_X h^k_{FS(H(k))}(,)\frac{\omega^n_{H(k)}}{n!}$-orthonormal basis, then $H(k)$ satisfies $\lich{H(k)} \bar{\rho}_k (\omega_{H(k)}) =0$.

\end{proposition}

In particular, we get the following results (cf.~Remark \ref{remmataificola}).

\begin{corollary} \label{crdlaprbal}
\
\begin{enumerate}
\item Suppose that we have a sequence $\{ H(k) \}_k$ with $H(k) \in \mathcal{B}_k^K$, each of which satisfies the equation $\lich{H(k)} \bar{\rho}_k (\omega_{H(k)})=0$ and $\sup_X |k  \omega_{H(k)} - k \omega |_{\omega} < R'$ for some constant $R' >0$ uniformly of $k$. Then there exists $g \in SL(N, \cx)$ such that 
\begin{equation} \label{rbalitmsobmu}
\bar{\mu}_X (g) = \frac{Vk^n}{N} \left( I + C_A I  - \frac{A}{2 \pi k}  \right)^{-1} ,
\end{equation}
where $A:= \theta_* ( - \mathrm{grad} \bar{\rho}_k (\omega_{H(k)})) \in \theta_* (\ai \mathfrak{z})$ satisfies $||A||_{op} < C(R')$ uniformly of $k$ and $C_A \in \rl$ is some constant of order $O(k^{-1})$ which satisfies $C_A = \frac{1}{2 \pi k^{n+1} V} \tr (A \bar{\mu}_X (g))$.
\item Conversely, if there exist a basis $\{ s'_i \}$ for $H^0 (X , L^k)$ defining a $\theta (K)$-invariant hermitian form $H (k) \in \mathcal{B}_k^K$, $A \in \theta_* (\ai \mathfrak{z})$, and some constant $C_A$, which satisfy $\bar{\mu}'_X  = \frac{Vk^n}{N} \left( I + C_AI- \frac{A}{2 \pi k} \right)^{-1}$, then $H(k) \in \mathcal{B}_k^K$ satisfies
\begin{equation*}
\lich{H(k)} \bar{\rho}_k (\omega_{H(k)}) =0
\end{equation*}
with $ \theta_* (- \mathrm{grad} \bar{\rho}_k (\omega_{H(k)})) = A$.
\end{enumerate}
\end{corollary}

\begin{remark} \label{expformcstcabtr}
Suppose that we have $\bar{\mu}_X (g) = \frac{Vk^n}{N} (I + C_AI- \frac{A}{2 \pi k})^{-1}$ for some constant $C_A$. Then multiplying both sides by $\bar{\mu}_X (g)^{-1}$ and taking the inverse, we have $\frac{Vk^n}{N} I = \bar{\mu}_X (g) + C_A \bar{\mu}_X (g) - \frac{\bar{\mu}_X (g) A}{2 \pi k}$, and hence by taking the trace, we have $C_A  =  \frac{1}{2 \pi k^{n+1} V} \tr (A \bar{ \mu}_X (g))$, by noting $\tr( \bar{\mu}_X (g)) =k^n V$ for any $g$. Thus, recalling Remark \ref{remexpcaitoav}, $C_A$ for which the trace is consistent in (\ref{rbalitmsobmu}) is the same as the one for which the averages are consistent in (\ref{brohitocala}).
\end{remark}

The appearance of the inverse on the right hand side of (\ref{rbalitmsobmu}) may look surprising, but this essentially comes from the one in Lemma \ref{lembergfnitoham}; see also Proposition \ref{rbalmiffinvocomtgtiox}. The reader is referred to \S \ref{reltrecrelwks} for comparisons to other ``relative balanced'' metrics (see also \cite[\S 6]{yhstability}).


\subsection{Approximate solutions to $\bar{\partial} \mathrm{grad}^{1,0}_{\omega_k} \rho_k (\omega_k) =0$ in terms of the centre of mass} \label{apsoldbgitotcom}

Now take the approximately q-ext metric $\omega_{(m)}$, as obtained in Corollary \ref{coraprbalmsolprob}, which satisfies (\ref{eqapprbalwefk}). By Lemma \ref{lemhamlinalg}, we have
\begin{equation} \label{bergfnapproxrbalm1}
 \bar{\rho}_k (\omega_{(m)}) + \frac{F_{m,k}}{4 \pi k^{m+2}} = - \frac{1}{2 \pi k} \sum_{ij} A_{ij} h^k_{(m)} (s_i , s_j) + \cst
\end{equation}
where $A := \theta_* (Jv_{(m)})$ and we recall that $h_{(m)}$ is of the form $FS(H)$ and that $\{ s_i \}$ in the above formula is an $H$-orthonormal basis. Noting that $\omega_{(m)}$ satisfies (for all large enough $k$)
\begin{equation} \label{unesommaprbfexmbrp}
\sup_X |k \omega_{(m)} - k \omega|_{\omega} < \cst . |\ddbar \phi_1|_{\omega} < R' ,
\end{equation}
say, and recalling Remark \ref{roeovfgbbrik}, we find that $||A||_{op} < C(R')$ by Lemma \ref{lemposdefciak} (by taking $\omega_0$ to be the extremal metric $\omega$).

\begin{remark}
In this section, the Hilbert--Schmidt norm $|| \cdot ||_{HS}$ will be with respect to $H$ which defines $\omega_{(m)}$ by $\omega_H$, as obtained in Corollary \ref{coraprbalmsolprob}. 
\end{remark}

Suppose that we write $P$ for the change of basis matrix from $\{s_i \}$ to a $\int_X h_{FS(H)}^k (,) \frac{\omega^n_{H}}{n!}$-orthonormal basis $\{ s_i' \}$, so that we have $\bar{\mu}_X ' = k^n (P^*P)^{-1} $, where $\bar{\mu}'_X$ is the centre of mass defined with respect to the basis $\{ s_i \}$ (cf.~the proof of Lemma \ref{lembergfnitoham}).

Now re-write the equation (\ref{bergfnapproxrbalm1}) as
\begin{equation*}
 \sum_{i,j} \left( \frac{V}{N} P^*P - C_0  I + \frac{A}{2 \pi k}    \right)_{ij} h^k_{(m)} (s_i, s_j) = 1 - \frac{F_{m,k}}{4 \pi k^{m+2}} ,
\end{equation*}
where the constant $C_0$ can be determined by taking the average of both sides; namely $C_0$ can be determined by the equation $1- \frac{1}{V} \sum_{i,j} \left(  C_0  I - \frac{A}{2 \pi k}    \right)_{ij} \int_X h^k_{(m)} (s_i, s_j) \frac{\omega_{(m)}}{n!} = 1+ O(k^{-m-2})$. Arguing as in (\ref{bdofaitfopnoa}), we see that $C_0$ can be estimated as
\begin{equation} \label{bdofc0itfaop}
|C_0| \le \frac{||A||_{op} +1}{2 \pi k}
\end{equation}
for all sufficiently large $k$, which is of order $1/k$ since $||A||_{op}$ is uniformly bounded by (\ref{unesommaprbfexmbrp}). We thus have, by noting $\sum_i |s_i|^2_{h^k_{(m)}}=1$,
\begin{equation} \label{approxeqitopafk0}
 \sum_{i,j} \left( \frac{V}{N} P^*P - C_0  I + \frac{A}{2 \pi k}    \right)_{ij}  h^k_{(m)} (s_i, s_j) - \left( 1 - \frac{F_{m,k}}{4 \pi k^{m+2}} \right) \sum_i |s_i|^2_{h^k_{(m)}} = 0 ,
\end{equation}
with some constant $C_0$ that is of order $1/k$. Since $ \frac{V}{N} P^*P - C_0  I + \frac{A}{2 \pi k} $ is a hermitian matrix, we can replace $\{ s_i \}$ by an $H$-unitarily equivalent basis so that $\frac{V}{N} P^*P - C_0 I + \frac{A}{2 \pi k} = \mathrm{diag} (d_1 , \dots , d_N) $, with $d_i \in \rl$, with respect to the basis $\{ s_i \}$. We can thus re-write the equation (\ref{approxeqitopafk0}) as
\begin{equation*}
\sum_i \left( d_i  - \left( 1 - \frac{F_{m,k}}{4 \pi k^{m+2}} \right) \right) |s_i|^2_{h^k_{(m)}} = 0 . 
\end{equation*}
Hence, arguing exactly as in the proof of Lemma \ref{lemfsinj} \cite[Lemma 3.1]{yhhilb}, we find
\begin{equation*}
|d_i - 1 | \le \frac{1}{2 \pi} N^2 \sup_X |F_{m,k}| k^{-m-2} = O(k^{2n-m-2}),
\end{equation*}
by recalling that $\sup_X |F_k|$ is bounded uniformly of $k$ (see the discussion preceding the equation (\ref{eqapprbalwefk})).

We thus see that there exists a hermitian matrix $E$ with $||E||_{op} = O(k^{2n-m-2})$ such that $\frac{V}{N} P^*P  = I   + C_0  I - \frac{A}{2 \pi k} +E$, or
\begin{equation*}
\bar{\mu}_X = k^n (P^*P)^{-1} = \frac{Vk^n }{N} \left( I + C_0 I - \frac{A}{2 \pi k} + E \right)^{-1} .
\end{equation*}
Define
\begin{equation*}
E' :=  \left(I + C_0 I - \frac{A}{2 \pi k} \right)^{-1} E,
\end{equation*}
which has $||E'||_{op} = O(k^{2n-m-2})$ by $||A||_{op} < C(R')$ and (\ref{bdofc0itfaop}). We may take $m$ and $k$ to be large enough so that $||E'||_{op} < 1/2$, say. We thus have
\begin{align*}
(\bar{\mu}_X' )_{ij} = \int_X h^k_{(m)}(s_i , s_j) \frac{(k \omega_{(m)})^n }{n!} 
&=\frac{V k^n}{N } \left[  \left(   I - E' + (E')^2+ \cdots \right)  \left(I + C_0 I - \frac{A}{2 \pi k} \right)^{-1}  \right]_{ij} \\
&=\frac{V k^n}{N } \left(I + C_0 I - \frac{A}{2 \pi k} \right)_{ij}^{-1} + (E'')_{ij}
\end{align*}
where $E'' \in T_H \mathcal{B}_k^K$ is a hermitian matrix defined by 
\begin{equation*}
E'' := \frac{Vk^n}{N} \left( - E' + (E')^2+ \cdots \right) \left(I + C_0 I - \frac{A}{2 \pi k} \right)^{-1}
\end{equation*}
which satisfies $||E''||_{op} = O(k^{2n-m-2})$ (by $||A||_{op} < C(R')$ and (\ref{bdofc0itfaop})). Since $m$ could be any positive integer, and recalling $|| E'' ||_{HS} = \mathrm{tr} (E''E'') \le \sqrt{N} || E'' ||_{op}$, we may replace $m$ by $m + 2n + n/2$ so as to have $||E''||_{HS} = O(k^{-m})$ (for notational convenience).

We now show that by perturbing $C_0$ slightly, we can assume that $\tr (E'')=0$. More precisely, we have the following.

\begin{lemma}  \label{lexpformcstcabtr}
Suppose $||E''||_{HS} = O(k^{-m})$, $||A||_{op} \le \cst$, and $C_0 = O(1/k)$. Then there exists a constant $\delta \in \rl$ with $|\delta| < 4 N^{-1/2} ||E''||_{HS} = O(k^{-m-n/2})$ such that 
\begin{equation*}
\frac{V k^n}{N } \left(I + C_0 I - \frac{A}{2 \pi k} \right)^{-1} + E'' - \frac{V k^n}{N } \left(I + (C_0 + \delta ) I - \frac{A}{2 \pi k} \right)^{-1}
\end{equation*}
is a trace free hermitian endomorphism which has the Hilbert-Schmidt norm of order $k^{-m}$; more precisely, we have
\begin{align}
&\left| \left| \frac{V k^n}{N } \left(I + C_0 I - \frac{A}{2 \pi k} \right)^{-1} + E'' - \frac{V k^n}{N } \left(I + (C_0 + \delta ) I- \frac{A}{2 \pi k} \right)^{-1} \right| \right|_{HS} \notag \\
& \le \left(  1 +  8  \left| \left| \left(I + C_0  I- \frac{A}{2 \pi k} \right)^{-1} \right| \right|_{op}  \right) ||E''||_{HS}
= O(k^{-m}). \label{estoftfedpitfedp}
\end{align}
\end{lemma}

\begin{proof}
We show that the map $U : \rl \to \rl$ defined by
\begin{equation*}
U (\delta) := \tr \left( \left(I + (C_0 + \delta ) I- \frac{A}{2 \pi k} \right)^{-1} \right)
\end{equation*}
is a local diffeomorphism with a particular lower bound on its linearisation. Writing $\left(I + C_0  I- \frac{A}{2 \pi k} \right) = \mathrm{diag} (a_1 , \dots , a_N)$ by unitarily diagonalising it, where $a_i = 1+ O(1/k)$ by $||A||_{op} < C(R')$ and (\ref{bdofc0itfaop}), we have $U(\delta) = \sum_i (a_i + \delta)^{-1} = \sum_i a_i^{-1} (1 + \delta/ a_i)^{-1} $, whose linearisation at $0$ is $DU |_0 (\delta) = - \delta \sum_i a_i^{-2} $.


Since $a_i = 1+ O(1/k)$ implies $\left| DU|_0 \right| > N/2$ if $k$ is sufficiently large, we see that $U$ is indeed a local diffeomorphism whose linearisation can be bounded from below by $N/2$.

Thus, by using the quantitative version of the inverse function theorem (see e.g.~\cite[Theorem 5.3]{finethesis}), we can show that there exists some $\delta \in \rl$ so that we have
\begin{equation*}
 \tr \left( \frac{V k^n}{N } \left(I + (C_0 + \delta ) I - \frac{A}{2 \pi k} \right)^{-1} \right) = \tr \left( \frac{V k^n}{N } \left(I + C_0 I- \frac{A}{2 \pi k} \right)^{-1} \right) + \tr (E''),
\end{equation*}
which satisfies
\begin{equation} \label{estondeltaitfedp}
|\delta| < 4 \frac{|\tr(E'')|}{N} \le 4 N^{-1/2} ||E''||_{HS} = O(k^{-m-\frac{n}{2}})
\end{equation}
since $|\tr(E'')| \le \sqrt{N} ||E''||_{HS}$.

We now estimate the Hilbert--Schmidt norm of the trace free hermitian matrix
\begin{align*}
&\frac{V k^n}{N } \left(I + C_0 I- \frac{A}{2 \pi k} \right)^{-1} + E'' - \frac{V k^n}{N } \left(I + (C_0 + \delta ) I- \frac{A}{2 \pi k} \right)^{-1} \\
&=  E'' + \frac{V k^n}{N }  \delta \left(I + C_0  I- \frac{A}{2 \pi k} \right)^{-2} - \frac{V k^n}{N }  \delta^2 \left(I + C_0  I- \frac{A}{2 \pi k} \right)^{-3} + \cdots .
\end{align*}
Recalling $\left| \left| \left(I + C_0  I- \frac{A}{2 \pi k} \right)^{-1} \right| \right|_{op} \le \cst$ independently of $k$, we find
\begin{equation*}
\left| \left| \delta \left(I + C_0  I- \frac{A}{2 \pi k} \right)^{-2} - \frac{V k^n}{N }  \delta^2 \left(I + C_0  I- \frac{A}{2 \pi k} \right)^{-3} + \cdots \right| \right|_{op} = O(k^{-m-\frac{n}{2}})
\end{equation*}
for all sufficiently large $k$, and hence has Hilbert--Schmidt norm of order $k^{-m}$. Recalling $||E''||_{HS} = O(k^{-m})$, we finally see
\begin{align}
&\left| \left| \frac{V k^n}{N } \left(I + C_0 I- \frac{A}{2 \pi k} \right)^{-1} + E'' - \frac{V k^n}{N } \left(I + (C_0 + \delta ) I- \frac{A}{2 \pi k} \right)^{-1} \right| \right|_{HS} \notag \\
& \le ||E''||_{HS}+ \sqrt{N} \left| \left| \delta \left(I + C_0  I- \frac{A}{2 \pi k} \right)^{-2} - \frac{V k^n}{N }  \delta^2 \left(I + C_0  I- \frac{A}{2 \pi k} \right)^{-3} + \cdots \right| \right|_{op} \notag \\
&\le  ||E''||_{HS}+ 2 \sqrt{N} \left| \left| \left(I + C_0  I- \frac{A}{2 \pi k} \right)^{-1} \right| \right|_{op} |\delta| \notag \\
& \le \left(  1 +  8  \left| \left| \left(I + C_0  I- \frac{A}{2 \pi k} \right)^{-1} \right| \right|_{op}  \right) ||E''||_{HS}
= O(k^{-m}), \notag
\end{align}
where we used $||\cdot ||_{HS} \le \sqrt{N} || \cdot ||_{op}$ in the first inequality and (\ref{estondeltaitfedp}) in the last inequality.
\end{proof}

Summarising the argument above, we obtain the following.

\begin{corollary} \label{capproxsbal}
For any $m \ge1$ and any large enough $k \gg 1$ there exists a $\theta (K)$-invariant hermitian form $H= H_{m}(k) \in \mathcal{B}_k^K$ and a traceless hermitian $\theta(K)$-invariant endomorphism $\tilde{E}= \tilde{E}_{m} (k)$ on $H^0 (X , L^k)$ which satisfy the following: there exists an element $A \in \theta_* (\ai \mathfrak{z})$ with $||A||_{op} < \cst$ uniformly of $k$ and a constant $C_0 \in \rl$ which is of order $1/k$ such that the equation
\begin{equation*}
(\bar{\mu}_X ')_{ij} = \int_X  h^k_{FS(H)} (s_i , s_j) \frac{(k\omega_{FS(H)})^n }{n!}  = \frac{V k^n}{N } \left(I + C_0 I - \frac{A}{2 \pi k} \right)_{ij}^{-1} + (\tilde{E})_{ij}
\end{equation*}
holds with respect to an $H$-orthonormal basis $\{ s_i \}$, with $||\tilde{E}||_{HS} \le \cst . k^{-m}$ where the Hilbert--Schmidt norm $|| \cdot ||_{HS}$ is defined with respect to $H$. 
\end{corollary}

\begin{remark}
Since $H$ is $\theta (K)$-invariant, $\bar{\mu}_X'$ is $\theta(K)$-invariant. This means that $\tilde{E}$ is $\theta (K)$-invariant and hermitian, since $\bar{\mu}_X' - \frac{Vk^n}{N}(I + C_0 I - \frac{A}{2 \pi k})^{-1}$ is.
\end{remark}

Henceforth we write $H_0$ for $H_{m} (k)$ above, and $\tilde{E}_0$ for $\tilde{E}_{m} (k)$ above.



\section{Gradient flow} \label{qexmsgf}

\subsection{Modified balancing energy $\mathcal{Z}^A$} \label{secmbaleza}

Recall first of all that in the cscK case, i.e.~when $\mathrm{Aut}_0 (X,L)$ is trivial, balanced metrics are precisely the critical points of a functional $\mathcal{Z} : \mathcal{B}_k \to \rl$ called the \textbf{balancing energy}, defined for a geodesic $\{ H(t) \}$ in $\mathcal{B}_k$, where $H(t) = e^{tB} H(0)$ with $B \in T_{H(0)} \mathcal{B}_k \cong \mathrm{Herm} (H^0 (X,L^k))$, as
\begin{equation*}
\mathcal{Z} (H(t)) := I \circ FS(H(t)) + \frac{Vk^n}{N} \tr ( \log H(t)) .
\end{equation*}
In the above, $I : \mathcal{H} (X,L) \to \rl$ is defined for a path $\{ e^{\phi_t}h \}$ in $\mathcal{H} (X,L)$ by
\begin{equation*}
I (e^{\phi_t} h):= -k^{n+1} \int_X \phi_t \sum_{i=0}^n (\omega_h - \ai \ddbar \phi_t)^i \wedge \omega_h^{n-i} ,
\end{equation*}
where $h$ is some reference metric, and changing the reference metric $h$ will only result in an overall additive constant.

The original argument for finding the balanced metric (in the cscK case) in \cite{donproj1} was to find an approximately balanced metric, which is very close to attaining the minimum of $\mathcal{Z}$, and then perturb it to a genuinely balanced metric (i.e.~the minimum of $\mathcal{Z}$) by driving it along the gradient flow of $\mathcal{Z}$ to attain the global minimum. The reader is referred to \cite{donproj1} for the details. The crucial point is that $\mathcal{Z}$ is \textit{convex} along geodesics in $\mathcal{B}_k$ (with respect to the bi-invariant metric), as we recall in Theorem \ref{lemhessbalenfinod}.


We now consider the following functional, which is more appropriate for our purpose of finding q-ext metrics.

\begin{definition}
We define a functional $\mathcal{Z}^A : \mathcal{B}_k^K \to \rl$ by
\begin{equation*}
\mathcal{Z}^A (H(t)) := I \circ FS(H(t)) + \frac{Vk^n}{N} \tr \left( \left(I + C_AI - \frac{A}{2 \pi k} \right)^{-1} \log H(t) \right),
\end{equation*}
for some fixed $A \in \theta_* (\ai \mathfrak{z})$ and some fixed constant $C_A \in \rl$. We call $\mathcal{Z}^A$ the \textbf{modified balancing energy}.
\end{definition}
\begin{remark} \label{remmbalhesbhess}
Note that the Hessian of $\mathcal{Z}$ is equal to the Hessian of $\mathcal{Z}^A$, since their difference
\begin{equation*}
\mathcal{Z} (H(t)) - \mathcal{Z}^A (H(t)) = \frac{Vk^n}{N} \tr ( \log H(t)) - \frac{Vk^n}{N} \tr \left( \left(I + C_AI-  \frac{A}{2 \pi k} \right)^{-1} \log H(t) \right),
\end{equation*}
with $H(t) = e^{tB} H(0)$, is linear in $t$. Thus, we see that $\mathcal{Z}^A$ is \textit{convex} along geodesics in $\mathcal{B}_k^K$ (cf.~Theorem \ref{lemhessbalenfinod}).
\end{remark}


Similarly to the usual balanced case (cf.~\cite[Lemma 3]{donproj2}), the first variation of $\mathcal{Z}^A$ can be computed as follows
\begin{equation*}
\delta \mathcal{Z}^A (H (t)) =  - \int_X  h^k_{FS(H (t))} ( s^{H(t)}_i ,   s^{H(t)}_j)  \frac{k^n \omega_{FS(H(t))}^n}{n!}  +  \frac{Vk^n}{N} \left( I+ C_AI - \frac{A}{2 \pi k} \right)_{ij}^{-1}  ,
\end{equation*}
where $\{ s^{H(t)}_i \}$ is an $H(t)$-orthonormal basis and, $\left( I+ C_AI - \frac{A}{2 \pi k} \right)_{ij}^{-1}$ in the above is the hermitian endomorphism $\left( I+ C_AI - \frac{A}{2 \pi k} \right)^{-1}$ represented with respect to $\{ s_i^{H(t)} \}$. This implies that $\delta \mathcal{Z}^A (H (t)) =0$ if and only if $\{ s^{H(t)}_i  \}$ defines an embedding with $\bar{\mu}'_X = \frac{Vk^n}{N} \left( I + C_A I - \frac{A}{2 \pi k} \right)^{-1}$. Summarising the discussion above, the solution of the equation $\bar{\mu}'_X = \frac{Vk^n}{N} (I + C_A I - \frac{A}{2 \pi k})^{-1}$ can be characterised as the critical point of the functional $\mathcal{Z}^A$, which is convex along geodesics in $\mathcal{B}_k^K$.

\subsection{Hessian of the balancing energy} \label{sechessbalen}

We now recall the Hessian of the (usual) balancing energy $\mathcal{Z}$, following the exposition given in \cite{fine10, fine12}. Fixing $H(t) \in \mathcal{B}_k^K$ for the moment, consider now the orthogonal decomposition $\iota^* T \prj^{N-1} = TX \oplus \mathcal{N}_t$ (as $C^{\infty}$-vector bundles on $X$) with respect to the Fubini--Study metric $\omega_{\widetilde{FS} (H (t))}$ (or more precisely $\iota^* \omega_{\widetilde{FS} (H(t))}$) on $\prj^{N-1}$ induced from $H (t)$ (cf.~\cite[p703]{ps04}). Given a hermitian endomorphism $\xi \in T_{H(t)} \mathcal{B}_k^K \cong \mathrm{Herm} (H^0 (X , L^k))^K$, we write $X_{\xi}$ for the corresponding holomorphic vector field on $\prj^{N-1}$. Write $\pi_{\mathcal{N}_t} (X_{\xi})$ for the projection of $X_{\xi}$ on the $\mathcal{N}_t$-factor in $\iota^* T \prj^{N-1} = TX \oplus \mathcal{N}_t$, and $\pi_T (X_{\xi})$ for the one on the $TX$-factor. We thus get a map $P: T_{H(t)} \mathcal{B}_k^K \to C^{\infty} (\mathcal{N}_t)$ defined by $P (\xi):= \pi_{\mathcal{N}_t} (X_{\xi})$. Write $P^*$ for the adjoint of $P$ defined with respect to the inner product $\tr(\xi_1 \xi_2)$ on $T_{H(t)} \mathcal{B}_k^K$ and the $L^2$-metric defined by $\omega_{\widetilde{FS} (H(t))}$ on the fibres and $k \omega_{H (t)}$ on the base. Note that the inner product $\tr(\xi_1 \xi_2)$ is nothing but the Hilbert--Schmidt inner product defined with respect to $H(t)$, since $\xi_1 , \xi_2 \in T_{H(t)} \mathcal{B}_k^K$.

\begin{theorem} \emph{(Fine \cite[Lemma 17]{fine10})} \label{lemhessbalenfinod}
Writing $P: T_{H(t)} \mathcal{B}_k^K \to C^{\infty} (\mathcal{N}_t)$ defined by $P (\xi):= \pi_{\mathcal{N}_t} (X_{\xi})$, as above, we have $ \mathrm{Hess} (\mathcal{Z} (H(t))) = P^*P$. In particular,
\begin{align*}
\tr(\xi_1 \mathrm{Hess} (\mathcal{Z} (H(t)) \xi_2) &= (\pi_{\mathcal{N}_t} (X_{\xi_1}) , \pi_{\mathcal{N}_t} (X_{\xi_2}))_{L^2 (t)} \\
&= \int_X \rea (\pi_{\mathcal{N}_t} (X_{\xi_1}) , \pi_{\mathcal{N}_t} (X_{\xi_2}))_{\iota^* \omega_{\widetilde{FS} (H(t))}} \frac{k^n \omega^n_{H(t)}}{n!}.
\end{align*}
\end{theorem}

\begin{remark}
The ``diagonal'' elements of $\mathrm{Hess} (\mathcal{Z} (H(t))) $ are in fact computed by Phong and Sturm \cite{ps03, ps04} and implicitly by Donaldson \cite{donproj1}. We will later need to know some of the off-diagonal terms of $ \mathrm{Hess} (\mathcal{Z} (H(t))) =  \mathrm{Hess} (\mathcal{Z}^A (H(t))) $.
\end{remark}


We now wish to estimate $|| \pi_{\mathcal{N}_t} (X_{\xi})||_{L^2 (t)}^2$. This was done originally by Donaldson \cite{donproj1} and improved by Phong and Sturm \cite{ps04} when the connected component $\text{Aut}_0 (X,L)$ of the automorphism group was trivial. In our situation we cannot assume this hypothesis, but we now invoke the following trick used by Mabuchi \cite{mab05, mab09}. Recall that, by Lemma \ref{lemdefofthtosl}, $\mathrm{Aut}_0 (X,L)$ (with the Lie algebra $\mathfrak{g}$) is a subgroup of $SL(H^0 (X,L^k))$ (with Lie algebra $\mathfrak{sl} = \mathfrak{sl} (H^0 (X ,L^k))$), and hence we have $\mathfrak{sl}  = \mathfrak{g} \oplus \mathfrak{g}^{\perp}_t$, where $\lieg^{\perp}_t$ is the orthogonal complement of $\lieg$ in $\mathfrak{sl}$ with respect to the $L^2$-inner product defined by the Fubini--Study metric on $\prj^{N-1}$ given by $H (t)$, i.e.~with respect to the metric $(,)$ defined by $(\xi_1 , \xi_2) : = (X_{\xi_1} , X_{\xi_2})_{L^2 (t)}$, where the $L^2$-product is defined by $\omega_{\widetilde{FS} (H(t))}$ on the fibres and $k \omega_{FS(H(t))}$ on the base, as we mentioned above. Note that this $L^2$-product does define a metric on $\mathfrak{sl}$ since $X$ is not contained in any proper linear subspace of $\prj^{N-1}$.



Writing ${\xi} = \alpha + \beta$ where $\alpha \in \mathfrak{g}$ and $\beta \in \mathfrak{g}^{\perp}_t$, we obviously have $\pi_{\mathcal{N}_t} (X_{\xi} ) = \pi_{\mathcal{N}_t} (X_{\beta})$. An intuitive idea is that, if $\xi \in \mathfrak{sl} $ is contained in the $\lieg^{\perp}_t$-factor, we can apply the well-known estimate (Theorem \ref{psestimate}) due to Donaldson, Phong--Sturm, and Fine, to get the lower bound of the eigenvalues of the Hessian of $\mathcal{Z}^{A} (H(t))$ (restricted to $\lieg^{\perp}_t$) so that we can run the downward gradient flow on the space of positive definite $K$-invariant hermitian matrices $\mathcal{B}_k^K$ driven by $\text{pr}_{\perp, t} (\delta \mathcal{Z}^{A} (H(t)))$; see \S  \ref{secgrflza} for the details. 

We now recall the following notion from \cite{donproj1}.

\begin{definition} \label{defofrbddgeom}
A metric $\tilde{\omega} \in k c_1 (L)$ has \textbf{$R$-bounded geometry} if it satisfies the following conditions: fixing an integer $l \ge 4$ and a reference metric $\omega_0 \in c_1 (L)$, $\tilde{\omega}$ satisfies $\tilde{\omega} > R^{-1} k {\omega}_0$ and $|| \tilde{\omega} - k \omega_0 ||_{C^l , k \omega_0} < R$ where $|| \cdot ||_{C^l ,k \omega_0}$ is the $C^l$-norm on the space of 2-forms defined with respect to the metric $k \omega_0$. The basis $\{ s_i \}$ is said to have \textbf{$R$-bounded geometry} if the hermitian endomorphism $H (t)$ which has $\{ s_i \}$ as its orthonormal basis has $R$-bounded geometry.
\end{definition}

With these preparations, we can now state the following theorem (cf.~\cite[Theorem 2]{ps04}).

\begin{theorem} \label{psestimate}
\emph{(Donaldson \cite{donproj1}, Phong--Sturm \cite{ps04}, Fine \cite{fine12})}
Suppose that $\textup{Aut}_0 (X,L)$ is trivial. Suppose also that we have a basis $\{ s_i \}$ with respect to which $\bar{\mu}_X' = D_k +E_k$, where $D_k$ is a scalar matrix with $D_k \to I$ as $k \to \infty$. For any $R>0$ there exists a positive constant $C_R$ depending only on $R$ and $\epsilon < 1/10$ such that, for any $k$, if the basis $\{ s_i \}$ for $H^0 (X , L^k)$ has $R$-bounded geometry and if $||E_k||_{op } < \epsilon$, then 
\begin{equation*}
|| \pi_{\mathcal{N}_t} ( X_{\xi}) ||^2_{L^2 (t) } > C_R k^{-2} || \xi ||^2_{HS (t)} ,
\end{equation*}
where the $L^2$-metric $|| \cdot ||_{L^2 (t)}$ on the vector fields on $X$ is defined by the Fubini-Study metric of the hermitian form $H(t)$ which has $\{ s_i \}$ as its orthonormal basis, and the Hilbert-Schmidt norm $|| \cdot ||_{HS (t)}$ is defined by the hermitian form $H(t)$ which has $\{ s_i \}$ as its orthonormal basis.
\end{theorem}

\begin{remark} \label{deccomitodkek}
The hypothesis $\bar{\mu}_X' = D_k +E_k$ is satisfied when we have $\bar{\mu}_X '  = \frac{V k^n}{N } \left(I + C_0 I - \frac{A}{2 \pi k} \right)^{-1} + \tilde{E}$ 
with $||\tilde{E}||_{HS} = O(k^{-m})$, as in Corollary \ref{capproxsbal}, by noting that we can define $D_k := \frac{1}{N} \tr \left( \frac{Vk^n}{N} \left( I + C_0 I - \frac{A}{2 \pi k} \right)^{-1} \right) I$, which does converge to $I$ as $k \to \infty$, and that the operator norm of $E_k := \bar{\mu}_X' - D_k$ is of order $1/k$, since $||A||_{op} < \cst $ and $C_0=  O(1/k)$ by (\ref{bdofc0itfaop}).
\end{remark}




We now recall the proof of this theorem, where we closely follow the exposition given in \cite[pp702-710]{ps04}. The theorem is a consequence of the following three estimates:
\begin{align}
|| \xi||^2_{HS (t)} &\le C'_R k ||X_{\xi}||^2_{L^2 (t)} \label{esthsitol2novf} \\
||X_{\xi}||^2_{L^2 (t)}  &= ||\pi_T ( X_{\xi} )||^2_{L^2 (t)} + ||\pi_{\mathcal{N}_t}( X_{\xi})||^2_{L^2 (t)} \label{esthsitol2novf2} \\
C_R ||\pi_T (X_{\xi})||^2_{L^2 (t)} &\le k ||\pi_{\mathcal{N}_t}( X_{\xi})||^2_{L^2 (t)} \label{esthsitol2novf3}
\end{align}
The second equality (\ref{esthsitol2novf2}) is an obvious consequence of the orthogonal decomposition $\iota^* T \prj^{N-1} = TX \oplus \mathcal{N}_t$ with respect to $\omega_{\widetilde{FS} (H (t))}$, and the first inequality (\ref{esthsitol2novf}) does not use the hypothesis that $\textup{Aut}_0 (X,L)$ is trivial, and hence carries over word by word to the case when $\textup{Aut}_0 (X,L)$ is not trivial.

The hypothesis of $\textup{Aut}_0 (X,L)$ being trivial was crucially used in the third estimate (\ref{esthsitol2novf3}), which relies on the following estimate \cite[(5.12)]{ps04} for an arbitrary smooth vector field $W$ on $X$
\begin{equation} \label{dbarest}
 ||W||^2_{L^2 (t)} \le \cst . || \bar{\partial} (W)||^2_{L^2 (t)}
\end{equation}
which is true if and only if $\textup{Aut}(X)$ is discrete. Phong--Sturm's argument was to apply this inequality to $W = \pi_T (X_{\xi})$ and combine it with the estimate \cite[(5.15)]{ps04}
\begin{equation*}
|| \pi_{\mathcal{N}_t} (\tilde{V}) ||^2_{L^2 (t)} \ge C_R ||\bar{\partial} (\pi_{\mathcal{N}_t}(\tilde{V})) ||^2_{L^2 (t)}
\end{equation*}
which holds for any holomorphic vector field $\tilde{V}$ on $\prj^{N-1}$ (which we take to be $X_{\xi}$), irrespective of whether $\textup{Aut}(X)$ is discrete or not. Observe that $\bar{\partial} \tilde{V} = 0 = \bar{\partial} (\pi_T (\tilde{V})) + \bar{\partial} (\pi_{\mathcal{N}_t} (\tilde{V}))$ implies $c_R || \bar{\partial} (\pi_T (\tilde{V}))||_{L^2 ( t )} \le || \pi_{\mathcal{N}_t} (\tilde{V}) ||_{L^2 (t )}$. Thus, by applying this and the estimate (\ref{dbarest}) applied to $W = \pi_T (X_{\xi})$, we get (\ref{esthsitol2novf3}).

Thus, the only hindrance to extending Phong--Sturm's theorem to the case where $\textup{Aut}_0 (X,L)$ is not trivial is the lack of (\ref{dbarest}), which is substantial. However, the decomposition $\mathfrak{sl}= \lieg \oplus \lieg^{\perp}_t$ means that the estimate (\ref{dbarest}) holds for the (smooth) vector fields of the form $\pi_T (X_{\beta})$ where $\beta \in \lieg^{\perp}$, since the elements $\alpha \in \lieg$ are precisely the ones that generate $X_{\alpha}$ with $\bar{\partial} (\pi_T  (X_{\alpha})) =0$, i.e.~the kernel $\ker \bar{\partial}$ is precisely the image $\{ X_{\alpha} | \alpha \in \lieg \}$ of $\lieg$. Since the image $\{ X_{\beta} | \beta \in \lieg^{\perp}_t \}$ of $\lieg^{\perp}_t$ is precisely the $L^2$-orthogonal complement of $\ker \bar{\partial}$ in $\mathfrak{sl}$, recalling that $\lieg^{\perp}_t$ is defined as an orthogonal complement of $\lieg$ with respect to the $L^2$ metric induced from $\omega_{\widetilde{FS} (H(t))}$, $\bar{\partial}$ is invertible on the set of vector fields $\pi_T (X_{\beta})$ with $\beta \in \lieg^{\perp}_t$, with the estimate (\ref{dbarest}).


Thus we have the following estimate.
\begin{lemma}
\emph{(cf.~Mabuchi \cite[p235]{mab05}, \cite[p130]{mab09})}
Suppose that we have the same hypotheses as in Theorem \ref{psestimate}, apart from that $\mathrm{Aut}_0 (X,L)$ is no longer trivial. We have
\begin{equation} \label{psestperp}
||\pi_{\mathcal{N}_t} (X_{\beta})||^2_{L^2 (t)} \ge C_R k^{-2} || \beta ||^2_{HS (t)}
\end{equation}
for any $\beta \in \lieg^{\perp}_t$. 
\end{lemma}

\subsection{Gradient flow} \label{secgrflza}

Let $H_0$ be the approximately q-ext matrix as obtained in Corollary \ref{capproxsbal}. We now aim to perturb this matrix to a genuine q-ext one by using a geometric flow on a finite dimensional manifold $\mathcal{B}^K_k$. In this section, we show that such flow does converge, but also show that $\mathrm{Aut}_0 (X,L)$ being nontrivial implies that the limit of the flow is not quite the (genuine) q-ext metric that we seek (cf.~Proposition \ref{proplimgfh1}); it will be obtained in Proposition \ref{propiterschrbalm}, \S \ref{seciterconsorbalm}, by an iterative construction.

Recall the decomposition $\mathfrak{sl} = \lieg \oplus \lieg^{\perp}_t $ with respect to $H(t) \in \mathcal{B}_k^K$, as introduced in \S \ref{sechessbalen}. Suppose that we write $\mathrm{pr}_{\lieg} : \mathfrak{sl} \surj \lieg$ for the projection onto $\lieg$ and $\mathrm{pr}_{\perp, t} : \mathfrak{sl} \surj \lieg^{\perp}_t$ for the projection onto $\lieg^{\perp}_t$. We consider the following ODE 
\begin{equation} \label{grflbmblfokivs}
\frac{d H(t)}{dt}  =  -  \text{pr}_{\perp , t} \left( \delta \mathcal{Z}^{A} (H(t)) \right)
\end{equation}
on the finite dimensional symmetric space $\mathcal{B}_k^K$, with the initial condition $H(0) = H_0$. This is well-defined, since at $t=0$, $\delta \mathcal{Z}^{A} (H_0)$ is $K$-invariant and hermitian by Corollary \ref{capproxsbal}, and hence $ \text{pr}_{\perp , t} \left( \delta \mathcal{Z}^{A} (H(0)) \right)$ is indeed $K$-invariant (since $K$ acts on $\lieg$ and hence preserves $\mathfrak{sl} = \lieg \oplus \lieg^{\perp}_t$, by noting that the orthogonality is defined by a $K$-invariant metric $\widetilde{FS} (H (t))$) and hermitian, defining a vector in $T_{H_0} \mathcal{B}^K_k$. By exactly the same argument, along the flow (\ref{grflbmblfokivs}), $ \text{pr}_{\perp , t} \left( \delta \mathcal{Z}^{A} (H(t)) \right)$ remains $K$-invariant and hermitian for $t >0$ since $H(t) \in \mathcal{B}_k^K$. 

Moreover, we can multiply the right hand side of the equation (\ref{grflbmblfokivs}) by a cutoff function that is supported on a compact neighbourhood of radius 1 around $H_0$ without changing the flow; this will be justified in (\ref{estlimgrflza0}) and (\ref{estlimgrflza}), as they state that the the flow is contained in this neighbourhood for all time if we start from $H_0$. Then the vector field on the right hand side of (\ref{grflbmblfokivs}) is compactly supported, and the flow can be extended indefinitely by the standard ODE theory, i.e.~the solution to (\ref{grflbmblfokivs}) exists for all time. 



Note
\begin{equation*}
\frac{d }{dt} \left( \delta \mathcal{Z}^A (H(t)) \right) = \mathrm{Hess} (\mathcal{Z}^A  (H(t))) \cdot \frac{d H(t)}{dt}  = - \text{Hess}(\mathcal{Z}^{A} (H(t)) ) \cdot \text{pr}_{\perp, t} \left( \delta \mathcal{Z}^{A} (H(t)) \right).
\end{equation*}
and recall that the Hessian of $\mathcal{Z}^A$ is exactly the same as that of $\mathcal{Z}$, the usual balancing energy (cf.~Remark \ref{remmbalhesbhess}), and that the Hessian of $\mathcal{Z}$ is degenerate along the $\lieg$-direction, as we saw in Theorem \ref{lemhessbalenfinod}. This means that we have a block diagonal decomposition of $\mathrm{Hess} (\mathcal{Z}^A (H(t)))$ as
\begin{equation*}
\mathrm{Hess} (\mathcal{Z}^A (H(t))) = 
\begin{pmatrix}
0 & 0 \\
0 & \tilde{P}_t
\end{pmatrix},
\end{equation*}
according to the decomposition $\mathfrak{sl} = \lieg \oplus \lieg^{\perp}_t$, where $\tilde{P}_t$ is a positive definite matrix whose lowest eigenvalue can be estimated as in (\ref{psestperp}). In particular, we obtain the following.

\begin{lemma} \label{lgcptombleacsvd}
\begin{equation*}
\frac{d }{dt} \mathrm{pr}_{\lieg} \left(  \delta \mathcal{Z}^A (H(t)) \right)=0
\end{equation*}
along the flow $H(t)$ defined by (\ref{grflbmblfokivs}).
\end{lemma}

Suppose that we write $\mathcal{G}^A (t)$ for $\text{pr}_{\perp , t} \left( \delta \mathcal{Z}^{A} (H(t)) \right)$ in order to simplify the notation. We then have
\begin{equation*}
\frac{1}{2} \frac{d}{dt} || \mathcal{G}^A (t) ||^2_{HS(t)} = \frac{1}{2} \frac{d}{dt} \mathrm{tr} ( \mathcal{G}^A (t) \mathcal{G}^A (t)) =  - \mathrm{tr} \left( \mathcal{G}^A (t) \cdot \text{Hess}(\mathcal{Z}^{A} (H(t)) ) \cdot \mathcal{G}^A (t) \right) ,
\end{equation*}
by recalling that $\mathrm{tr} ( \mathcal{G}^A (t) \mathcal{G}^A (t))$ is equal to $|| \mathcal{G}^A (t) ||^2_{HS(t)}$. Recall (cf.~Remark \ref{remmbalhesbhess}) that $\mathcal{Z}^{A} (H(t)) $ is convex along geodesics for all $t$. Thus, the above equation means that, along the flow, $||\mathcal{G}^A (t) ||_{HS(t)}$ is monotonically decreasing. Combined with Lemma \ref{lgcptombleacsvd} and Lemma \ref{lemdecml2hs} to be proved later, this means that the hypotheses in Theorem \ref{psestimate} are always satisfied along the flow. Thus we can apply the estimate given by Theorem \ref{psestimate} along the flow for all $t >0$. Theorem \ref{lemhessbalenfinod} and the estimate (\ref{psestperp}) imply that we have
\begin{equation*}
\frac{1}{2} \frac{d}{dt} || \mathcal{G}^A (t) ||^2_{HS(t)} \le - \lambda_1 || \mathcal{G}^A (t) ||_{HS(t)}^2 ,
\end{equation*}
where we wrote 
\begin{equation} \label{estlevhobekm2}
\lambda_1 := C_R k^{-2} >0
\end{equation}
for the lowest eigenvalue of $\text{Hess} (\mathcal{Z}^A (H(t)))  $ restricted to $\lieg^{\perp}_t$, as estimated in (\ref{psestperp}). It easily follows that we have
\begin{equation} \label{bdoprmbaleib0}
|| \mathcal{G}^A (t) ||_{HS(t)} \le e^{- \lambda_1 t } || \mathcal{G}^A (0) ||_{HS(0)} .
\end{equation}

We now evaluate the length of the path $\{ H(t) \}$ with respect to the bi-invariant metric. Namely, we compute $\mathrm{dist} (H(t_1) , H(t_2)) := \int_{t_1}^{t_2} || H' (s) ||_{HS(s)} ds$ for $t_1 > t_2$. Observe first of all that
\begin{align} 
\int_{t_1}^{t_2} || H' (s) ||_{HS(s)} ds &= \int_{t_1}^{t_2} || \mathcal{G}^A (s) ||_{HS(s)} ds \notag \\
&\le \frac{1}{\lambda_1} (e^{- \lambda_1 t_1} - e^{- \lambda_1 t_2}) || \mathcal{G}^A (0) ||_{HS(0)}, \label{estlimgrflza0}
\end{align}
where we used $H' (t) = - \mathcal{G}^A (t)$, which is just (\ref{grflbmblfokivs}), and the estimate (\ref{bdoprmbaleib0}). Thus, given an increasing sequence $\{ t_i \}_i$ of positive real numbers, we see that the sequence $\{ H(t_i) \}_i$ is Cauchy in $\mathcal{B}_k^K$ with respect to the bi-invariant metric. Thus the limit exists in $\mathcal{B}_k^K$, and the distance from the initial metric $H_0$ to the limit can be estimated as
\begin{align}
\mathrm{dist} (H(\infty) , H(0)) = \int_{0}^{\infty} || H' (s) ||_{HS(s)} ds &= \int_{0}^{\infty} || \mathcal{G}^A (s) ||_{HS(s)} ds \notag \\
&\le \frac{1}{\lambda_1} || \mathcal{G}^A (0) ||_{HS(0)} = O(k^{-m+2}).  \label{estlimgrflza}
\end{align}

\begin{remark} \label{remdifthsnhiftclth0}
Observe that (\ref{estlimgrflza}) implies that we can write $H (\infty) = e^{\xi} H (0)$ with $\xi \in T_{H (0)} \mathcal{B}^K_k$ satisfying $|| \xi ||_{HS(0)}  \le    || \mathcal{G}^A (0) ||_{HS(0)} / \lambda_1 = O(k^{-m+2})$. We thus get
\begin{equation} \label{difthsnhiftclth0}
\frac{1}{2} || \cdot ||_{HS(0)} \le || \cdot ||_{HS(\infty)} \le 2 || \cdot ||_{HS(0)}
\end{equation}
for all large enough $k$.

\end{remark}

In particular, since the limit $H (\infty)$ exists, we get $\lim_{t \to \infty} \mathcal{G}^A (t)=0$ from (\ref{bdoprmbaleib0}). Thus, combined with Lemma \ref{lgcptombleacsvd}, we get the following.

\begin{lemma} \label{lemlimgrflggp}
The limit $H_1 := H (\infty)$ of the gradient flow (\ref{grflbmblfokivs}) exists and satisfies $\mathrm{pr}_{\perp , \infty}(\delta \mathcal{Z}^{A} (H_1) ) = 0$ and $\mathrm{pr}_{\lieg}(\delta \mathcal{Z}^{A}( H_1)) = \mathrm{pr}_{\lieg}(\delta Z^{A} (H(0))) $. In other words, the flow (\ref{grflbmblfokivs}) annihilates the $\lieg^{\perp}_t$-component of $\delta \mathcal{Z}^A (H(t))$.
\end{lemma}

This means $\delta (\mathcal{Z}^A (H_1 )) \in \theta_* (\lieg)$, but we can prove the following more precise result.

\begin{lemma}
We have $\delta \mathcal{Z}^A (H_1) \in \theta_* (\ai \mathfrak{z})$ at the limit of the flow $H(t)$.
\end{lemma}

\begin{proof}

Write $G$ for $\mathrm{Aut}_0 (X,L)$ and $\mathfrak{g}$ for its Lie algebra. By Lemma \ref{lemlimgrflggp}, we have $\delta \mathcal{Z}^{A} (H_1) \in \theta_* ( \lieg)$ at the limit $H_1$ of the gradient flow (\ref{grflbmblfokivs}). Suppose that $\delta \mathcal{Z}^{A} (H_1) = \tilde{A}_1 \in \theta_*(\lieg)$. Since $\delta \mathcal{Z}^{A} (H_1)$ is a $K$-invariant hermitian matrix (as $H_1 \in \mathcal{B}_k^K$), $\tilde{A}_1$ must be a $\theta(K)$-invariant hermitian matrix in $\theta_*(\lieg)$. This means that $\theta(f)^* \tilde{A}_1 \theta(f) = \tilde{A}_1$ for any $f \in K$, and hence $\tilde{A}_1$ commutes with any element in $\theta_* (\mathfrak{k})$. Thus $\tilde{A}_1$ is contained in the Lie algebra $\textup{Lie}(Z_G(K))$ of the centraliser $Z_G (K)$ of $K$ in $G$. If $G$ is reductive, we see that $Z_G (K)$ is equal to the complexification of the centre $Z(K)$ of $K$. Thus $\tilde{A}_1 \in \theta_* ( \mathfrak{z} \oplus \ai \mathfrak{z})$, but $\tilde{A}_1$ being hermitian implies $\tilde{A}_1 \in \theta_* (\ai \mathfrak{z})$ by Lemma \ref{rmonadjunilem}. If $G$ is not reductive, we write $\lieg = \left( \mathfrak{k} \oplus \ai \mathfrak{k} \right) \oplus_{\pi} \mathfrak{n}$ where $\mathfrak{n}:= \textup{Lie} (R_u)$ is a nilpotent Lie algebra and $\oplus_{\pi}$ is the semidirect product in the Lie algebra corresponding to $G = K^{\cx} \ltimes R_u$ (cf.~Notation \ref{notgrautliealg}). Since $\tilde{A}_1 = \delta \mathcal{Z}^{A} (H_1)$ is hermitian, Jordan--Chevalley decomposition immediately tells us that the $\mathfrak{n}$-component of $\tilde{A}_1$ is zero, and hence $\tilde{A}_1 \in \theta_*(  \mathfrak{k} \oplus \ai \mathfrak{k})$. Thus, exactly as in the case when $G$ is reductive, $\tilde{A}_1$ commuting with any element in $\theta_* (\mathfrak{k})$ and $\tilde{A}_1$ being hermitian implies $\tilde{A}_1 \in \theta_*( \ai \mathfrak{z})$ by Lemma \ref{rmonadjunilem}.

\end{proof}

Summarising these results, we get the following.

\begin{proposition} \label{proplimgfh1}
At the limit $H_1$ of the gradient flow (\ref{grflbmblfokivs}), we have 
\begin{equation*}
 \int_X  h^k_{FS(H_1)} (s_i ,s_j)  \frac{(k \omega_{FS(H_1)})^n}{n!} = \frac{V k^n}{N}  \left( I + C_0 I - \frac{A}{2 \pi k} \right)^{-1}_{ij} +  \frac{V k^n}{N }   (\tilde{A}_1)_{ij} 
\end{equation*}
where $\frac{V k^n}{N} \tilde{A}_1 \in \theta_* (\ai \mathfrak{z})$ is equal to $-\mathrm{pr}_{\lieg} \left( \delta \mathcal{Z}^A (H(0)) \right)$, and $\{ s_i \}$ is an $H_1$-orthonormal basis. 
\end{proposition}

\subsection{Iterative construction and the completion of the proof of Theorem \ref{sbalmqext}} \label{seciterconsorbalm}

Although Proposition \ref{proplimgfh1} does not provide us with the q-ext metric that we seek, we can use it to construct an iterative procedure which converges to one, as we discuss in the following.


We first need to estimate the Hilbert--Schmidt norm of $\tilde{A}_1$ (in Proposition \ref{proplimgfh1}) in terms of the one of $\tilde{E}_0$.

\begin{lemma} \label{lemdecml2hs}
There exists a constant $C'(R, \epsilon)$ which depends only on $R$ and $\epsilon$ as in Theorem \ref{psestimate} such that
\begin{equation*}
||\tilde{A}_1||_{HS(t)} \le C'(R , \epsilon) k^{1/2} ||\tilde{E}_0||_{HS(t)} ,
\end{equation*}
where $|| \cdot ||_{HS(t)}$ is defined in terms of $H(t)$.
\end{lemma}

\begin{proof}

The equation (5.10) in \cite{ps04}, together with the hypothesis $\bar{\mu}_X' = D_k + E_k$, $D_k$ being a scalar matrix with $D_k \to I$ as $k \to \infty$ and $||E||_{op } < \epsilon$ (cf.~Remark \ref{deccomitodkek}), implies 
\begin{equation*}
||X_{\xi}||^2_{L^2 (t)} \le C (R, \epsilon) ||\xi||_{HS(t)}^2 
\end{equation*}
for any $\xi \in \mathfrak{sl}$ in general. On the other hand, the estimate (\ref{esthsitol2novf}) (cf.~\cite[(5.7)]{ps04}) implies
\begin{equation} \label{bdhsnitol2npsest1}
||\xi||_{HS(t)}^2  \le C'_R k ||X_{\xi}||^2_{L^2(t)} 
\end{equation}
for any $\xi \in \mathfrak{sl}$ in general.

Since $\tilde{E}_0 = - \delta \mathcal{Z}^A  (H(0))$ and $\frac{Vk^n}{N}\tilde{A}_1 = - \mathrm{pr}_{\lieg} \left( \delta \mathcal{Z}^A (H(0)) \right)$, it is sufficient to bound $|| \alpha ||_{HS(t)}$ in the decomposition $\xi = \alpha + \beta$ (according to $\mathfrak{sl} = \lieg \oplus \lieg^{\perp}_t$) in terms of $||\xi||_{HS(t)}$. Then, noting
\begin{equation*}
||X_{\alpha + \beta}||^2_{L^2(t)}=||X_{\alpha} + X_{\beta}||^2_{L^2(t)} = ||X_{\alpha}||^2_{L^2(t)}  + ||X_{\beta}||^2_{L^2(t)} 
\end{equation*}
since $\lieg^{\perp}_t$ is defined with respect to the $L^2$-metric induced from $H (t)$ (cf.~\S \ref{sechessbalen}), we have
\begin{equation} \label{bdhsnitol2npsest2}
||X_{\alpha}||^2_{L^2 (t)} + ||X_{\beta}||^2_{L^2(t)} \le C (R, \epsilon) ||\xi||_{HS(t)}^2 .
\end{equation}
Thus, by (\ref{bdhsnitol2npsest1}) and (\ref{bdhsnitol2npsest2}), there exists a constant $C'(R , \epsilon) >0$ such that
\begin{equation*}
\frac{1}{C' (R , \epsilon) k } \left( ||\alpha||_{HS(t)}^2 + ||\beta||_{HS(t)}^2 \right) \le  ||\xi||_{HS(t)}^2 =  ||\alpha + \beta||_{HS(t)}^2 .
\end{equation*}
which implies $||\alpha||_{HS(t)}^2 \le C' (R , \epsilon) k ||\alpha +\beta||_{HS(t)}^2 \le C' (R , \epsilon) k ||\xi||_{HS(t)}^2$ as required.

\end{proof}

In what follows, we write $|| \cdot ||_{HS,0}$ for the Hilbert--Schmidt norm defined with respect to $H_0$ and $|| \cdot ||_{HS, 1}$ for the one with respect to $H_1$ (which is equal to the limit $H(\infty)$ of the flow (\ref{grflbmblfokivs})).

In particular, Lemma \ref{lemdecml2hs} and (\ref{difthsnhiftclth0}) imply that we have 
\begin{equation} \label{ta1nw0iiton00}
|| \tilde{A}_1||_{HS ,1} \le 2 ||\tilde{A}_1||_{HS , 0} \le 2 C'(R , \epsilon) k^{1/2} ||\tilde{E}_0||_{HS , 0}  = O(k^{-m+\frac{1}{2}}) .
\end{equation}
Now, writing $A_1 = A + 2 \pi k \tilde{A}_1$, we observe
\begin{equation*}
\frac{V k^n}{N}  \left( I + C_0 I - \frac{A}{2 \pi k} \right)^{-1} +  \frac{V k^n}{N }   \tilde{A}_1 = \frac{V k^n}{N} \left(  \left( I + C_0 I - \frac{A_1}{2 \pi k} + \tilde{A}_1 \right)^{-1} +   \tilde{A}_1 \right) 
\end{equation*}
and noting that all the matrices appearing here commute (as $A , \tilde{A}_1 \in \theta_* (\ai \mathfrak{z})$), we have
\begin{align*}
\left( I + C_0 I - \frac{A_1}{2 \pi k} + \tilde{A}_1 \right)^{-1} +   \tilde{A}_1 
&= \left( I + C_0 I - \frac{A_1}{2 \pi k}  \right)^{-1} - \left( I + C_0 I - \frac{A_1}{2 \pi k}  \right)^{-2} \tilde{A}_1 + \tilde{A}_1 \\
&\ \ \ \ \ \ \ \ \ \ \ \ \ \ \ \ \ \ \ \ \ \ \ \ \ \ \ \ \ \ \ \ \ \ \ \ \  + \text{ terms at least quadratic in } \tilde{A}_1 \\
&= \left( I + C_0 I - \frac{A_1}{2 \pi k}  \right)^{-1}  + 2 \left(  C_0 \tilde{A}_1 - \frac{A_1 \tilde{A}_1}{2 \pi k}  \right) \\
&\ \ \ \ \ \ \ \ \ \ \ \ \ \ \ \ \ \ \ \ \ \ \ \ \ \ \ \ \ \ \ \ \ \ \ \ \ + \text{ higher order terms in } k
\end{align*}
by recalling (\ref{bdofc0itfaop}), $||A||_{op } \le \cst$, and (\ref{ta1nw0iiton00}). Now the Hilbert--Schmidt norm of
\begin{align*}
\tilde{E}'_1 &:= \left( I + C_0 I - \frac{A_1}{2 \pi k} + \tilde{A}_1 \right)^{-1} +   \tilde{A}_1 - \left( I + C_0 I - \frac{A_1}{2 \pi k}  \right)^{-1} \\
&=2 \left(  C_0 \tilde{A}_1 - \frac{A_1 \tilde{A}_1}{2 \pi k}  \right) + \text{ higher order terms in } k
\end{align*}
with respect to $H_1$ can be estimated as
\begin{align}
||\tilde{E}'_1||_{HS , 1} &\le 4 \left| \left| C_0 \tilde{A}_1 - \frac{A_1 \tilde{A}_1}{2 \pi k} \right| \right|_{HS ,1} \label{este1bycta1t00} \\
&\le 8 \left| \left| C_0 I - \frac{A}{2 \pi k} \right| \right|_{op } || \tilde{A}_1||_{HS , 1} \notag \\ 
 &\le 8 C' (R , \epsilon) k^{1/2} \left| \left| C_0 I- \frac{A}{2 \pi k} \right| \right|_{op} || \tilde{E}_0 ||_{HS , 1} \notag \\ 
&\le 8 C' (R , \epsilon)  \frac{||A||_{op }+1}{k^{1/2}} || \tilde{E}_0 ||_{HS , 0} = O(k^{-m-\frac{1}{2}}) \label{este1bycta1t3}, 
\end{align}
for all large enough $k$, by recalling the estimate (\ref{bdofc0itfaop}),  $||A||_{op } \le \cst$, and (\ref{difthsnhiftclth0}); we also used (cf.~Proposition \ref{proplimgfh1}) $|| \tilde{A}_1 ||^2_{op} \le \sqrt{\mathrm{tr} ( \tilde{A}_1  \tilde{A}_1)} = ||\tilde{A}_1||_{HS,1}$. We then modify the constant $C_0$ to make $\tilde{E}_1$ term trace free, by arguing as we did in Lemma  \ref{lexpformcstcabtr}. This will change $C_0$ by a constant of order $k^{-m- \frac{1}{2} - \frac{n}{2}}$, to $C_1$ say, satisfying the bound
\begin{equation} \label{estc0c1itote1}
|C_0 - C_1| < 4 N^{-1/2} ||\tilde{E}'_1||_{HS , 1} \le 8 N^{-1/2} || \tilde{E}'_1||_{HS , 0}
\end{equation}
as in (\ref{estondeltaitfedp}). Hence there exists a trace free hermitian matrix $\tilde{E}_1$ which satisfies
\begin{equation*}
 \frac{V k^n}{N}  \left( I + C_0I - \frac{A_1}{2 \pi k} \right)^{-1} +  \frac{V k^n}{N }   \tilde{E}'_1= \frac{V k^n}{N }  \left(I + C_1I - \frac{A_1}{2 \pi k} \right)^{-1} +  \tilde{E}_1 
\end{equation*}
and $||\tilde{E}_1||_{HS , 1}$ can be bounded by
\begin{align}
||\tilde{E}_1||_{HS , 1}  &\le \left(  1 +  8  \left| \left| \left(I + C_0  I- \frac{A_1}{2 \pi k} \right)^{-1} \right| \right|_{op }  \right) \frac{Vk^n}{N} ||\tilde{E}'_1||_{HS , 1} \label{00estoe1itfe01} \\
&\le 4  \left(  1 +  8  \left| \left| \left(I + C_0  I- \frac{A_1}{2 \pi k} \right)^{-1} \right| \right|_{op }  \right) \left| \left| C_0 I- \frac{A_1}{2 \pi k} \right| \right|_{op } || \tilde{A}_1||_{HS , 1} \label{0estoe1itfe01} \\
&\le 16 C'(R , \epsilon) k^{1/2} \left(  1 +  8  \left| \left| \left(I + C_0  I- \frac{A}{2 \pi k} \right)^{-1} \right| \right|_{op }  \right) \left| \left| C_0 I- \frac{A}{2 \pi k} \right| \right|_{op } || \tilde{E}_0||_{HS , 1} \notag \\
&\le 32 C'(R , \epsilon) k^{1/2} \left(  1 +  8  \left| \left| \left(I + C_0  I- \frac{A}{2 \pi k} \right)^{-1} \right| \right|_{op }  \right) \frac{||A||_{op } +1}{k} || \tilde{E}_0||_{HS , 0}  \label{estoe1itfe0} \\
&= O(k^{-m-\frac{1}{2}}) \label{este1itfat1}
\end{align}
where we used (\ref{estoftfedpitfedp}) in the first line, (\ref{este1bycta1t00}) in the second line, Lemma \ref{lemdecml2hs} and (\ref{ta1nw0iiton00}) in the third line, and (\ref{bdofc0itfaop}) in the fourth line.





Recalling Proposition \ref{proplimgfh1}, the above calculations mean that we get
\begin{equation*}
\int_X  h^k_{FS(H_1)} ( s^{H_1}_i ,  s^{H_1}_j)  \frac{ (k\omega_{FS(H_1)})^n}{n!}  -  \frac{Vk^n}{N} \left( I + C_1 I - \frac{A_1}{2 \pi k} \right)_{ij}^{-1} = (\tilde{E}_1)_{ij}
\end{equation*}
where $\tilde{E}_1 \in T_{H_1} \mathcal{B}_k^K$ is a trace free hermitian matrix which satisfies $|| \tilde{E}_1||_{HS , 1} = O(k^{-m - \frac{1}{2}})$ by (\ref{este1itfat1}).
We now return to the gradient flow (\ref{grflbmblfokivs}), starting at $H_1$, apart from that it is now driven by $\mathrm{pr}_{\perp , t} (\delta \mathcal{Z}^{A_1} (H(t)))$, replacing $A$ by $A_1$; namely we run the new gradient flow
\begin{equation} \label{gradflbmblfokvis1}
\frac{d H^{(1)} (t)}{dt}  =  -  \text{pr}_{\perp , t} \left( \delta \mathcal{Z}^{A_1} (H^{(1)} (t))\right)
\end{equation}
starting at $H (\infty)  := H_1$, where we observe that the error term $\tilde{E}_1$ (at $t=0$) has now been improved to $|| \tilde{E}_1 ||_{HS , 1} = O(k^{-m-\frac{1}{2}})$ by (\ref{este1itfat1}), as opposed to $|| \tilde{E}_0 ||_{HS , 0} = O(k^{-m})$ that we initially had in Corollary \ref{capproxsbal}. Also note that now the projection $\mathrm{pr}_{\perp , t} : \mathfrak{sl} \surj \lieg^{\perp}_t$ is onto the $L^2$-orthogonal complement of $\lieg$ in $\mathfrak{sl}$ with respect to the Fubini--Study metric induced from $H^{(1)}(t)$.


We summarise what we have achieved as follows. We started with an approximately q-ext metric $H_0 \in \mathcal{B}_k^K$, obtained in Corollary \ref{capproxsbal} which satisfies $\delta \mathcal{Z}^A (H_0) = \tilde{E}_0$ with $||\tilde{E}_0||_{HS, 0} \le \cst . k^{-m}$; ran the gradient flow (\ref{grflbmblfokivs}) to annihilate $\mathrm{pr}_{\perp , t} (\delta \mathcal{Z}^{A})$, so that at the limit $H_1 \in \mathcal{B}_k^K$ of the flow we have $\mathrm{pr}_{\perp, \infty} (\delta \mathcal{Z}^{A} (H_1)) =0$; set $\tilde{A}_1:= - \frac{N}{Vk^n} \mathrm{pr}_{\lieg}(\delta \mathcal{Z}^{A} (H_1) ) \in \theta_* (\ai \mathfrak{z})$ and replaced $A$ by $A_1: = A + 2 \pi k \tilde{A}_1$, to consider the functional $\mathcal{Z}^{A_1}$ with a new constant $C_1$, which differs from $C_0$ by $O(k^{-m- \frac{1}{2} - \frac{n}{2}})$; wrote $\tilde{E}_1 := - \delta \mathcal{Z}^{A_1} (H_1) $ with $\tilde{E}_1$ satisfying $|| \tilde{E}_1 ||_{HS , 1} \le \cst. k^{-1/2} ||\tilde{E}_0||_{HS, 0} = O(k^{-m-\frac{1}{2}})$ as given in (\ref{estoe1itfe0}), i.e.~$H_1$ is an approximately q-ext metric of order $k^{-m-\frac{1}{2}}$. We then go back to the first step, by replacing $H_0$ with $H_1$. We repeat the above process inductively, as in the following proposition.

\begin{proposition} \label{propiterschrbalm}
Suppose that we run the iterative procedure, starting with $i=0$, to find q-ext metrics as follows:
\begin{description}
\item[Step 1] start with an approximately q-ext metric $H_i \in \mathcal{B}_k^K$ of order $k^{-m- i/2}$;
\item[Step 2] run the gradient flow 
\begin{equation*}
\frac{d H^{(i)} (t)}{dt} =  -  \mathrm{pr}_{\perp , t} \left( \delta \mathcal{Z}^{A_i} (H^{(i)} (t)) \right)
\end{equation*}
to annihilate $\mathrm{pr}_{\perp , t }(\delta \mathcal{Z}^{A_i})$, so that at the limit $H^{(i)} (\infty) =: H_{i+1} \in \mathcal{B}_k^K$ of the flow we have 
\begin{equation*}
\mathrm{pr}_{\perp ,\infty}(\delta \mathcal{Z}^{A_i} (H_{i+1}) )=0 ;
\end{equation*}
\item[Step 3] set $\tilde{A}_{i+1}:= - \frac{N}{Vk^n} \mathrm{pr}_{\lieg}(\delta \mathcal{Z}^{A_i} (H_{i+1}) ) \in \theta_* (\ai \mathfrak{z})$ and replace $A_i$ by $A_{i+1}: = A_i + 2 \pi k \tilde{A}_{i+1}$, to consider the functional $\mathcal{Z}^{A_{i+1}}$ with a new constant $C_{i+1}$, which differs from $C_{i}$ by $O(k^{-m - (n+i)/2 })$;
\item[Step 4] observe that $H_{i+1}$ satisfies $|| \delta \mathcal{Z}^{A_{i+1}} (H_{i+1}) ||_{HS , i+1} = O(k^{-m- (i+1)/2})$, where $|| \cdot ||_{HS , i+1}$ is the Hilbert--Schmidt norm defined with respect to $H_{i+1}$ i.e. $H_{i+1}$ is an approximately q-ext metric of order $k^{-m- (i+1)/2}$;
\item[Step 5] go back to the step 1, with an improved error term (i.e.~the approximately q-ext metric $H_{i+1}$ now has order $k^{-m-(i+1)/2}$);
\end{description}
so that, by repeating these steps, we get a sequence $\{ (A_i , C_i ,  H_i) \}_i$ in $\theta_* (\ai \mathfrak{z}) \times \rl \times \mathcal{B}_k^K$.

Then, as $i \to \infty$, $A_i$, $C_i$, and $H_i$ converges to $A_{\infty} \in \theta_* (\ai \mathfrak{z})$, $C_{\infty} \in \rl$, and $H_{\infty} \in \mathcal{B}^K_k$, respectively.
\end{proposition}

The proof is given in the following two lemmas, which rely on the estimates that we have established so far. We first prove the existence of $A_{\infty}$ and $C_{\infty}$.

\begin{lemma} \label{convoftildeai}
$A/k + 2 \pi \tilde{A}_1 + 2 \pi \tilde{A}_2 + \cdots$ converges, and hence $A_{\infty} := A + 2 \pi k \tilde{A}_1 + 2 \pi k \tilde{A}_2 + \cdots$ exists. Also $C_{\infty}$ exists.
\end{lemma}

\begin{proof}
We first claim that there exist some constants $\gamma_1 , \gamma_2 >0$ such that $|| \tilde{A}_i||_{HS , i} \le k^{-m+1} (k^{- 1/2} \gamma_1 )^i$ and $| C_i - C_{i-1}| \le N^{-1/2} k^{-m } (k^{-1/2} \gamma_2)^i$. Observe that $|| \tilde{A}_i||_{HS , i} \le k^{-m+1} (k^{- 1/2} \gamma_1 )^i$ implies $|| \tilde{A}_i||_{HS, 0} \le k^{-m+1} (2k^{- 1/2} \gamma_1 )^i$ by inductively using (\ref{difthsnhiftclth0}). 

Note that these estimates are satisfied when $i=1$; more specifically, Lemma \ref{lemdecml2hs}, (\ref{difthsnhiftclth0}), and $||\tilde{E}_0||_{HS , 0} = O(k^{-m})$ imply that there exists a constant $\gamma >0$ such that
\begin{equation*}
||\tilde{A}_1||_{HS, 1} \le 2 C'(R , \epsilon) k^{1/2} ||\tilde{E}_0||_{HS, 0} \le \gamma C'(R , \epsilon) k^{-m+\frac{1}{2}},
\end{equation*}
and (\ref{este1bycta1t3}), (\ref{estc0c1itote1}) imply
\begin{equation*}
|C_0 - C_1| \le 4 N^{-1/2} 8C' (R , \epsilon) \frac{||A||_{op} +1}{k^{1/2}}  ||\tilde{E}_0||_{HS , 0} \le 32 N^{-1/2}  \gamma C' (R , \epsilon) (||A||_{op} +1) k^{-m - \frac{1}{2}}  .
\end{equation*}
In what follows, we assume $C' (R , \epsilon) \ge 1$ and $\gamma \ge 1$ without loss of generality.

We argue by induction; suppose that the statement holds at the $(i-1)$-th step. Combined with Lemma \ref{lemdecml2hs} and (\ref{difthsnhiftclth0}), the argument in (\ref{0estoe1itfe01}) at the $i$-th step implies
\begin{align*}
||\tilde{A}_i||_{HS , i} &\le C' (R , \epsilon) k^{1/2} ||\tilde{E}_{i-1}||_{HS , i} \\
&\le 8 C' (R , \epsilon) k^{1/2} \left(  1 +  8  \left| \left| \left(I + C_{i-2}  I- \frac{A_{i-1}}{2 \pi k} \right)^{-1} \right| \right|_{op }  \right) \left| \left| C_{i-2} I- \frac{A_{i-1}}{2 \pi k} \right| \right|_{op} || \tilde{A}_{i-1}||_{HS , i-1} 
\end{align*}
for all $i \ge 2$. Then the induction hypothesis and (\ref{bdofc0itfaop}) imply 
\begin{equation*}
1 +  8  \left| \left| \left(I + C_{i-2}  I- \frac{A_{i-1}}{2 \pi k} \right)^{-1} \right| \right|_{op }  \le 2 \left(  1 +  8  \left| \left| \left(I + C_{0}  I- \frac{A}{2 \pi k} \right)^{-1} \right| \right|_{op }  \right)
\end{equation*}
and $\left| \left| C_{i-2} I- \frac{A_{i-1}}{2 \pi k} \right| \right|_{op} \le \frac{||A||_{op} +1}{k}$ (cf.~(\ref{estoe1itfe0})). Thus
\begin{equation} \label{bdotiaiitftaim1}
||\tilde{A}_i||_{HS , i} \le 16  C' (R , \epsilon)  \left(  1 +  8  \left| \left| \left(I + C_{0}  I- \frac{A}{2 \pi k} \right)^{-1} \right| \right|_{op }  \right)  \frac{||A||_{op} +1}{k^{1/2}}||\tilde{A}_{i-1}||_{HS , i-1}
\end{equation}
for all large enough $k$. We can thus take 
\begin{equation*}
\gamma_1 :=\max \left\{ 16 C' (R , \epsilon)  \left(  1 +  8  \left| \left| \left(I + C_{0}  I- \frac{A}{2 \pi k} \right)^{-1} \right| \right|_{op }  \right)  (||A||_{op} +1) , \ \ \gamma C' (R , \epsilon) \right\}.
\end{equation*}
We also have
\begin{equation*}
|C_i - C_{i-1}| \le 4 N^{-1/2} ||\tilde{E}'_i||_{HS , i} 
\le 16 N^{-1/2} \left| \left| C_{i-1} - \frac{A_i }{2 \pi k} \right| \right|_{op} ||\tilde{A}_i||_{HS , i} 
\end{equation*}
by arguing as in (\ref{este1bycta1t00}) and (\ref{estc0c1itote1}). The induction hypothesis and (\ref{bdotiaiitftaim1}) imply that
\begin{equation*}
|C_i - C_{i-1}|  \le 32 N^{-1/2} \left| \left| C_{0} - \frac{A }{2 \pi k} \right| \right|_{op } \gamma_1^i k^{-m+1-\frac{i}{2}} \le 16 N^{-1/2} (\left| \left| A \right| \right|_{op } +1) \gamma_1^i k^{-m-\frac{i}{2}}
\end{equation*}
where we used (\ref{bdofc0itfaop}) and $||A||_{op } \le \cst$, and hence we can take $\gamma_2 := 16  C' (R , \epsilon) (\left| \left| A \right| \right|_{op } +1) \gamma_1$, by noting $16  C' (R , \epsilon) (\left| \left| A \right| \right|_{op } +1) \gamma_1^i < \gamma_2^i$.


Having established the claim as above, we thus have
\begin{equation*}
|| A/k + 2 \pi  \tilde{A}_1 + 2 \pi  \tilde{A}_2 + \cdots ||_{HS , 0} \le \left( \frac{\gamma_1}{k} + k^{-m+1} (2 \gamma_1 k^{-1/2}) + k^{-m+1} (2 \gamma_1 k^{-1/2})^2 + \cdots  \right) < \infty
\end{equation*}
for all large enough $k$, and
\begin{equation*}
| C_{\infty}| \le   \left( \frac{\gamma_2}{k} + N^{-1/2} k^{-m} (\gamma_2 k^{-1/2}) + N^{-1/2} k^{-m} (\gamma_2 k^{-1/2})^2 + \cdots  \right)   < \infty .
\end{equation*}

\end{proof}

We now prove the existence of $H_{\infty}$.

\begin{lemma}
Repeating the procedure as given in Proposition \ref{propiterschrbalm} infinitely many times moves $H_0$ by a finite distance in $\mathcal{B}_k^K$ with respect to the bi-invariant metric, i.e.~$\mathrm{dist} (H_{\infty} , H_0) < \infty$.
\end{lemma}

\begin{proof}

Consider first the case $i=1$. Recall that we use the limit $H_1 = H(\infty)$ of the first gradient flow (\ref{grflbmblfokivs}) as the initial condition for the second gradient flow (\ref{gradflbmblfokvis1}). By proceeding as we did in (\ref{estlimgrflza}), we get
\begin{equation*}
\mathrm{dist} (H_2 , H_0) \le \frac{1}{\lambda_1} \left( ||\mathcal{G}^A (0)||_{HS , 0} + ||\mathcal{G}^{A_1} (0)||_{HS ,1} \right)  .
\end{equation*}

Recalling $\mathcal{G}^A (0) = \tilde{E}_0$ and $\mathcal{G}^{A_1} (0)  = \tilde{E}_1$, we get $\mathrm{dist} (H_2 , H_0) \le \frac{1}{\lambda_1} \left( ||\tilde{E}_0||_{HS , 0} + ||\tilde{E}_1||_{HS , 1} \right)$. 
Inductively continuing as described in Proposition \ref{propiterschrbalm}, we have $\mathrm{dist} (H_{i+1} , H_j) \le \frac{1}{\lambda_1} \left( ||\tilde{E}_{j}||_{HS , j} + \cdots + ||\tilde{E}_i||_{HS , i} \right)$ 
for $i > j$, and also
\begin{equation} \label{eststepihmmgf}
\mathrm{dist} (H_{i+1} , H_0) \le \frac{1}{\lambda_1} \left( ||\tilde{E}_0||_{HS , 0} + ||\tilde{E}_1||_{HS , 1} + \cdots + ||\tilde{E}_i||_{HS , i} \right) . 
\end{equation}
Now the estimates as in (\ref{00estoe1itfe01})-(\ref{estoe1itfe0}) at the $i$-th step (and also Lemma \ref{convoftildeai}) implies that we have
\begin{equation*}
||\tilde{E}_i||_{HS , i} \le 32 C'(R , \epsilon)  \left(  1 +  8  \left| \left| \left(I + C_{0}  I-  \frac{A}{2 \pi k} \right)^{-1} \right| \right|_{op}  \right) \frac{||A||_{op} +1}{k^{1/2}} || \tilde{E}_{i-1}||_{HS , i-1} 
\end{equation*}
and hence
\begin{equation} \label{esteiitoi0indali}
||\tilde{E}_i||_{HS , i}  \le \left( 32 C'(R , \epsilon)  \left(  1 +  8  \left| \left| \left(I + C_0  I-  \frac{A}{2 \pi k} \right)^{-1} \right| \right|_{op}  \right) \frac{||A||_{op} +1}{k^{1/2}} \right)^{i} || \tilde{E}_{0}||_{HS , 0}.
\end{equation}

Thus we find that there exists a constant $c>0$, independent of $k$, such that $||\tilde{E}_i||_{HS , i} \le (ck^{-1/2})^{i+1} ||\tilde{E}_0||_{HS , 0}$, and hence we get
\begin{equation*} 
\mathrm{dist} (H_{i+1} , H_j)  \le k^{2} ||\tilde{E}_0||_{HS , 0} \left( (ck^{-1/2})^{j} + \cdots + (ck^{-1/2})^{i} \right)
\end{equation*}
for $i > j$, and
\begin{align}
&\mathrm{dist} (H_{i+1} , H_0) \notag \\
&\le \frac{1}{\lambda_1} \left( ||\tilde{E}_0||_{HS , 0} + c k^{-1/2}||\tilde{E}_0||_{HS , 0} + c^2k^{-1} ||\tilde{E}_0||_{HS , 0} + \cdots + + (ck^{-1/2})^{i} ||\tilde{E}_0||_{HS , 0} \right) \notag \\
&= O(k^{-m+2}) \label{estofiinfinfiot0}
\end{align}
for all large enough $i$, where we recall $|| \tilde{E}_0||_{HS , 0} = O(k^{-m})$ (cf.~Corollary \ref{capproxsbal}) and (\ref{estlevhobekm2}). Thus the sequence $\{ H_i \}_i$ is Cauchy in $\mathcal{B}^K_k$ with respect to the bi-invariant metric, and hence the limit $H_{\infty} \in \mathcal{B}^K_k$ exists.



\end{proof}


We finally see that (\ref{estofiinfinfiot0}) implies $||H_{\infty} - H_0 ||_{HS, 0} = O(k^{-m+2})$ (cf.~Remark \ref{remdifthsnhiftclth0}). We claim $ || \omega_{H_{\infty}} - \omega_{(m)} ||_{C^l , \omega} = O(k^{-m+ n+l+1})$, recalling the definitional $\omega_{H_0} = \omega_{(m)}$. To make explicit the dependence on $k$ and $m$, we write $H_{\infty} (k, m)$ for $H_{\infty} \in \mathcal{B}_k^K$ and $H_0 (k,m )$ for $H_0 \in \mathcal{B}_k^K$. By taking a suitable $H_0 (k,m)$-orthonormal basis $\{ s_i \}$, we may assume that $H_0(k ,m)$ is the identity matrix and $H_{\infty}(k ,m)$ is given by $\mathrm{diag} (d^2_1 , \dots , d^2_{N})$. $||H_{\infty}(k, m) - H_0(k, m) ||_{HS, 0} = O(k^{-m+2})$ implies that we have $ d_i^2- 1  = O(k^{-m+2})$, which in turn implies $ d_i^{-2}- 1  = O(k^{-m+2})$. Observe that we can write 
\begin{equation*}
\omega_{H_{\infty} (k,m)} = \omega_{(m)} + \frac{\ai}{2 \pi k} \ddbar \log \left( \sum_i d_i^{-2} |s_i|^2_{FS(H_0 (k,m))^k} \right) .
\end{equation*}
We may choose local coordinates $(z_1 , \dots , z_n)$ and reduce to local computation. The equation (\ref{defoffseq}) and $ d_i^{-2}- 1  = O(k^{-m+2})$ imply that we have $\sum_i d_i^{-2} |s_i|^2_{FS(H_0 (k,m))^k} = 1 + O (k^{-m+n+2})$, and hence it suffices to evaluate its derivatives.

We fix a local trivialisation of the line bundle $L$ to write $h_{FS(H_0 (k,m))} = e^{- \phi_{m,k}}$, and regard each $s_i$ as a holomorphic function. Observe that (\ref{defoffseq}) implies $\sum_i |s_i|^2 = e^{k \phi_{m,k}}$. We then apply $\frac{\partial^2}{\partial z_j \partial \bar{z}_j}$ on both sides to find
\begin{equation*}
\sum_i e^{-k \phi_{m,k}} \left| \frac{\partial}{\partial z_j} s_i \right|^2  \le  k^2 C_1(\phi_{m,k}) , 
\end{equation*}
for a constant $C_1 (\phi_{m,k})$ which depends only on (first and second derivatives of) $\phi_{m,k}$. Higher order derivatives can be similarly bounded in terms of $C^l$-norms of $\phi_{m,k}$; namely we get $\sum_i e^{-k \phi_{m,k}} \left| \frac{\partial^l}{\partial z_{j_1} \cdots \partial z_{j_l}} s_i \right|^2  \le  k^{2l}  C_2 (\phi_{m,k} , l)$ for a constant $C_2 (\phi_{m,k}, l)$ which depends only on the $C^{2l}$-norm of $\phi_{m,k}$. In particular, we have
\begin{equation*}
e^{-k \phi_{m,k} / 2} \left| \frac{\partial^l}{\partial z_{j_1} \cdots \partial z_{j_l}} s_i \right|  \le  k^l C_3 (\phi_{m,k} ,l ) 
\end{equation*}
for each $i = 1, \dots , N$ and $ j_1 , \dots , j_l \in \{ 1 , \dots , n \}$.

Observe that (\ref{defoffseq}) implies $e^{- k \phi_{m,k} / 2} |s_i| \le 1$. Thus we get, again using (\ref{defoffseq}),
\begin{align*}
\frac{\partial}{\partial z_j}  \sum_i d_i^{-2} |s_i|^2_{FS(H_0 (k,m))^k} &= \frac{\partial}{\partial z_j}  \sum_i (d_i^{-2} -1) e^{- k \phi_{m,k}} |s_i|^2 \\
&= -k \frac{\partial \phi_{m,k}}{\partial z_j} \sum_i (d_i^{-2} -1) e^{- k \phi_{m,k}} |s_i|^2 + \sum_i (d_i^{-2} -1) e^{- k \phi_{m,k}} \bar{s}_i \frac{\partial s_i}{\partial z_j} ,
\end{align*}
and hence
\begin{equation*}
\left| \frac{\partial}{\partial z_j}  \sum_i d_i^{-2} |s_i|^2_{FS(H_0 (k,m))^k} \right| \le  k^{-m+n+3} C_4 (\phi_{m,k}) .
\end{equation*}
Thus, inductively continuing, we get
\begin{equation*}
\left|   \frac{\partial^r}{\partial \bar{z}_{j_1} \cdots \partial \bar{z}_{j_r}} \frac{\partial^l}{\partial z_{j_1} \cdots \partial z_{j_l}}  \sum_i d_i^{-2} |s_i|^2_{FS(H_0 (k,m))^k} \right| \le  k^{-m+n+l+r+2} C_5 (\phi_{m,k} , l+r) .
\end{equation*}
Thus we get $|| \omega_{H_{\infty}} - \omega_{(m)} ||_{C^l , \omega_{H_0 (k,m)}} \le  C_6 (\phi_{m,k} , l) k^{-m+n+(l+2)-1}$.

Writing $h = e^{- \phi}$ for the hermitian metric corresponding to the extremal metric $\omega$ (i.e.~$\omega= - \ai \ddbar \log h$), we have $\phi_{m,k} \to \phi$ in $C^{\infty}$ as $k \to \infty$ (cf.~the proof of Corollary \ref{coraprbalmsolprob}). Thus we get $|| \omega_{H_{\infty}} - \omega_{(m)} ||_{C^l , \omega_{H_0 (k,m)}} \le  C_l k^{-m+n+(l+2)-1}$ for a constant $C_l$ which depends only on $l$, as claimed.




We thus get
\begin{align} 
 || \omega_{H_{\infty} (k,m)} -  \omega ||_{C^l , \omega} &\le 2 ||  \omega_{H_{\infty} (k,m)} -  \omega_{H_0 (k,m)} ||_{C^l , \omega_{H_0 (k,m)}} + ||  \omega_{H_{0} (k,m)} -  \omega ||_{C^l , \omega} \notag \\
 &\le \tilde{C}_l  ( k^{-m+ n+l+1} + k^{-1}) . \label{diffohokmlkobd}
\end{align}
Thus, given $l \in \mathbb{N}$, we can choose $m$ to be large enough so that the sequence $\{ \omega_{H_{\infty} (k,m)} \}_k$ converges to $\omega$ in $C^l$, establishing all the statements claimed in Theorem \ref{sbalmqext}.

\begin{remark} \label{convclcinftydiagarg}
It is tempting to say that, given such $ \omega_{H_{\infty} (k,m)}$'s, there exists a sequence $\{ \omega_k \}_k$ which converges to $\omega$ in $C^{\infty}$ by diagonal argument. However, $k$ must be chosen to be large enough for $ \omega_{H_{\infty} (k,m)}$ to be well-defined, and how large $k$ must be depends on $m$ (cf.~\S \ref{approximately}), and hence on $l$. Thus, by diagonal argument, we can only claim the existence of $\omega_k$'s (with $\omega_k \to \omega$ in $C^{\infty}$) satisfying $\bar{\partial} \mathrm{grad}^{1,0}_{\omega_k} \rho_k (\omega_k)=0$ for \textit{infinitely many} $k$'s rather than for \textit{all sufficiently large} $k$'s.

We finally note that, if we have the uniqueness theorem as mentioned in Remark \ref{remconjfunitfrbmet}, it follows that $ \omega_{H_{\infty} (k,m)} =  \omega_{H_{\infty} (k,m')}$ for all $m$ and $m'$, and hence we can say that the sequence converges in $C^{\infty}$ (cf.~\cite[\S 4.2]{donproj1}).

\end{remark}

\section{Relationship to recent related works} \label{reltrecrelwks}
After the appearance of the first version of this paper, several papers \cite{mab2016,santip17,seyrel} proving similar results have appeared. In this section we discuss the subtle yet nontrivial differences among them.

Recall first that the existence of usual balanced metric for $(X,L^k)$ can be characterised by four equivalent conditions as follows.

\begin{description}
	\item[1] Chow stability of $(X,L^k)$,
	\item[2] the existence of a hermitian metric $h$ such that $FS(Hilb(h))= h$,
	\item[3] the existence of a \kah metric $\omega$ such that $\rho_k (\omega)  = \cst$,
	\item[4] the centre of mass is a constant multiple of the identity.
\end{description}

When $\mathrm{Aut}_0(X,L)$ is nontrivial, each of the first three conditions has a natural generalisation as follows.

\begin{description}
	\item[1'] Chow stability can be generalised to \textit{relative} Chow stability \cite{ah, mab04ext, mab11},
	\item[2'] we seek $h$ such that $FS(Hilb(h))= \sigma^* h$ for some $\sigma \in \mathrm{Aut}_0(X,L)$ \cite{santip},
	\item[3'] we seek $\omega$ such that $\bar{\partial} \mathrm{grad}^{1,0}_{\omega} \rho_k (\omega) = 0$.
\end{description}

Unlike the discrete automorphism case, the equivalence of \textbf{1'-3'} above is currently unknown when $\mathrm{Aut}_0(X,L)$ is nontrivial. In terms of the centre of mass, each of above conditions can be characterised as follows.

\begin{description}
	\item[4-1'] $(X,L^k)$ is relatively Chow stable if and only if there exists a basis $\{Z'_i\}_i$ for $H^0 (X,L^k)$ such that the centre of mass $\bar{\mu}'_X$ defines an element of $\mathfrak{aut}(X,L)$,
	\item[4-2'] there exists a $\sigma$-balanced metric if and only if there exists a basis such that the centre of mass $\bar{\mu}'_X$ defines an element of the group $\mathrm{Aut}_0(X,L)$,
	\item[4-3'] there exists $\omega$ with $\bar{\partial} \mathrm{grad}^{1,0}_{\omega} \rho_k (\omega) = 0$ if and only if there exists a basis such that the inverse $(\bar{\mu}'_X)^{-1}$ of the centre of mass defines an element of $\mathfrak{aut}(X,L)$.
\end{description}

On the other hand, however, the existence of balanced metrics for all large enough $k$ has been established for each definition of ``balanced metrics'' as in \textbf{1'-3'}, by assuming that $(X,L)$ admits an extremal metric. Namely, we have the following result.

\begin{theorem}
	Suppose that $(X,L)$ admits an extremal metric. Then
	\begin{enumerate}
	\item Seyyedali \cite{seyrel} and Mabuchi \cite{mab2016} proved that $(X,L^k)$ is relatively Chow polystable for all large enough $k$,
	\item Sano--Tipler \cite{santip17} proved that for all large enough $k$ there exists a $\sigma$-balanced metric,
	\item see Theorem \ref{sbalmqext}.
	\end{enumerate}
\end{theorem}

As each different characterisation \textbf{1-4} of the usual balanced metric has its own merit, each different definition \textbf{1'-3'} of ``generalised'' balanced metrics has its own merit. \textbf{1'} establishes connection to relative Chow stability, and \textbf{2'} can be used to study the lower bound of the modified $K$-energy \cite{santip}. A merit of \textbf{3'} is that the holomorphic vector field defined by $(\bar{\mu}'_X)^{-1}$ has an explicit geometric interpretation, namely as the one generated by the Bergman function; geometric interpretation of the elements of $\mathfrak{aut}(X,L)$ or $\mathrm{Aut}_0 (X,L)$ in \textbf{1'} and \textbf{2'} does not seem obvious yet (although the automorphism in \textbf{2'} can be characterised as a zero of a Futaki-type invariant \cite{santip17}). 

In fact, \textbf{3'} further gives an alternative proof of the fact \cite[Theorem 2.3]{stosze} that $(X,L)$ admitting an extremal metric is $K$-semistable relative to the extremal $\cx^*$-action, as a consequence of Theorem \ref{sbalmqext}. Moreover we can prove that the existence of balanced metric in the sense of \textbf{1'-3'} implies that $(X,L^k)$ is \textit{weakly} Chow stable. The details of these results are presented in \cite{yhstability}.


\appendix

\section{Some results on the Lichnerowicz operator used in \S \ref{asymptapprox}}
\label{chapterappendix}

\begin{lemma} \label{lich1}
For any $F \in C^{\infty} (X , \rl)$, there exists $F_1 \in C^{\infty} (X , \rl)$, $F_2 \in C^{\infty} (X , \rl)$ such that $\lich{\omega} F_1 = F + F_2$ with $\lich{\omega} F_2=0$. Moreover, writing $\textup{pr}_{\omega}: C^{\infty} (X , \rl) \surj \ker \lich{\omega}$ by recalling the $L^2$-orthogonal direct sum decomposition $C^{\infty} (X, \rl) \cong \textup{im}\lich{\omega} \oplus \ker \lich{\omega}$, $F_2$ is in fact $F_2 = - \textup{pr}_{\omega} (F)$. 
\end{lemma}

\begin{proof}
This is a well-known result, which follows from the self-adjointness and the elliptic regularity of $\lich{\omega}$.
\end{proof}

\begin{lemma} \label{projconv}
Let $\{ F_{k} \}$ be a family of smooth functions parametrised by $k$, converging to a smooth function $F_{\infty}$ in $C^{\infty}$ as $k \to \infty$, and $(\phi_{1,k} , \dots , \phi_{m,k})$ be smooth functions, each of which converges to a smooth function $\phi_{i ,\infty}$ as $k \to \infty$. Write $\omega_{(m)} := \omega + \ai \ddbar (\sum_{i=1}^m \phi_{i,k} / k^i)$. Let $\textup{pr}_{\omega} : C^{\infty} (X , \rl) \to \ker \lich{\omega}$ and $\textup{pr}_{(m)} : C^{\infty} (X , \rl) \to \ker \lich{(m)}$ be the projection to $\ker \lich{\omega}$ and $\ker \lich{(m)}$, respectively. Then, $\textup{pr}_{(m)} F_k$ converges to $\textup{pr}_{\omega} F_{\infty}$ in $C^{\infty}$.
\end{lemma}

\begin{proof}
Note that we can write $\lich{(m)} = \lich{\omega} + D/k$ for some differential operator $D$ of order at most 4, which depends on $\omega$ and $(\phi_{1,k} , \dots , \phi_{m,k})$. Since we know that each $\phi_{i,k}$ converges to a smooth function $\phi_{i ,\infty}$ in $C^{\infty}$, the operator norm of $D$ can be controlled by a constant which depends only on $\omega$ and $(\phi_{1,\infty} , \dots , \phi_{m,\infty})$ but not on $k$. Thus, $|| \textup{pr}_{(m)} F - \textup{pr}_{\omega}F ||_{C^{\infty}} \to 0$ for any fixed $F \in C^{\infty} (X, \rl)$ as $k \to \infty$. On the other hand, $|| \textup{pr}_{(m)} F_k - \textup{pr}_{(m)} F_{\infty} ||_{C^{\infty}} \to 0$ since $F_k$ converges to $F_{\infty}$ in $C^{\infty}$. Combining these estimates,
\begin{equation*}
|| \textup{pr}_{(m)} F_k - \textup{pr}_{\omega} F_{\infty} ||_{C^{\infty}} \le || \textup{pr}_{(m)} ( F_k -  F_{\infty} ) ||_{C^{\infty}} + || \textup{pr}_{(m)} F_{\infty} - \textup{pr}_{\omega} F_{\infty} ||_{C^{\infty}} \to 0
\end{equation*}
as $k \to \infty$.
\end{proof}

\begin{lemma} \label{lemconv}
Suppose that the following four conditions hold for an arbitrary but fixed $m \ge 1$.
\begin{enumerate}
\item $(\phi_{1,k} , \dots, \phi_{m,k})$ are smooth functions parametrised by $k$ such that each $\phi_{i,k}$ converges to a smooth function $\phi_{i ,\infty}$ in $C^{\infty}$ as $k \to \infty$, so that $\omega_{(m)}:= \omega + \ai \ddbar ( \sum_{i=1}^m \phi_{i,k}/k^i)$ converges to $\omega$ in $C^{\infty}$,
\item $\{ G_k \}$ is a family of smooth functions on $X$ parametrised by $k$ such that it converges to a smooth function $G_{\infty}$ in $C^{\infty}$ as $k \to \infty$,
\item $\{ F_k \}$ is another family of smooth functions on $X$ parametrised by $k$, each of which is the solution to the equation
\begin{equation*}
\lich{(m)} F_k = G_k ,
\end{equation*}
with the minimum $L^2$-norm,
\item there exists a smooth function $F_{\infty}$ which is the solution to the equation
\begin{equation*}
\lich{\omega} F_{\infty} = G_{\infty} 
\end{equation*}
with the minimum $L^2$-norm.
\end{enumerate}
Then $F_{k}$ converges to $F_{\infty}$ in ${C^{\infty}}$ as $k \to \infty$.
\end{lemma}

\begin{proof}
Consider the equation
\begin{equation*}
\lich{(m)} (F_{\infty} - F_k) = \lich{\omega} F_{\infty} + O(1/k) - G_k = G_{\infty} - G_k +O(1/k)
\end{equation*}
in $C^{\infty} (X, \rl)$. Recalling the $L^2$-orthogonal direct sum decomposition $C^{\infty} (X , \rl) = \ker \lich{\omega} \oplus \textup{im} \lich{\omega}$ (and hence $\textup{im} \lich{\omega} = \ker \lich{\omega}^{\perp}$), we write $(F_{\infty} - F_k)^{\perp}$ for the $\ker \lich{(m)}^{\perp}$-component of $F_{\infty} - F_k$. By the standard elliptic estimate, we have
\begin{equation*}
|| (F_{\infty} - F_k)^{\perp} ||_{L^2_{p+4}} \le C_{1,p}(\omega , \{ \phi_{i,k} \}) || G_{\infty} - G_k + O(1/k)||_{L^2_{p}} \to 0
\end{equation*}

Recalling also $\textup{im} \lich{\omega} = \ker \lich{\omega}^{\perp}$, the hypothesis 4 implies $F_{\infty} \in \mathrm{im} \lich{\omega}$, and hence there exists a function $F' \in C^{\infty} (X , \rl)$ such that $F_{\infty} = \lich{\omega} F'$ with the estimate
\begin{equation*}
||F'||_{L^2_{p}} \le C_{2,p} (\omega  ) ||F_{\infty}||_{L^2_{p-4}}
\end{equation*}
following from the standard elliptic regularity. On the other hand,
\begin{equation*}
F_{\infty} = \lich{\omega} F' = \lich{(m)} F' + \frac{1}{k}D(F'),
\end{equation*}
with some differential operator $D$ of order at most 4 which depends on $\omega$ and $(\phi_{1,k} , \dots , \phi_{m,k})$. This means that $F_{\infty} - D(F')/k \in \ker \lich{(m)}^{\perp}$, and hence
\begin{align*}
|| F_{\infty} - (F_{\infty})^{\perp} ||_{L^2_{p+4}} < ||D(F')||_{L^2_{p+4}}/k &< C_{3,p} (\omega , \{ \phi_{i,k} \} ) ||F'||_{L^2_{p}} /k \\
&< C_{4,p} (\omega , \{ \phi_{i,k} \} ) ||F_{\infty}||_{L^2_{p-4}} /k \to 0  
\end{align*}
as $k \to \infty$, where we used the fact that $\phi_i$'s are the functions that converge to some smooth function as $k \to \infty$, so that $C_4 (\omega , \{ \phi_{i,k} \})$ stays bounded when $k$ goes to infinity. Thus, recalling that $F_k$ is the solution to $\lich{(m)} F_k = G_k$ with the minimum $L^2$-norm (implying $(F_k)^{\perp} = F_k$), we have
\begin{equation*}
|| F_{\infty} -F_k||_{L^2_{p+4}} \le || (F_{\infty} -F_k)^{\perp}||_{L^2_{p+4}} + || F_{\infty} - (F_{\infty})^{\perp}||_{L^2_{p+4}} \to 0
\end{equation*}
as $k \to \infty$.

Since the above argument holds for all large enough $p$, we see that $F_k$ converges to $F_{\infty}$ in $C^{\infty}$.


\end{proof}

\bibliographystyle{amsplain}
\bibliography{2017_22_stability}

\begin{flushleft}
Aix Marseille Universit\'e, CNRS, Centrale Marseille, \\
Institut de Math\'ematiques de Marseille, UMR 7373, \\
13453 Marseille, France. \\
Email: \verb|yoshinori.hashimoto@univ-amu.fr|
\end{flushleft}

\end{document}